\documentclass[leqno]{amsart}
\usepackage[left=1in, right=1in, top=1in, bottom=1in]{geometry}
\usepackage{boilerplate}
\usepackage{mathrsfs}
\usepackage{tikz-cd}
\usepackage{mathdots}
\usetikzlibrary{arrows,arrows.meta}
\usepackage[mathscr]{euscript}
\usepackage{algorithm}
\usepackage{algpseudocode}
\usepackage{hyperref}
\hypersetup{
    colorlinks,
    citecolor=black,
    filecolor=black,
    linkcolor=black,
    urlcolor=black
}
\usepackage{enumitem}
\usepackage{subcaption}

\makeatletter
\renewcommand\subsection{\leftskip 0pt\@startsection{subsection}{2}{\z@}%
                                     {-3.25ex\@plus -1ex \@minus -.2ex}%
                                     {1.5ex \@plus .2ex}%
                                     {\normalfont\normalsize\bfseries}}
\renewcommand\subsubsection{\@startsection{subsubsection}{3}{\z@}%
                                     {-3.25ex\@plus -1ex \@minus -.2ex}%
                                     {1.5ex \@plus .2ex}%
                                     {\normalfont\normalsize\bfseries\leftskip 3ex}}
\makeatother

\setlist[enumerate]{label*=\arabic*.}

\usepackage{dsfont}

\title{Recollements and stratification}

\author{Jay Shah}
\address{Fachbereich Mathematik und Informatik, WWU Münster, 48149 M\"{u}nster, Germany}
\email{jayhshah@gmail.com}

\begin{document}

\tikzcdset{arrow style=tikz, diagrams={>=stealth}}

\begin{abstract} 
We develop various aspects of the theory of recollements of $\infty$-categories, including a symmetric monoidal refinement of the theory. Our main result establishes a formula for the gluing functor of a recollement on the right-lax limit of a locally cocartesian fibration determined by a sieve-cosieve decomposition of the base. As an application, we prove a reconstruction theorem for sheaves in an $\infty$-topos stratified over a finite poset $P$ in the sense of Barwick--Glasman--Haine. Combining our theorem with methods from the work of Ayala--Mazel-Gee--Rozenblyum, we then prove a conjecture of Barwick--Glasman--Haine that asserts an equivalence between the $\infty$-category of $P$-stratified $\infty$-topoi and that of toposic locally cocartesian fibrations over $P^{\op}$.
\end{abstract}

\date{\today}
\maketitle

\tableofcontents

\section{Introduction}

The theory of recollements plays an important and ubiquitous role throughout topology, algebraic geometry, and representation theory. It is a common axiomatization of, on the one hand, the adjunctions
\[ \begin{tikzcd}[column sep=4em]
\Shv(U) \ar[hookrightarrow, shift left=2]{r}{j_!} \ar[hookrightarrow, shift right=4]{r}[swap]{j_*} & \Shv(X) \ar[shift left=1]{l}[description]{j^*} \ar[shift left=2]{r}{i^*} \ar[shift right=1, hookleftarrow]{r}[swap]{i_*} & \Shv(Z)
\end{tikzcd} \]
associated to $\infty$-categories of sheaves of spaces on a topological space $X$ decomposed by an open subspace $j: \into{U}{X}$ and its closed complement $i: \into{Z = X \setminus U}{X}$, and, on the other hand, the adjunctions
\[ \begin{tikzcd}[column sep=4em]
\QCoh_Z(X) \ar[hookrightarrow, shift left=2]{r}{j_!} \ar[hookrightarrow, shift right=4]{r}[swap]{j_*} & \QCoh(X) \ar[shift left=1]{l}[description]{j^*} \ar[shift left=2]{r}{i^*} \ar[shift right=1, hookleftarrow]{r}[swap, description]{i_*} \ar[shift right=4]{r}[swap]{i^!} & \QCoh(U)
\end{tikzcd} \]
associated to stable $\infty$-categories of quasicoherent complexes on a qcqs scheme $X$ with open subscheme $i: \into{U}{X}$, where $\QCoh_Z(X)$ denotes those quasicoherent complexes set-theoretically supported on $Z = X \setminus U$.\footnote{The fully faithful left adjoint $j_!$ is the definitional embedding of $\QCoh_Z(X)$ in $\QCoh(X)$, whereas the fully faithful right adjoint $j_*$ embeds $\QCoh_Z(X)$ as $\QCoh(X^{\wedge}_Z) \subset \QCoh(X)$, the full subcategory of quasicoherent complexes on $X$ complete along $Z$; cf. \cite{BarwickGlasmanNoteRecoll}.}

Recollements were introduced by Beilinson--Bernstein--Deligne \cite{BBD} in the context of derived categories of perverse sheaves and were later defined by Lurie in the $\infty$-categorical context in the course of his study of constructible sheaves on stratified spaces \cite[\S A]{HA}. The goal of this article is to continue the development of the general theory of recollements from \cite[\S A.8]{HA}, which we recapitulate in \S\ref{sec:recoll} beginning with the basic definition \ref{dfn:recollement}. Our first contribution is to establish a symmetric monoidal refinement of this theory:

\begin{dfn}[\cref{dfn:monoidalRecollement}]
Let $\sX$ be a symmetric monoidal $\infty$-category that admits finite limits. Then a recollement
\[ \begin{tikzcd}[column sep=4em]
\sU \ar[hookrightarrow, shift right=1]{r}[swap]{j_*} & \sX \ar[shift right=2]{l}[swap]{j^*} \ar[shift left=2]{r}{i^*} \ar[hookleftarrow, shift right=1]{r}[swap]{i_*} & \sZ
\end{tikzcd} \]
is \emph{symmetric monoidal} if the localization functors $j_* j^*$ and $i_* i^*$ are compatible with the symmetric monoidal structure, so that $\sU$ and $\sZ$ uniquely inherit symmetric monoidal structures from $\sX$ such that the functors $j^*$ and $i^*$ uniquely refine to (strong) symmetric monoidal functors.
\end{dfn}

Recall that Lurie shows that given a recollement $(\sU, \sZ)$ on $\sX$, if we define the \emph{gluing functor} of the recollement to be $\phi = i^* j_*$ then we may reconstruct $\sX$ as the fiber product $\Ar(\sZ) \times_{\ev_1, \sZ, \phi} \sU$, where $\Ar(\sZ) \coloneq \Fun(\Delta^1, \sZ)$ is the $\infty$-category of arrows in $\sZ$.\footnote{To be precise, Lurie doesn't quite formulate his result in this way. See \cref{recollEquivalenceToOplaxLim} and the discussion thereafter.} Now given a lax symmetric monoidal functor $\phi: \sU \to \sZ$ of symmetric monoidal $\infty$-categories, we may construct a certain \emph{canonical} symmetric monoidal structure on $\Ar(\sZ) \times_{\ev_1, \sZ, \phi} \sU$ (\cref{dfn:canonicalSMConOplaxLimit}). We then have:

\begin{thm}[\cref{thm:CanonicalMonoidalStructureOnMonoidalRecollement}]
Let $\sX$ be a symmetric monoidal $\infty$-category decomposed by a symmetric monoidal recollement $(\sU, \sZ)$. Then the natural equivalence $\sX \xto{\simeq} \Ar(\sZ) \times_{\ev_1, \sZ, \phi} \sU$ refines to an equivalence of symmetric monoidal $\infty$-categories. In other words, the lax symmetric monoidal structure on the gluing functor reconstructs the symmetric monoidal structure on $\sX$.
\end{thm}

\begin{rem}
Although this result is a simple exercise in the theory of $\infty$-operads, it appears that our work was the first to give a proof, and indeed a construction of the canonical symmetric monoidal structure. We note that the work of Ayala--Mazel-Gee--Rozenblyum has since placed this sort of construction within the context of endowing right-lax limits with $\sO$-monoidal structure \cite[\S 4.4]{AMGRb}.
\end{rem}

Our next contribution is motivated by the following problem from equivariant stable homotopy theory:

\begin{problem} \label{problem}
Let $G$ be a finite group and $\cF$ a $G$-family (i.e., a set of subgroups of $G$ closed under taking subgroups and conjugation). Given a (genuine) $G$-spectrum $X \in \Sp^G$ that is $\cF$-complete and a subgroup $H \leq G$ not in $\cF$, give a formula for the $H$-geometric fixed points of $X$ in terms of the $K$-geometric fixed points of $X$ ranging over $K \in \cF$.
\end{problem}

Recollement theory is relevant here because any $G$-family $\cF$ defines a recollement on $\Sp^G$ whose open part is spanned by the $\cF$-complete $G$-spectra (cf. \cite{MATHEW2017994} or \cite{QS21a}). In fact, we may further recast this problem using the stratification theory of Ayala--Mazel-Gee--Rozenblyum \cite{AMGR-NaiveApproach,AMGRb}. In their work, they construct a certain locally cocartesian fibration $\Sp^G_{\phi\text{-locus}} \to P$, where $P$ is the poset of conjugacy classes of subgroups of $G$ and the fiber over $[H]$ is $\Fun(B W_G H, \Sp)$ for $W_G H = N_G H/H$ the Weyl group, such that one has a canonical equivalence
\[ \Fun^{\cocart}_{/P}(\sd(P), \Sp^G_{\phi\text{-locus}}) \simeq \Sp^G \]
where $\sd(P)$ is the barycentric subdivision\footnote{Recall that $\sd(P)$ is the poset whose objects are strings $[a_0 < ... < a_n]$ in $P$ and whose morphisms are string inclusions.} of $P$ regarded as a locally cocartesian fibration over $P$ via the functor that takes the maximum, and the lefthand side denotes the full subcategory spanned by those functors $\sd(P) \to \Sp^G_{\phi\text{-locus}}$ over $P$ preserving locally cocartesian edges. The idea is that this equivalence parametrizes a $G$-spectrum in terms of its geometric fixed points, and indeed given a $G$-spectrum $X$, under this equivalence $X$ transports to a functor $\sd(P) \to \Sp^G$ that sends $[H]$ to $\Phi^H X$. Now by definition any $G$-family $\cF$ defines a \emph{sieve} (i.e., a downward closed subposet) in $P$, and the $\cF$-recollement on $\Sp^G$ transports to a recollement on $\Fun^{\cocart}_{/P}(\sd(P), \Sp^G_{\phi\text{-locus}})$ given by the pair
\[ (\Fun^{\cocart}_{/\cF}(\sd(\cF), \Sp^G_{\phi\text{-locus}}|_{\cF}), \Fun^{\cocart}_{/(P \setminus \cF)}(\sd(P \setminus \cF), \Sp^G_{\phi\text{-locus}}|_{P \setminus \cF}). \]
Establishing a pointwise formula for the gluing functor of this recollement would then yield a solution to \cref{problem}. In general, we prove:

\begin{thm}[\cref{thm:subdivisionExtension}] \label{thm:subdivExtendIntro}
Let $P$ be a poset and let $P_0$ be a sieve in $P$. Let $\sd(P)_0 \subset \sd(P)$ be the subposet on those strings that originate in $P_0$, and note that $\max|_{\sd(P)_0}$ remains a locally cocartesian fibration. Then for every locally cocartesian fibration $C \to P$, the restriction functor
\[ \Fun^{\cocart}_{/P}(\sd(P)_0, C) \to \Fun^{\cocart}_{/P_0}(\sd(P_0), C|_{P_0}) \]
is a trivial fibration.
\end{thm}

\begin{thmx}[\cref{thm:ExistenceLaxRightKanExtension}, \cref{prp:closedPartOfRecollement}, and \cref{thm:RecollementRlaxLimitOfLlaxFunctor}] \label{thmA}
Let $P$ be a down-finite poset\footnote{A poset $P$ is \emph{down-finite} if for every $p \in P$, the subposet $P^{\leq p}$ is finite.} and let $p: C \to P$ be a locally cocartesian fibration such that for every $p \in P$, the fiber $C_p$ admits finite limits, and for every $p \leq q$, the associated pushforward functor $C_p \to C_q$ preserves finite limits. Then for every sieve-cosieve decomposition $P_0, P_1 = P \setminus P_0$ of $P$, we obtain a recollement
\[ \begin{tikzcd}[column sep=4em]
\Fun^{\cocart}_{/P_0}(\sd(P_0), C|_{P_0}) \ar[hookrightarrow, shift right=1]{r}[swap]{j_*} & \Fun^{\cocart}_{/P}(\sd(P), C) \ar[shift right=2]{l}[swap]{j^*} \ar[shift left=2]{r}{i^*} \ar[hookleftarrow, shift right=1]{r}[swap]{i_*} & \Fun^{\cocart}_{/P_1}(\sd(P_1), C|_{P_1})
\end{tikzcd} \]
where $j^*, i^*$ are given by restriction and their fully faithful right adjoints $j_*, i_*$ are describable by the following pointwise formulas:
\begin{enumerate}
\item For every $x \in P_1$, let $J_x \subset \sd(P)$ be the subposet on strings $[a_0 < ... < a_n < x]$, $n \geq 0$ with $a_i \in P_0$. Then for every $[f: \sd(P_0) \to C|_{P_0}]$ on the lefthand side, if we let $\overline{f}$ denote the unique extension of $f$ over $\sd(P)_0$ given by \cref{thm:subdivExtendIntro}, then $j_*(f)$  evaluates on $x \in P_1$ to $\lim (\overline{f}|_{J_x}: J_x \to C_x)$.
\item For every $[f: \sd(P_1) \to C|_{P_1}]$ on the righthand side, $i_*(f)$ evaluates on $x \in P_0$ to the final object $\ast \in C_x$.
\end{enumerate}
\end{thmx}

\begin{rem}
In \cite{QS21a}, we use \cref{thmA} to answer \cref{problem} in the form of \cite[Thm.~F]{QS21a}.
\end{rem}

In fact, we prove a more general theorem where we replace $P$ and the sieve $P_0$ by any $\infty$-category $S$ and functor $\pi: S \to \Delta^1$ determining a sieve-cosieve decomposition of $S$, at the possible cost of demanding more conditions on our locally cocartesian fibration $p: C \to S$.

\begin{rem}
Conceptually, a locally cocartesian fibration $C \to P$ is the unstraightening of a \emph{left-lax} diagram $P \to \Cat_{\infty}$, and the $\infty$-category $\Fun^{\cocart}_{/P}(\sd(P), C)$ is then the \emph{right-lax limit} of this left-lax diagram (cf. \cite[\S A]{AMGRb}). \cref{thmA} then amounts to an \emph{existence theorem} for the (pointwise) right-lax Kan extension of $[C \to P]$ along a functor $\pi: P \to \Delta^1$, along with a \emph{transitivity property} of right-lax Kan extensions with respect to the composite $P \to \Delta^1 \to \ast$.
\end{rem}

Although \cref{thmA} may appear innocuous, we can leverage it to great effect in inductive arguments that build up the right-lax limit of a locally cocartesian fibration from its strata. For example, we will use \cref{thmA} to establish the theory of \emph{$1$-generated and extendable objects} in \S \ref{sec:extendable}, which furnishes a proof of an assertion of Nikolaus and Scholze \cite[Rem.~II.4.8]{NS18} on decomposing the $\infty$-category of bounded-below $C_{p^n}$-spectra as an iterated pullback; for a precise statement, see \cref{rem:NikolausScholze}.

In this paper, our main application of \cref{thmA} will be to prove a \emph{reconstruction theorem} for sheaves on an $\infty$-topos stratified over a finite poset $P$ that was conjectured in the work of Barwick--Glasman--Haine \cite[Rem.~8.2.7]{Exodromy}. We recall the definition of a $P$-stratified $\infty$-topos as \cref{dfn:strattopos} and that of a toposic locally cocartesian fibration as \cref{dfn:toposicfibration}. The reader may want to bear in mind the example of a $P$-stratified $\infty$-topos given by $\Shv(X)$ for $X$ a topological space equipped with a continuous map $\pi: X \to P$, where we endow $P$ with the Alexandroff topology (so that its open sets are cosieves).

\begin{thmx}[\cref{thm:ReconstructionTopoi} and \cref{thm:AdjunctionTopos}] \label{thmB}
Let $\sX$ be an $\infty$-topos equipped with a $P$-stratification $\pi_*: \sX \to \Shv(P)$ for a finite poset $P$. Then we may functorially associate to $(\sX, \pi_*)$ a locally cocartesian fibration $\sG(\sX) \to P^{\op}$ such that we have a canonical equivalence
\begin{equation} \label{eqn:comparison_map}
 \Theta_P: \Fun^{\cocart}_{/P^{\op}}(\sd(P^{\op}), \sG(\sX)) \xto{\simeq} \sX.
 \end{equation}
Moreover, $\Theta_P$ is the counit of an adjoint equivalence
\begin{equation} \label{eqn:adjoint_equivalence}
\adjunct{\mathrm{lim}^{\rlax}}{\LocCocart^{\mathrm{top}}_{P^{\op}}}{\StrTop_{\infty,P}}{\sG} 
\end{equation}
between the $\infty$-category of toposic locally cocartesian fibrations over $P^{\op}$ and the $\infty$-category of $P$-stratified $\infty$-topoi.
\end{thmx}

\begin{rem}
We explain how to interpret \cref{thmB} as a reconstruction theorem. Define the \emph{$p$th stratum} $\sX_p$ to be $\Shv(\{ p \}) \times_{\Shv(P), \pi_*} \sX$, where the fiber product is formed in the $\infty$-category $\Top_{\infty}$ of $\infty$-topoi and geometric morphisms thereof. (For example, if $\sX = \Shv(X)$ for a $P$-stratified space $\pi: X \to P$, then $\sX_p \simeq \Shv(X_p)$.) Let
$$\adjunct{\Phi^p}{\sX}{\sX_p}{\rho_p}$$
denote the associated geometric morphism adjunction. Then $\rho_p$ is fully faithful, and we in fact define
$$\sG(\sX) \coloneq \{ (x,p) \in \sX \times P^{\op}: x \in \sX_p \}$$
with respect to $\rho_p: \into{\sX_p}{\sX}$, so that $\sG(\sX)_p \simeq \sX_p$ (\cref{con:GluingDiagram}). Now under the equivalence $\Theta_P$, a sheaf (i.e., object) $x \in \sX$ transports to a functor $f_x: \sd(P^{\op}) \to \sG(\sX)$ whose value on $[p]$ is given by $\Phi^p(x)$ (\cref{rem:transportingSheaf}). The functor $\Theta_p$ then sends $f_x$ to the limit of its projection into $\sX$.
\end{rem}

\begin{rem}
The strategy of our proof of \cref{thmB} is heavily inspired by the work of Ayala--Mazel-Gee--Rozenblyum, who prove a similar statement in the setting of presentable stable $\infty$-categories \cite[Thm.~A]{AMGRb}. Note however that our proof of the equivalence \eqref{eqn:comparison_map} (but not \eqref{eqn:adjoint_equivalence}) is independent of any explicit use of $(\infty,2)$-category theory in the form of the fibrational mate correspondence for locally cocartesian fibrations (that is, \cite[Lems.~A.8.1 and A.8.3]{AMGRb}), which we recall in this paper as \cref{leftlaxTorightlax}. Indeed, we instead use \cref{thmA} as the basis for an inductive argument that establishes \eqref{eqn:comparison_map}. Similarly, one can supply an alternative proof of the comparable part of \cite[Thm.~A]{AMGRb} using the same strategy; as already mentioned, we implement this idea in context of equivariant stable homotopy theory in our proof of \cite[Thm.~F]{QS21a} (cf. the discussion below \cite[Thm.~2.42]{QS21a}).
\end{rem}

\subsection{What's new in this paper}

We briefly comment on the relation of this paper to \cite{QS19}, which we have since split up into this paper, \cite{QS21a}, and \cite{QS21b}. Sections 2, 3, and 4 of this paper are lightly revised versions of the corresponding sections of \cite{QS19}, whereas section 5 on the application to stratified $\infty$-topoi is entirely new. Also, in the intervening time since we wrote \cite{QS19}, Ayala, Mazel-Gee, and Rozenblyum released their work on stratified noncommutative geometry \cite{AMGRb}; this is an expansion of \cite{AMGR-NaiveApproach} and bears greatly on many of the topics treated in this paper. As such, we have added a few remarks throughout (in particular, \cref{rem:amgrCompare} and the new \S \ref{sec:symmMon}) explaining how our work relates to \cite{AMGRb}. One of the main takeaways here is that one can leverage \cref{thmA} to remove the presentability hypotheses in \cite[Thm.~A]{AMGRb}. Finally, our application to the description of bounded-below $C_{p^n}$-spectra as given in \cite{QS19} relied on some work that has now been moved into \cite{QS21b}.

\subsection{Acknowledgements}

I would like to thank J.D. Quigley for an inspiring collaboration that led to the genesis of this work. The author was supported first by NSF grant DMS-1547292 and later by the Deutsche Forschungsgemeinschaft (DFG, German Research Foundation) under Germany’s Excellence Strategy EXC 2044–390685587, Mathematics Münster: Dynamics–Geometry–Structure.

\section{Recollements} \label{sec:recoll}

In this section, we establish the basic theory of recollements, expanding upon \cite[\S A.8]{HA} and \cite{BarwickGlasmanNoteRecoll}. After setting up the definitions and summarizing Lurie's results on recollements, we explain a symmetric monoidal refinement of the theory of recollements, connect the theory of stable symmetric monoidal recollements to that of smashing localizations, and record some useful projection formulas. We conclude by proving a few lemmas concerning families of recollements that we will need in \cite{QS21a,QS21b}.

\begin{dfn}[{\cite[Def.~A.8.1]{HA}}] \label{dfn:recollement} Let $\sX$ be an $\infty$-category that admits finite limits and let $\sU, \sZ \subset \sX$ be full subcategories that are stable under equivalences. Then ($\sU$, $\sZ$) is a \emph{recollement} of $\sX$ if the inclusion functors $j_{\ast}: \sU \subset \sX$ and $i_{\ast}: \sZ \subset \sX$ admit left exact left adjoints $j^{\ast}$ and $i^{\ast}$ such that
\begin{enumerate} \item $j^{\ast} i_{\ast}$ is equivalent to the constant functor at the terminal object $\ast$ of $\sU$.
\item $j^{\ast}$ and $i^{\ast}$ are \emph{jointly conservative}, i.e., if $f: x \to y$ is a morphism in $\sX$ such that $j^{\ast} f$ and $i^{\ast} f$ are equivalences, then $f$ is an equivalence.
\end{enumerate}
We will call $\sU$ the \emph{open} part of the recollement, $\sZ$ the \emph{closed} part of the recollement, and $i^{\ast} j_{\ast}$ the \emph{gluing functor}.\footnote{Our convention on which subcategory is open and which is closed matches that for constructible sheaves, whereas other authors (e.g., \cite{BarwickGlasmanNoteRecoll}) use the opposite convention, which matches that for quasi-coherent sheaves. Also note that in \cite[Def.~A.8.1]{HA}, Lurie calls the open part $C_1$ and the closed part $C_0$.}
\end{dfn}

The main purpose of the theory of recollements is to codify the various ``fracture square'' decompositions that recur throughout algebra and topology. Abstractly, we have:

\begin{prp} \label{recollementFractureSquare} Let $(\sU, \sZ)$ be a recollement of $\sX$ and let $\eta_j: \id \to j_{\ast} j^{\ast}$, $\eta_i: \id \to i_{\ast} i^{\ast}$ denote the unit transformations. Then we have a pullback square of functors
\[ \begin{tikzcd}[row sep=4ex, column sep=6ex, text height=1.5ex, text depth=0.25ex]
\id \ar{r}{\eta_i} \ar{d}[swap]{\eta_j} & i_{\ast} i^{\ast} \ar{d}{i_{\ast} i^{\ast} \eta_j} \\
j_{\ast} j^{\ast} \ar{r}{\eta_i j_{\ast} j^{\ast}} & i_{\ast} i^{\ast} j_{\ast} j^{\ast}.
\end{tikzcd} \]
\end{prp}
\begin{proof}
By joint conservativity of the left-exact functors $j^*$ and $i^*$, it suffices to check that we have a pullback square after applying $j^*$ and $i^*$, which is clear.
\end{proof}

Next, we define morphisms of recollements.

\begin{dfn}
Suppose that $(\sU_1, \sZ_1)$ and $(\sU_2, \sZ_2)$ are recollements on $\sX_1$ and $\sX_2$. Then a functor $F: \sX_1 \to \sX_2$ is a \emph{morphism of recollements} if $F$ sends $j^{\ast}_1$-equivalences to $j^{\ast}_2$-equivalences and $i^{\ast}_1$-equivalences to $i^{\ast}_2$-equivalences. Let $\Recoll$ denote the resulting $\infty$-category of recollements, and let $\Recoll^{\lex}$ be the full subcategory on those morphisms of recollements that are also left-exact.
\end{dfn}

\begin{obs} \label{obs:MorphismOfRecollements} Suppose that $F:\sX_1 \to \sX_2$ is a morphism of recollements $(\sU_1, \sZ_1) \to (\sU_2, \sZ_2)$. Then we may define $F_U = j_2^{\ast} F {j_1}_{\ast}: \sU_1 \to \sU_2$ and $F_Z = i_2^{\ast} F {i_1}_{\ast}$ so that we have a commutative diagram
\[ \begin{tikzcd}[row sep=4ex, column sep=4ex, text height=1.5ex, text depth=0.25ex]
\sU_1 \ar{d}{F_U} & \sX_1 \ar{d}{F} \ar{l}[swap]{j_1^{\ast}} \ar{r}{i_2^{\ast}} & \sZ_1 \ar{d}{F_Z} \\
\sU_2 & \sX_2 \ar{l}[swap]{j_2^{\ast}} \ar{r}{i_2^{\ast}} & \sZ_2,
\end{tikzcd} \]
such that $F$ is left-exact if and only if $F_U$ and $F_Z$ are left-exact. Conversely, if we are given such a commutative diagram, then $F$ is a morphism of recollements.  Indeed, for any morphism $[f: x \rightarrow y] \in \sX_1$ such that $j^* (f)$ resp. $i^*(f)$ is an equivalence, $j^* F(f) \simeq F j^*(f)$ resp. $i^* F(f) \simeq F i^*(f)$ is an equivalence. Moreover, since $F_U \simeq j^* F j_*$ and $F_Z \simeq i^* F i_*$, it follows that functors $\sU_1 \to \sU_2$ and $\sZ_1 \to \sZ_2$ induced by $F$ as a morphism of recollements are then canonically equivalent to $F_U$ and $F_Z$.
\end{obs}

\begin{obs}\label{LaxVsStrictMorphismsOfRecollements}
In the situation of \cref{obs:MorphismOfRecollements}, by adjunction we get natural transformations $\nu: F {j_1}_{\ast} \Rightarrow {j_2}_{\ast} F_U$ and $\nu': F {i_1}_{\ast} \Rightarrow {i_2}_{\ast} F_Z$. Note that if $F$ preserves the terminal object, then $\nu'$ is an equivalence; indeed, for all $z \in \sZ_1$ we then have
\[ {j_2}^{\ast} F {i_1}_{\ast}(z) \simeq F_U {j_1}^{\ast} {i_1}_{\ast}(z) \simeq F_U (\ast) \simeq \ast, \]
so the unit map $F {i_1}_{\ast} (z) \to {i_2}_{\ast} i_2^{\ast} F {i_1}_{\ast} (z) = {i_2}_{\ast} F_Z (z)$ is an equivalence. In particular, if $F$ is left exact, then $\nu'$ is an equivalence \cite[Rmk.~A.8.10]{HA}. On the other hand, $\nu$ is an equivalence if and only if
$$\nu'': F_Z {i_1}^{\ast} {j_1}_{\ast} \Rightarrow i_2^{\ast} {j_2}_{\ast} F_U$$
is an equivalence -- indeed, the `only if' direction is obvious, and for the `if' direction we may readily check that ${j_2}^{\ast} \nu$ and ${i_2}^{\ast} \nu$ are equivalences and then invoke the joint conservativity of ${j_2}^{\ast}$ and ${i_2}^{\ast}$.
\end{obs}

% \begin{obs}
% Suppose $\sX$ and $\sX'$ are $\infty$-categories decomposed by recollements $(\sU, \sZ)$ and $(\sU', \sZ')$, and we have a commutative diagram
% \[ \begin{tikzcd}
% \sU \ar{d}[swap]{F_U} & \sX \ar{d}{F} \ar{l}[swap]{j^*} \ar{r}{i^*}  & \sZ \ar{d}{F_Z} \\ 
% \sU' & \sX' \ar{l}[swap]{j^*} \ar{r}{i^*} & \sZ'.
% \end{tikzcd} \]
% Then $F$ is a morphism of recollements. Indeed, for any morphism $[f: x \rightarrow y]$ such that $j^* (f)$ resp. $i^*(f)$ is an equivalence, $j^* F(f) \simeq F j^*(f)$ resp. $i^* F(f) \simeq F i^*(f)$ is an equivalence.

% Moreover, since $F_U \simeq j^* F j_*$ and $F_Z \simeq i^* F i_*$, it follows that functors $\sU \to \sU'$ and $\sZ \to \sZ'$ induced by $F$ as a morphism of recollements as in \cref{LaxVsStrictMorphismsOfRecollements} are canonically equivalent to $F_U$ and $F_Z$.
% \end{obs}

\begin{dfn} If $\nu''$ in \cref{LaxVsStrictMorphismsOfRecollements} is an equivalence, then we call $F$ a \emph{strict} morphism of recollements. Let $\Recoll_{\str} \subset \Recoll$ and $\Recoll^{\lex}_{\str} \subset \Recoll^{\lex}$ be the wide subcategories on the strict morphisms.
\end{dfn}

\begin{rem} \label{rem:TwoOutOfThreePropertyEquivalencesStrictMorphismRecoll} If $F: \sX_1 \to \sX_2$ is a strict left-exact morphism of recollements, then $F$ is an equivalence if and only if $F_U$ and $F_Z$ are equivalences \cite[Prop.~A.8.14]{HA}.
\end{rem}

\begin{dfn} Let $\pi: \cM \to \Delta^1$ be a functor of $\infty$-categories with fibers $\cM_0 = \sZ$ and $\cM_1 = \sU$. Then $\pi$ is a \emph{left-exact correspondence} \cite[Def.~A.8.6]{HA} if
\begin{enumerate} \item $\pi$ is a cartesian fibration, so determines a functor $\phi: \sU \to \sZ$.
\item $\sU$ and $\sZ$ admit finite limits and $\phi$ is left-exact.
\end{enumerate}
A \emph{morphism of left-exact correspondences} is a functor $F: \cM_1 \to \cM_2$ over $\Delta^1$. In terms of the left-exact functors $\phi_1$ and $\phi_2$, this corresponds to a right-lax commutative diagram
\[ \begin{tikzcd}[row sep=4ex, column sep=4ex, text height=1.5ex, text depth=0.25ex]
\sU_1 \ar{r}{\phi_1} \ar{d}[swap]{F_U} \ar[phantom]{rd}{\SWarrow} & \sZ_1 \ar{d}{F_Z} \\
\sU_2 \ar{r}[swap]{\phi_2} & \sZ_2.
\end{tikzcd} \]
Let $\Ar^{\rlax}_{\lex}(\Cat_{\infty})$ denote the resulting $\infty$-category of left-exact correspondences as a full subcategory of $(\Cat_{\infty})_{/\Delta^1}$, and let $\Ar_{\lex}(\Cat_{\infty})$ be the wide subcategory on those morphisms that preserve cartesian edges, so that the right-lax commutativity is actually strict. Note that under the straightening correspondence, $\Ar_{\lex}(\Cat_{\infty})$ is the full subcategory of $\Ar(\Cat_{\infty})$ on left-exact functors $\phi: \sU \to \sZ$.

If $F_U$ and $F_Z$ are also left-exact, we say that the morphism $F$ of left-exact correspondences is \emph{left-exact}. We may then view (lax) commutative squares as residing inside $\Cat_{\infty}^{\lex}$ itself. Let $\Ar^{\rlax}(\Cat^{\lex}_{\infty}) \subset \Ar^{\rlax}_{\lex}(\Cat_{\infty})$ and $\Ar(\Cat^{\lex}_{\infty}) \subset \Ar_{\lex}(\Cat_{\infty})$ denote the resulting wide subcategories.
\end{dfn}

\begin{obs} \label{recollEquivalenceToOplaxLim} Let $\cM \to \Delta^1$ be a left-exact correspondence and let $\sX = \Fun_{/\Delta^1}(\Delta^1, \cM)$ be its $\infty$-category of sections. Let $\sU \subset \sX$ be the full subcategory on the cartesian sections and let $\sZ \subset \sX$ be the full subcategory on those sections $\sigma$ such that $\sigma(1)$ is a terminal object of $\sU$. Then $(\sU, \sZ)$ is a recollement of $\sX$ \cite[Prop.~A.8.7]{HA}. Moreover, the formation of sections
\[ \goesto{\cM}{\Fun_{/\Delta^1}(\Delta^1, \cM)} \]
carries morphisms of left-exact correspondences to morphisms of recollements, and thereby defines a functor\footnote{We denote this by $\mathrm{lim}^{\rlax}$ in view of the interpretation of the sections of a cartesian fibration as defining the right-lax limit of the corresponding functor.}
\[ \mathrm{lim}^{\rlax} : \Ar^{\rlax}_{\lex}(\Cat_{\infty}) \xto{\simeq} \Recoll, \]
which is an equivalence of $\infty$-categories by \cite[Prop.~A.8.8]{HA} (for full faithfulness) and \cite[Prop.~A.8.11]{HA} (which shows that if $(\sU,\sZ)$ is a recollement of $\sX$, then $\sX$ is equivalent to the right-lax limit of $i^{\ast} j_{\ast}: \sU \to \sZ$). Furthermore, in view of the discussion in \cref{LaxVsStrictMorphismsOfRecollements}, $\mathrm{lim}^{\rlax}$ restricts to equivalences of subcategories
\begin{align*} \Ar_{\lex}(\Cat_{\infty}) \xto{\simeq} \Recoll_{\str}, \; \Ar^{\rlax}(\Cat^{\lex}_{\infty}) \xto{\simeq} \Recoll^{\lex}, \; \Ar(\Cat^{\lex}_{\infty}) \xto{\simeq} \Recoll^{\lex}_{\str}.
\end{align*}
\end{obs}

We next explain how to identify the $\infty$-category of sections of a cartesian fibration classified by the functor $\phi: \sU \to \sZ$ with the pullback $\Ar(\sZ) \times_{\ev_1,\sZ,\phi} \sU$.

\begin{cnstr} Let $\pi: \cM \to \Delta^1$ be a cartesian fibration. By the dual of \cite[Lem.~2.23]{Exp2}, we have a trivial fibration $\Ar^{\cart}(\cM) \to \Ar(\Delta^1) \times_{\ev_1, \Delta^1, \pi} \cM$, which restricts to a trivial fibration $\ev_1: \Fun^{\cart}_{/\Delta^1}(\Delta^1,\cM) \to \cM_1$. Let $\chi$ be a section of $\ev_1$. 

Because $i: \rightnat{\Lambda^2_2} \to \rightnat{\Delta^2}$ is right marked anodyne, with the structure map $\sigma^0: \Delta^2 \to \Delta^1$, $(\sigma^0)^{-1}(0)= \{0, 1\}$ and $(\sigma^0)^{-1}(1)= \{2\}$, we have a trivial fibration
\[ i^\ast: \Fun_{/\Delta^1}(\rightnat{\Delta^2},\rightnat{\cM}) \to \Fun_{/\Delta^1}(\rightnat{\Lambda_2^2},\rightnat{\cM}) \cong \Fun_{/\Delta^1}(\Delta^1,\cM) \times_{\ev_1, X_1, \ev_1} \Fun^{\cart}_{/\Delta^1}(\Delta^1,\cM). \]
Let $\kappa$ be a section of $i^\ast$. The section $\chi$ yields a functor
\[ f = (\id,\chi \circ \ev_1): \Fun_{/\Delta^1}(\Delta^1,\cM) \to \Fun_{/\Delta^1}(\Delta^1,\cM) \times_{X_1} \Fun^{\cart}_{/\Delta^1}(\Delta^1,\cM).\]
Let $g = \kappa \circ f$. Then the various maps fit into the commutative diagram
\[ \begin{tikzcd}[row sep=4ex, column sep=4ex, text height=1.5ex, text depth=0.25ex]
\Fun_{/\Delta^1}(\Delta^1,\cM) \ar{r}{g} \ar{d}{\ev_1} & \Fun_{/\Delta^1}(\rightnat{\Delta^2},\rightnat{\cM}) \ar{r}{\ev_{01}} \ar{d}{\ev_{12}} & \Fun(\Delta^1,\cM_0) \ar{d}{\ev_1} \\
\cM_1 \ar{r}{\chi} &  \Fun^{\cart}_{/\Delta^1}(\Delta^1,\cM) \ar{r}{\ev_0} & \cM_0.
\end{tikzcd} \]
\end{cnstr}

\begin{lem} \label{lm:sectionsPullbackSquare} The natural map $\Fun_{/\Delta^1}(\Delta^1,\cM) \to \Ar(\cM_0) \times_{\ev_1,\cM_0} \cM_1$ is an equivalence, so the outer square is a homotopy pullback square of $\infty$-categories.
\end{lem}
\begin{proof} Because the sections $\chi$ and $\kappa$ are equivalences, the map $g$ is an equivalence. Moreover, because the map $\Lambda^2_1 \to \Delta^2$ is inner anodyne, the rightmost square is a homotopy pullback square. The claim follows.
\end{proof}

%Rephrase
\begin{cor} \label{cor:RecollementAsPullbackSquare} Suppose that $(\sU,\sZ)$ is a recollement of $\sX$ and consider the commutative\footnote{We can obtain a commutative diagram of simplicial sets using standard techniques in quasi-category theory.} diagram
\[ \begin{tikzcd}[row sep=4ex, column sep=6ex, text height=1.5ex, text depth=0.25ex]
\sX \ar{r}{i^{\ast} \eta_j} \ar{d}[swap]{j^{\ast}} & \Ar(\sZ) \ar{d}{\ev_1} \\
\sU \ar{r}{\phi = i^{\ast} j_{\ast}} & \sZ
\end{tikzcd} \]
where $\eta_j: \sX \to \Ar(\sX)$ is the functor that sends $x$ to the unit map $x \to j_{\ast} j^{\ast} x$. Then the induced map
\[ \sX \xto{\simeq} \Ar(\sZ) \times_{\ev_1, \sZ, \phi} \sU \]
is an equivalence of $\infty$-categories.
% with respect to which the functors $j_{\ast}$ and $i_{\ast}$ are given\footnote{This definition in terms of objects extends in an obvious way to the $\infty$-categories themselves.} by
% \begin{align*} j_{\ast}(u) = [u,\phi(u),\id: \phi(u) \to \phi(u)], \quad i_{\ast}(z) = [0,z,0:z \to 0]
% \end{align*}
% and their adjoints $j^{\ast}$ and $i^{\ast}$ are given by $\pr_U$ and $\ev_0 \pr_{\Ar(Z)}$.
\end{cor}
\begin{proof} Combine \cref{lm:sectionsPullbackSquare} with the equivalence $\mathrm{lim}^{\rlax}: \Ar^{\rlax}_{\lex}(\Cat_{\infty}) \xto{\simeq} \Recoll$ of \cref{recollEquivalenceToOplaxLim}.
\end{proof}

\begin{rem} In view of \cref{cor:RecollementAsPullbackSquare}, given a recollement $(\sU,\sZ)$ of $\sX$ with gluing functor $\phi = i^{\ast} j_{\ast}$ we will often write objects $x \in \sX$ as $[u,\alpha:z \rightarrow \phi(u)]$ or $[u,z,\alpha]$.
\end{rem}

Given a left-exact functor $\phi: \sU \to \sZ$, we may also extract the resulting recollement from the \emph{cocartesian} fibration classified by $\phi$, even though it is difficult to encode the right-lax functoriality when working with cocartesian fibrations.

\begin{obs} \label{dualizingOneSimplex} Let $S$ be an $\infty$-category and $C \to S$ a cocartesian fibration. Recall from \cite{BGN} or \cite[Rec.~5.17]{Exp2} that the \emph{dual cartesian fibration} $C^\vee \to S^{\op}$ is defined to have $n$-simplices\footnote{Here, $\TwAr(-)$ is the \emph{twisted arrow $\infty$-category}. We use the directionality convention of \cite{M1} instead of \cite[\S 5.2.1]{HA}, so twisted arrows are contravariant in the source and covariant in the target.}
\[ \begin{tikzcd}[row sep=4ex, column sep=4ex, text height=1.5ex, text depth=0.25ex]
\leftnat{\TwAr((\Delta^n)^{\op})} \ar{r} \ar{d}[swap]{\ev_1}  & \leftnat{C} \ar{d} \\
((\Delta^n)^{\op})^\sharp \ar{r} & S^\sharp,
\end{tikzcd} \]
where we mark the cocartesian edges in $C$ and $\TwAr((\Delta^n)^{\op})$. In fact, because the functor $\TwAr'(-): s\Set^+_{/S} \to s\Set^+_{/S}$ of \cite[Prop.~5.18]{Exp2} preserves colimits, it follows that for all simplicial sets $A$ over $S^{\op}$
\[ \Hom_{/S^{\op}}(A,C^{\vee}) \cong \Hom_{/S}(\TwAr'(A^{\op}),\leftnat{C}). \]
Consequently, we obtain an equivalence
\[ \Fun_{/S^{\op}}(S^{\op},C^{\vee}) \simeq \Fun^{\cocart}_{/S}(\TwAr(S),C). \]
Now note that the barycentric subdivision $\sd(\Delta^1) = [0 \rightarrow 01 \leftarrow 1]$ is isomorphic to the twisted arrow category $\TwAr(\Delta^1)$. Therefore, for a cocartesian fibration $C \to \Delta^1$, we deduce that
\[ \Fun^{\cocart}_{/\Delta^1}(\sd(\Delta^1),C) \simeq \Fun_{/\Delta^1}(\Delta^1,C^{\vee}) \]
and hence by \cref{lm:sectionsPullbackSquare} we can decompose $\Fun^{\cocart}_{/\Delta^1}(\sd(\Delta^1),C)$ as a pullback square $\Ar(\sZ) \times_{\ev_1, \sZ, \phi} \sU$ for a choice of pushforward functor $\phi: \sU \to \sZ$ (where $\sU \simeq C_0$ and $\sZ \simeq C_1$). This observation will be important for us when we discuss recollements on right-lax limits in the next section.
\end{obs}

\subsection{Stable recollements}

\begin{dfn} Let $\sX$ be a stable $\infty$-category and let ($\sU$, $\sZ$) be a recollement of $\sX$. Then this recollement is \emph{stable} if $\sU$ and $\sZ$ are stable subcategories. Let $\Recoll^{\stab}$, resp. $\Recoll^{\stab}_{\str}$ be the full subcategory of $\Recoll^{\lex}$, resp. $\Recoll_{\str}^{\lex}$ whose objects are the stable recollements.
\end{dfn}

\begin{dfn} If $\cM \to \Delta^1$ is a left-exact correspondence, then $\cM$ is \emph{exact} if the functor $\phi: \cM_1 \to \cM_0$ is an exact functor of stable $\infty$-categories. Let $\Ar^{\rlax}(\Cat^{\stab}_{\infty})$, resp. $\Ar(\Cat^{\stab}_{\infty})$ be the full subcategory of $\Ar^{\rlax}(\Cat^{\lex}_{\infty})$, resp. $\Ar(\Cat^{\lex}_{\infty})$ on the exact correspondences.
\end{dfn}

\begin{rem} The functor $\mathrm{lim}^{\rlax}$ of \cref{recollEquivalenceToOplaxLim} restricts to equivalences
\begin{align*} \Ar^{\rlax}(\Cat^{\stab}_{\infty}) \xto{\simeq} \Recoll^{\stab}, \quad \Ar(\Cat^{\stab}_{\infty}) \xto{\simeq} \Recoll^{\stab}_{\str}.
\end{align*}
\end{rem}

% which in terms of the description of $\sX$ as an oplax limit of $i^{\ast} j_{\ast}$ sends $u$ to $[u,0_Z, \can]$
\begin{obs} Let $(\sU,\sZ)$ be a stable recollement of $\sX$. Then $j^{\ast}: \sX \to \sU$ admits a fully faithful left adjoint\footnote{For the existence of $j_!$, we only need that $\sZ$ admits an initial object $\emptyset$ \cite[Cor.~A.8.13]{HA}. Then $j_!$ is defined by the formula $j_!(u) = [u, \emptyset \rightarrow \phi(u)]$.} $j_!$, $i_{\ast}$ admits a right adjoint $i^!$, and we have norm maps $\Nm: j_! \to j_{\ast}$ and $\Nm': i^! \to i^{\ast}$ that fit into fiber sequences
\begin{align*} j_! \to j_{\ast} \to i_{\ast} i^{\ast} j_{\ast}  \quad \text{and} \quad
i^! \to i^{\ast} \to i^{\ast} j_{\ast} j^{\ast} \:,
\end{align*}
where the other maps are induced by the unit transformations for $j^{\ast} \dashv j_{\ast}$ and $i^{\ast} \dashv i_{\ast}$. On objects $x=[u,z,\alpha] \in \sX$, these amount to the fiber sequences
\begin{align*} [u,0,0] \to [u,\phi u, \id] \to [0,\phi u, 0] \quad \text{and} \quad \fib(\alpha) \to z \xto{\alpha} \phi u \: .
\end{align*}
% Given a recollement $(\sU, \sZ)$ of $\sX$, let $\eta[j]: \id \to j_{\ast} j^{\ast}$ and $\eta[i]: \id \to i_{\ast} i^{\ast}$ denote the respective unit transformations.
Considering the various unit and counit transformations and the norm maps, we may extend the pullback square of \cref{recollementFractureSquare} to a commutative diagram 
\[ \begin{tikzcd}[row sep=4ex, column sep=4ex, text height=1.5ex, text depth=0.25ex]
 & i_{\ast} i^! \ar{r}{\simeq} \ar{d} & i_{\ast} i^! \ar{d}{i_{\ast} \Nm'} \\
j_! j^{\ast} \ar{r} \ar{d}{\simeq} & \id \ar{r} \ar{d} & i_{\ast} i^{\ast} \ar{d} \\
j_! j^{\ast} \ar{r}[swap]{\Nm j^{\ast}} & j_{\ast} j^{\ast} \ar{r} & i_{\ast} i^{\ast} j_{\ast} j^{\ast}
\end{tikzcd} \]
in which every row and column is a fiber sequence.
\end{obs}

\begin{obs} \label{stableRecollementComment} In the stable case, the datum of the closed part of a recollement determines the entire recollement. More precisely, if $\sZ \subset \sX$ is a stable reflective and coreflective subcategory of $\sX$ and we define $\sU$ to be the full subcategory on those objects $u \in \sX$ such that $\Map_{\sX}(z,u) \simeq \ast$ for all $z \in \sZ$, then ($\sU$, $\sZ$) is a stable recollement of $\sX$ \cite[Prop.~A.8.20]{HA}, and conversely, if $(\sU,\sZ)$ is a stable recollement of $\sX$ then $j_{\ast}: \sU \subset \sX$ is defined as above from $\sZ$. We may also identify $j_!(\sU)$ as given by those objects $u \in \sX$ such that $\Map_{\sX}(u,z) \simeq \ast$ for all $z \in \sZ$. 

Moreover, $F: \sX_1 \to \sX_2$ is a morphism of stable recollements $(\sU_1, \sZ_1) \to (\sU_2, \sZ_2)$ if and only if $F|_{\sZ_1} \subset \sZ_2$ and $F|_{j_!(\sU_1)} \subset j_!(\sU_2)$ (in particular, we then have ${j_2}_! F_U \simeq F {j_1}_!$). This is because $\sZ$ coincides with the $j^{\ast}$-null objects and $j_!(\sU)$ with the $i^{\ast}$-null objects. Given this, $F$ is then a strict morphism of stable recollements if and only if we also have that $F|_{j_{\ast}(\sU_1)} \subset j_{\ast}(\sU_2)$.
\end{obs}

% \begin{rem} \label{rem:MorphismPreservesTorsionSubcategory} Suppose that $F: \sX_1 \to \sX_2$ is a morphism of stable recollements $(\sU_1, \sZ_1) \to (\sU_2, \sZ_2)$. Then the natural map $\gamma: {j_2}_! F_U \to F {j_1}_!$ is an equivalence; indeed, we have $i^{\ast} \gamma \simeq 0$ and $j^{\ast} \gamma = \id_{F_U}$.
% \end{rem}

\subsection{Symmetric monoidal recollements}

We now extend the theory of recollements to the situation where $\sX$ admits a symmetric monoidal structure $(\sX, \otimes, \mathds{1})$. In what follows, we will call an adjunction $\adjunct{F}{C}{D}{G}$ between symmetric monoidal $\infty$-categories \emph{symmetric monoidal} if $F$ is (strong) symmetric monoidal.

\begin{dfn} \label{dfn:monoidalRecollement} Let $\sX$ be a symmetric monoidal $\infty$-category that admits finite limits. Then a recollement $(\sU,\sZ)$ of $\sX$ is \emph{symmetric monoidal} if the localization functors $j_{\ast} j^{\ast}$ and $i_{\ast} i^{\ast}$ are compatible with the symmetric monoidal structure in the sense of \cite[Def.~2.2.1.6]{HA}, i.e., for every $j^{\ast}$, resp. $i^{\ast}$-equivalence $f: x \to x'$ and any $y \in \sX$, $f \otimes \id: x \otimes y \to x' \otimes y$ is a $j^{\ast}$, resp. $i^{\ast}$-equivalence.

A morphism $F:(\sU, \sZ) \to (\sU',\sZ')$ of recollements on $\sX$ and $\sX'$ is \emph{symmetric monoidal} if the functor $F: \sX \to \sX'$ is symmetric monoidal. Let $\Recoll^{\otimes}$ denote the $\infty$-category of symmetric monoidal recollements and morphisms thereof.
\end{dfn}

\begin{obs} In the situation of \cref{dfn:monoidalRecollement}, by \cite[Prop.~2.2.1.9]{HA} $\sU$ and $\sZ$ obtain symmetric monoidal structures such that the adjunctions $j^{\ast} \dashv j_{\ast}$ and $i^{\ast} \dashv i_{\ast}$ are symmetric monoidal. In particular, the gluing functor $i^{\ast} j_{\ast}$ is lax symmetric monoidal. Furthermore, if $F$ is a morphism of symmetric monoidal recollements, then the induced functors $F_U$ and $F_Z$ of \cref{LaxVsStrictMorphismsOfRecollements} are also symmetric monoidal.
\end{obs}

\begin{rem}
Most of the results of this subsection will extend verbatim to an arbitrary reduced $\infty$-operad $\sO^{\otimes}$. We leave the details to the reader.
\end{rem}

We first show that given a lax symmetric monoidal functor $\phi: \sU \to \sZ$, the recollement $\mathrm{lim}^{\rlax} \phi$ is symmetric monoidal. We first recall the pointwise symmetric monoidal structure on a functor $\infty$-category.

\begin{cnstr} \label{pointwiseMonoidalStructure} Let $p: C^\otimes \to \Fin_{\ast}$ be an $\infty$-operad, and let $K$ be a simplicial set. We have the cotensor $p^K: (C^\otimes)^K \to \Fin_{\ast}$ defined by $$\Hom_{/\Fin_{\ast}}(A,(C^\otimes)^K) \cong \Hom_{/\Fin_{\ast}}(A \times K,C^\otimes).$$ Then $p^K$ is again an $\infty$-operad: this follows from the observation that for any $\mathfrak{O}$-anodyne morphism $A \to B$ of preoperads (with $\mathfrak{O}$ the defining categorical pattern for the model structure on preoperads), $A \times K \to B \times K$ is again $\mathfrak{O}$-anodyne \cite[Prop.~B.1.9]{HA}. Moreover, if $p$ is in addition a cocartesian fibration, then $p^K$ is also a cocartesian fibration. The fiber of $p^K$ over $\angs{n}$ is $\Fun(K,C^{\times n}) \simeq \prod_{i=1}^n\Fun(K,C)$, and for the unique active map $\angs{n} \to \angs{1}$, if $\phi: C^{\times n} \to C$ is a choice of pushforward functor encoded by $p$, then the postcomposition by $\phi$ functor $\phi_{\ast}: \Fun(K,C^{\times n}) \to \Fun(K,C)$ is a choice of pushforward functor encoded by $p^K$. In other words, $p^K$ is the pointwise symmetric monoidal structure on $\Fun(K,C)$.
\end{cnstr}

We will also need the following lemma.

\begin{lem} \label{lm:evaluationCocartesianMonoidal} Let $C^\otimes$ be a symmetric monoidal $\infty$-category. Then the functor
\[ e_L: (C^\otimes)^{K \star L} \to (C^\otimes)^L \]
induced by $L \subset K \star L$ is a cocartesian fibration of $\infty$-operads.
\end{lem}
\begin{proof} Because $e_L$ is induced by the monomorphism $L \subset K \star L$, $e_L$ is a fibration of $\infty$-operads. Using the inert-active factorization system on an $\infty$-operad, it then suffices to prove the following two properties of $e_L$: 
\begin{enumerate} \item For every object $\angs{n} \in \Fin_{\ast}$, $(e_L)_{\angs{n}}$ is a cocartesian fibration;
\item For every active edge $\alpha: \angs{n} \to \angs{1}$ and commutative square in $(C^\otimes)^{K \star L}$
\[ \begin{tikzcd}[row sep=4ex, column sep=4ex, text height=1.5ex, text depth=0.25ex]
f = (f_1,...,f_n) \ar{r} \ar{d}{\theta} & f' = \otimes_{i=1}^n f_i \ar{d}{\theta'} \\
g = (g_1,...,g_n) \ar{r} & g' = \otimes_{i=1}^n g_i
\end{tikzcd} \]

with the horizontal edges as $p^{K \star L}$-cocartesian edges covering $\alpha$, if $\theta$ is $(e_L)_{\angs{n}}$-cocartesian then $\theta'$ is $(e_L)_{\angs{1}}$-cocartesian.
\end{enumerate} 
For (1), by \cite[Lem.~4.8]{Exp2} we have that $(e_L)_{\angs{n}}: \Fun(K \star L,C^{\times n}) \to \Fun(L, C^{\times n})$ is a cocartesian fibration. Moreover, $\theta: f \to g$ is a $(e_L)_{\angs{n}}$-cocartesian edge if and only if its image in $\Fun(K,C^{\times n})$ is an equivalence. This proves (2), since the $n$-fold tensor product of equivalences is always an equivalence.
\end{proof}

We are now ready to define the symmetric monoidal structure on a right-lax limit.

\begin{dfn} \label{dfn:canonicalSMConOplaxLimit} Suppose $\phi^\otimes: \sU^\otimes \to \sZ^\otimes$ is a lax symmetric monoidal functor of symmetric monoidal $\infty$-categories (i.e., a map of $\infty$-operads). Consider the pullback square of $\infty$-operads
\[ \begin{tikzcd}[row sep=4ex, column sep=4ex, text height=1.5ex, text depth=0.25ex]
(\sZ^\otimes)^{\Delta^1} \times_{\sZ^\otimes} \sU^\otimes \ar{r} \ar{d} & (\sZ^\otimes)^{\Delta^1} \ar{d}{\ev_1} \\ 
\sU^\otimes \ar{r}{\phi^\otimes} & \sZ^\otimes. 
\end{tikzcd} \]
By \cref{lm:evaluationCocartesianMonoidal}, $\ev_1$ is a cocartesian fibration, so $(\sZ^\otimes)^{\Delta^1} \times_{\sZ^\otimes} \sU^\otimes \to \sU^\otimes \to \Fin_{\ast}$ is a cocartesian fibration and therefore a symmetric monoidal $\infty$-category. This defines the \emph{canonical} symmetric monoidal structure on the right-lax limit of $\phi$.
\end{dfn}

% Note that the two evaluation maps $\Ar(Z) \times_Z U \to U$ and $\Ar(Z) \times_Z U \to \Ar(Z) \xto{\ev_0} Z$ are symmetric monoidal.
\begin{rem} In \cref{dfn:canonicalSMConOplaxLimit}, at the level of objects the tensor product on $\Ar(\sZ) \times_{\sZ} \sU$ is defined in the following way: suppose given two objects $x=[u,z,\alpha]$ and $x' = [u',z',\alpha']$. Then $x \otimes x' = [u \otimes u', z \otimes z', \gamma]$, where $\gamma$ is given by the composite map
\[ z \otimes z' \xtolong{\alpha \otimes \alpha'}{1.2} \phi(u) \otimes \phi(u') \to \phi(u \otimes u') \]
using the lax symmetric monoidality of $\phi$ for the second map. 
\end{rem}

\begin{prp} If $\phi: \sU \to \sZ$ is a lax symmetric monoidal left-exact functor, then $\mathrm{lim}^{\rlax} \phi$ is a symmetric monoidal recollement with respect to the canonical symmetric monoidal structure on $\Ar(\sZ) \times_{\sZ} \sU$.
\end{prp}
\begin{proof} We only need to observe that in \cref{dfn:canonicalSMConOplaxLimit}, the two evaluation maps $j^{\ast}: \Ar(\sZ) \times_{\sZ} \sU \to \sU$ and $i^{\ast}: \Ar(\sZ) \times_{\sZ} \sU \to \Ar(\sZ) \xto{\ev_0} \sZ$ are symmetric monoidal.
\end{proof}

We next wish to show that given a symmetric monoidal recollement $(\sU, \sZ)$ of $\sX$, the symmetric monoidal structure on $\sX$ is the canonical one of \cref{dfn:canonicalSMConOplaxLimit}. We first observe that the unit transformation of a symmetric monoidal adjunction is itself a lax symmetric monoidal functor.

\begin{lem} \label{lem:unitLaxMonoidal} Let $C^\otimes$ and $D^\otimes$ be symmetric monoidal $\infty$-categories and let $\adjunct{F}{C}{D}{G}$ be a symmetric monoidal adjunction. Then the unit transformation $\eta: C \to \Ar(C)$ lifts to a lax symmetric monoidal functor $\eta^{\otimes}: C^{\otimes} \to (C^{\otimes})^{\Delta^1}$ such that $\ev_1 \eta^{\otimes} \simeq G^{\otimes} F^{\otimes}$ and $\ev_0 \eta^{\otimes} \simeq \id$.
 \end{lem}
\begin{proof} Let $\cM \to \Delta^1$ be the bicartesian fibration classified by the adjunction. We may factor (or define) $\eta$ as the composition
\[ C \simeq \Fun^{\cocart}_{/ \Delta^1}(\Delta^1,\cM) \subset \Fun_{/\Delta^1}(\Delta^1,\cM) \simeq \Ar(C) \times_{C} D  \to \Ar(C) \]
where we use \cref{lm:sectionsPullbackSquare} for the identification of the sections of $\cM$. Let $\Fun_{/\Delta^1}(\Delta^1,\cM)$ be equipped with its canonical symmetric monoidal structure. Because $F$ is symmetric monoidal, the inclusion $\Fun^{\cocart}_{/ \Delta^1}(\Delta^1,\cM) \subset \Fun_{/\Delta^1}(\Delta^1,\cM)$ defines a symmetric monoidal structure on $\Fun^{\cocart}_{/ \Delta^1}(\Delta^1,\cM)$ by restriction such that the equivalence $\ev_0: \Fun^{\cocart}_{/ \Delta^1}(\Delta^1,\cM) \xto{\simeq} C$ is an equivalence of symmetric monoidal $\infty$-categories. Also, the projection $\Fun_{/\Delta^1}(\Delta^1,\cM) \to \Ar(C)$ is lax symmetric monoidal by definition. We deduce that $\eta$ lifts to a lax symmetric monoidal functor $\eta^{\otimes}$ with the indicated properties.
\end{proof}

\begin{prp} Let ($\sU$, $\sZ$) be a symmetric monoidal recollement of $\sX$. Then the functor $\sX \to \Fun(\Delta^1 \times \Delta^1,\sX)$ realizing the pullback square of functors
\[ \begin{tikzcd}[row sep=4ex, column sep=4ex, text height=1.5ex, text depth=0.25ex]
\id \ar{r} \ar{d} & i_\ast i^\ast \ar{d} \\
j_\ast j^\ast \ar{r} & i_\ast i^\ast j_\ast j^\ast
\end{tikzcd} \]
lifts to a lax symmetric monoidal functor $\sX^{\otimes} \to (\sX^{\otimes})^{\Delta^1 \times \Delta^1}$. Consequently, if $A \in \sX$ is an algebra object, then we have an equivalence of algebras
\[ A \simeq (j_\ast j^\ast)(A) \times_{(i_\ast i^\ast j_\ast j^\ast)(A)} (i_\ast i^\ast)(A). \]
\end{prp}
\begin{proof} By \cref{lem:unitLaxMonoidal}, the symmetric monoidal adjunction $j^{\ast} \dashv j_{\ast}$ yields a lax symmetric monoidal functor $$(\eta_j)^{\otimes}: \sX^{\otimes} \to (\sX^{\otimes})^{\Delta^1}.$$ We also have the induced symmetric monoidal adjunction $\adjunct{\widehat{i}^\ast}{\Ar(\sX)}{\Ar(\sZ)}{\widehat{i}_{\ast}}$ which yields a lax symmetric monoidal functor $$(\eta_{\widehat{i}})^{\otimes}: (\sX^{\otimes})^{\Delta^1} \to (\sX^{\otimes})^{\Delta^1 \times \Delta^1}.$$ The composite $(\eta_{\widehat{i}})^{\otimes} \circ (\eta_j)^{\otimes}$ then defines the desired functor.
\end{proof}

\begin{thm}\label{thm:CanonicalMonoidalStructureOnMonoidalRecollement} Suppose $(\sU, \sZ)$ is a symmetric monoidal recollement of $\sX$. Then the equivalence $$\sX \xto{\simeq} \Ar(\sZ) \times_{\sZ} \sU$$ of \cref{cor:RecollementAsPullbackSquare} refines to an equivalence of symmetric monoidal $\infty$-categories, where we equip $\Ar(\sZ) \times_{\sZ} \sU$ with the canonical symmetric monoidal structure of \cref{dfn:canonicalSMConOplaxLimit}.
\end{thm}
\begin{proof} By \cref{lem:unitLaxMonoidal} and \cref{lem:strictificationOperadSquare}, we have a commutative diagram of $\infty$-operads
\[ \begin{tikzcd}[row sep=4ex, column sep=8ex, text height=1.5ex, text depth=0.5ex]
\sX^{\otimes} \ar{r}{(i^{\ast})^{\otimes} (\eta_j)^{\otimes}} \ar{d}[swap]{(j^{\ast})^{\otimes}} & (\sZ^{\otimes})^{\Delta^1} \ar{d}{\ev_1} \\
\sU^{\otimes} \ar{r}{(i^{\ast})^{\otimes} (j_{\ast})^{\otimes} } & Z^{\otimes}
\end{tikzcd} \]
such that the induced functor $\theta^{\otimes}: \sX^{\otimes} \to (\sZ^{\otimes})^{\Delta^1} \times_{\sZ^{\otimes}} \sU^{\otimes}$ covers the map $\theta$ of \cref{cor:RecollementAsPullbackSquare}. Since $\theta$ is an equivalence, to show that $\theta^{\otimes}$ is an equivalence it suffices to check that $\theta^{\otimes}$ is strongly symmetric monoidal. But this follows from the symmetric monoidality of the jointly conservative functors $j^*, i^*$.
\end{proof}

We include the following simple strictification result for completeness.

\begin{lem} \label{lem:strictificationOperadSquare} Suppose we have a homotopy commutative square of $\infty$-operads
\[ \begin{tikzcd}[row sep=4ex, column sep=4ex, text height=1.5ex, text depth=0.25ex]
A^\otimes \ar{r}{F'} \ar{d}{G'} & B^\otimes \ar{d}{G} \\
C^\otimes \ar{r}{F} & D^\otimes
\end{tikzcd} \]
in the sense that there is the data of a homotopy $\theta: G \circ F' \overset{\simeq}{\Rightarrow} F \circ G'$ over $\Fin_{\ast}$
\[ \begin{tikzcd}[row sep=4ex, column sep=4ex, text height=1.5ex, text depth=0.25ex]
A^\otimes \times \{0\} \ar{r}{F'} \ar{d} & B^\otimes \ar{d}{G} \\
A^\otimes \times \Delta^1 \ar{r}{\theta} & D^\otimes \\
A^\otimes \times \{1\} \ar{r}{G'} \ar{u} & C^\otimes \ar{u}[swap]{F}
\end{tikzcd} \]
such that $\theta$ sends every edge $(a,0) \to (a,1)$ to an equivalence. Suppose also that $G$ is a fibration of $\infty$-operads, i.e., a categorical fibration \cite[2.1.2.10]{HA}. Then there exists a functor $F'': A^\otimes \to B^\otimes$ homotopic to $F'$ as a map of $\infty$-operads such that the square
\[ \begin{tikzcd}[row sep=4ex, column sep=4ex, text height=1.5ex, text depth=0.25ex]
A^\otimes \ar{r}{F''} \ar{d}{G'} & B^\otimes \ar{d}{G} \\
C^\otimes \ar{r}{F} & D^\otimes
\end{tikzcd} \]
strictly commutes.
\end{lem}
\begin{proof} Given an $\infty$-operad $O^\otimes$, let $O^{\otimes,\natural}$ denote the marked simplicial set $(O^\otimes,\cE)$ where $\cE$ is the collection of inert morphisms in $O^\otimes$ \cite[2.1.4.5]{HA}. Consider the lifting problem in marked simplicial sets
\[ \begin{tikzcd}[row sep=4ex, column sep=4ex, text height=1.5ex, text depth=0.25ex]
A^{\otimes,\natural} \times \{0\} \ar{r}{F'} \ar{d} & B^{\otimes,\natural} \ar{d}{G} \\
A^{\otimes,\natural} \times (\Delta^1)^{\sharp} \ar{r}{\theta} \ar[dotted]{ru}{\overline{\theta}} & D^{\otimes,\natural}.
\end{tikzcd} \]
Because $G$ is assumed to be a fibration of $\infty$-operads, $G$ is a fibration in the model structure on $\infty$-preoperads \cite[2.1.4.6]{HA}. Therefore, the dotted lift $\overline{\theta}$ exists. If we then let $F'' = \overline{\theta}|_{A^\otimes \times \{1\}}$, the claim follows.
\end{proof}

We next turn to morphisms of symmetric monoidal recollements.

\begin{obs} Suppose we have a commutative diagram of symmetric monoidal $\infty$-categories and lax symmetric monoidal functors
\[ \begin{tikzcd}[row sep=4ex, column sep=4ex, text height=1.5ex, text depth=0.25ex]
{\sU}^{\otimes} \ar{r}{{\phi}^{\otimes}} \ar{d}[swap]{{F_U}^{\otimes}} & {\sZ}^{\otimes} \ar{d}{{F_Z}^{\otimes}} \\
{\sU'}^{\otimes} \ar{r}{{\phi'}^{\otimes}} & {\sZ'}^{\otimes}.
\end{tikzcd} \]
Then by way of the commutative diagram
\[ \begin{tikzcd}[row sep=4ex, column sep=4ex, text height=1.5ex, text depth=0.25ex]
(\sZ^\otimes)^{\Delta^1} \times_{\sZ^{\otimes}} \sU^{\otimes} \ar{r} \ar{d} & (\sZ^\otimes)^{\Delta^1} \ar{r}{F_Z^{\otimes}} \ar{d}{\ev_1} & (\sZ'^{\otimes})^{\Delta^1} \ar{d}{\ev_1} \\
\sU^{\otimes} \ar{r}{\phi^{\otimes}} \ar{rd}[swap]{F_U^{\otimes}} & \sZ^{\otimes} \ar{r}{F_Z^{\otimes}} & \sZ'^{\otimes} \\
& \sU'^{\otimes} \ar{ru}[swap]{\phi'^{\otimes}}
\end{tikzcd} \]
we obtain a lax symmetric monoidal functor $F^{\otimes}: (\sZ^\otimes)^{\Delta^1} \times_{\sZ^{\otimes}} \sU^{\otimes} \to (\sZ'^\otimes)^{\Delta^1} \times_{\sZ'^{\otimes}} \sU'^{\otimes}$, which is symmetric monoidal if $F_U^{\otimes}$ and $F_Z^{\otimes}$ are symmetric monoidal.

Let $\Ar_{\lex}(\Cat^{\otimes, \lax}_{\infty}) \subset \Ar(\Cat^{\otimes,\lax}_{\infty})$ be the subcategory whose objects are left-exact lax symmetric monoidal functors and whose morphisms are through symmetric monoidal functors. Then by the above construction\footnote{Technically, to make a rigorous construction we may work at the level of preoperads and then pass to the underlying $\infty$-categories.} we may lift the functor $\mathrm{lim}^{\rlax}: \Ar_{\lex}(\Cat_{\infty}) \to \Recoll_{\str}$ to
\[ ({\mathrm{lim}^{\rlax}})^{\otimes}: \Ar_{\lex}(\Cat^{\otimes, \lax}_{\infty}) \to \Recoll_{\str}^{\otimes}. \]
An elaboration of \cref{thm:CanonicalMonoidalStructureOnMonoidalRecollement} shows that $({\mathrm{lim}^{\rlax}})^{\otimes}$ is an equivalence -- we leave the details to the reader.

One also has a lift of $\mathrm{lim}^{\rlax}: \Ar^{\rlax}_{\lex}(\Cat_{\infty}) \to \Recoll$ if one considers right-lax commutative squares of $\infty$-operads. Since the details in this case are more involved, we leave a precise formulation to the reader.
\end{obs}

Our next goal is to establish certain \emph{projection formulas} satisfied by a (stable) symmetric monoidal recollement. First, we note the following about the situation in which the symmetric monoidal $\infty$-category $\sX$ is in addition closed.

\begin{obs} Let $\sX$ be a closed symmetric monoidal $\infty$-category and let $F(-,-)$ denote its internal hom. If $(\sU,\sZ)$ is a symmetric monoidal recollement of $\sX$, then we define
$$F_{\sU}(u,u') = j^{\ast} F(j_{\ast} u, j_{\ast} u') \: \text{ and } \: F_{\sZ}(z,z') = i^{\ast} F(i_{\ast} z, i_{\ast} z')$$
to be internal homs for $\sU$ and $\sZ$, so that $\sU$ and $\sZ$ are closed symmetric monoidal. Indeed, since $j^{\ast} \dashv j_{\ast}$ is monoidal, we have
\begin{align*} \Map_{\sU}(w,j^{\ast} F(j_{\ast} u, j_{\ast} v)) &\simeq \Map_{\sX}(j_{\ast} w,  F(j_{\ast} u, j_{\ast} v)) \simeq \Map_{\sX}(j_{\ast} w \otimes j_{\ast}v, j_{\ast}v) \\
& \Map_{\sU}(j^{\ast}(j_{\ast} w \otimes j_{\ast} u), v) \simeq \Map_{\sU}(w \otimes u,v),
\end{align*}
and similarly for $F_{\sZ}(-,-)$. Moreover we have natural equivalences
\begin{align*} F(x,j_{\ast} u) \simeq j_{\ast} F_{\sU}(j^{\ast} x,u), \quad F(x,i_{\ast} z) \simeq i_{\ast} F_{\sZ}(i^{\ast} x,z).
\end{align*}
For example, we may check
\begin{align*}
\Map_{\sX}(x,F(y,j_{\ast}u)) & \simeq \Map_{\sX}(x \otimes y, j_{\ast} u) \simeq \Map_{\sU}(j^{\ast} x \otimes j^{\ast} y,u) \\
& \simeq \Map_{\sU}(j^{\ast} x, F_{\sU}(j^{\ast} y ,u)) \simeq \Map_{\sX}(x,j_{\ast} F_{\sU}(j^{\ast} y ,u)).
\end{align*}
This implies that the unit maps
\begin{align*} F(j_{\ast} u , j_{\ast} u') & \to j_{\ast} j^{\ast} F(j_{\ast} u, j_{\ast} u') = j_{\ast} F_{\sU}(u,u') \\
F(i_{\ast} z, i_{\ast} z') & \to i_{\ast} i^{\ast} F(i_{\ast} z, i_{\ast} z') = i_{\ast} F_{\sZ}(z,z')
\end{align*}
are equivalences.
\end{obs}

% $i^{\ast}(i_{\ast} z \otimes x) \simeq i^{\ast} i_{\ast} z \otimes i^{\ast} x \to z \otimes i^{\ast} x$

\begin{prp}[Projection formulas] \label{prp:ProjectionFormulasMonoidalRecollement} Let $(\sU, \sZ)$ be a stable\footnote{We do not require stability for the $i^{\ast} \dashv i_{\ast}$ projection formula. For the assertions that only involve $j_!$, we only need that $\sX$ be pointed.} symmetric monoidal recollement of $\sX$.
% Let $x$, $z$, and $u$ generically denote objects of $\sX$, $\sZ$, and $\sU$, respectively.
\begin{enumerate} 
\item The natural maps $\alpha: i_{\ast}(z) \otimes x \to i_{\ast}(z \otimes i^\ast x)$ and $\beta: j_!(u \otimes j^\ast x) \to j_!(u) \otimes x$ are equivalences.
\item The fiber sequence $j_! j^{\ast} x \to x \to i_{\ast} i^{\ast} x$ is equivalent to
\[ j_!(1_U) \otimes x \to x \to i_{\ast}(1_Z) \otimes x. \]
% \item The natural map  is an equivalence, and thus $j_! j^\ast (x) \simeq j_! (1_U) \otimes x$.
% \item is an equivalence, and thus $i_{\ast} i^{\ast}(x) \simeq i_{\ast}(1_Z) \otimes x$.
\end{enumerate}
Now suppose also that $\sX$ is closed symmetric monoidal.
\begin{enumerate}
    \setcounter{enumi}{2}
    \item We have natural equivalences $F(j_! u, x) \simeq j_{\ast} F_{\sU}(u,j^{\ast} x)$ and $F(i_{\ast} z,x) \simeq i_{\ast} F_{\sZ}(z,i^! x)$.
    \item The fiber sequence $i_{\ast} i^! x \to x \to j_{\ast} j^{\ast} x$ is equivalent to
    \[ F(i_{\ast} 1_Z,x) \to x \to F(j_! 1_U,x). \]
    % thus $j_{\ast} j^{\ast} x \simeq F(j_! 1_U,x)$.
    \item  We have natural equivalences $j^{\ast} F(x,y) \simeq F_{\sU}(j^{\ast} x, j^{\ast} y)$ and $F_{\sZ}(i^{\ast} x, i^! y) \simeq i^! F(x,y)$.    
    % \item The natural map $\gamma: j_{\ast} j^{\ast}(x) \to F(j_! 1_U,x)$ adjoint to the counit $j_! j^{\ast} x \simeq j_{\ast} j^{\ast} x \otimes j_! 1_U \to x$ is an equivalence.
    % Therefore, the unit map $F(x,y) \to j_{\ast} j^{\ast} F(x,y)$ is an equivalence if $x = j_! u$ or $y = j_\ast v$.
    % \item $F(j_{\ast} u, j_{\ast} v)$ lies in the essential image of $j_{\ast}$, hence $F(j_{\ast} u, j_{\ast} v) \simeq j_{\ast} F(u,v)$.
%     \item Application of the fiber sequence $i^! \to i^{\ast} \to i^{\ast} j_{\ast} j^{\ast}$ yields a commutative diagram
% \[ \begin{tikzcd}[row sep=4ex, column sep=4ex, text height=1.5ex, text depth=0.25ex]
% i^! F(x,y) \ar{r} \ar{d}{\simeq} & i^{\ast} F(x,y) \ar{r} \ar{d} & i^{\ast} j_{\ast} F_{\sU}(j^{\ast} x, j^{\ast} y) \ar{d} \\
% F(i^{\ast} x, i^! y) \ar{r} & F(i^{\ast} x, i^{\ast} y) \ar{r} & F(i^{\ast} x, i^{\ast} j_{\ast} j^{\ast} y)
% \end{tikzcd} \]
% that exhibits $i^{\ast} F(x,y)$ as a pullback of the righthand square.
\end{enumerate}
\end{prp}
\begin{proof} For (1), it's easily checked that $i^{\ast} \alpha$, $j^{\ast} \alpha$ and $i^{\ast} \beta$, $j^{\ast} \beta$ are equivalences, hence $\alpha$ and $\beta$ are equivalences. (2) then follows as a corollary. For (3), we have sequences of equivalences
\begin{align*} \Map_{\sX}(y,F(j_!u,x)) & \simeq \Map_{\sX}(y \otimes j_! u, x) \simeq \Map_{\sX}(j_!(j^{\ast} y \otimes u),x) \simeq \Map_{\sU}(j^{\ast} y \otimes u,j^{\ast} x) \\
& \simeq \Map_{\sU}(j^{\ast} y , F_{\sU}(u,j^{\ast} x)) \simeq \Map_{\sX}(y,j_{\ast} F_{\sU}(u,j^{\ast} x)), \: \text{and} \\
\Map_{\sX}(y, F(i_{\ast} z, x)) & \simeq \Map_{\sX}(y \otimes i_{\ast} z, x) \simeq \Map_{\sX}(i_{\ast}(i^{\ast} y \otimes z), x) \simeq \Map_{\sZ}(i^{\ast} y \otimes z, i^! x) \\
& \simeq \Map_{\sZ}(i^{\ast} y, F_{\sZ}(z,i^! x)) \simeq \Map_{\sZ}(y, i_{\ast} F_{\sZ}(z,i^! x)).
\end{align*}
If we let $u = 1_U$, then $F_{\sU}(1_U, v) \simeq v$, hence $F(j_! 1_U, x) \simeq j_{\ast} F_{\sU}(1_U,j^{\ast} x) \simeq j_{\ast} j^{\ast} x$. (4) then follows as a corollary. For (5), we have sequences of equivalences
\begin{align*} \Map_{\sU}(u,j^{\ast} F(x,y)) & \simeq \Map_{\sX}(j_! u, F(x,y)) \simeq \Map_{\sX}(j_! u \otimes x,y) \simeq \Map_{\sX}(j_!(u \otimes j^{\ast} x),y) \\
& \simeq \Map_{\sU}(u \otimes j^{\ast} x, j^{\ast} y) \simeq \Map_{\sU}(u,F_{\sU}(j^{\ast} x, j^{\ast} y)), \: \text{and} \\
\Map_{\sZ}(z,F_{\sZ}(i^{\ast} x, i^! y)) & \simeq \Map_{\sZ}(z \otimes i^{\ast} x, i^! y) \simeq \Map_{\sX}(i_{\ast}(z \otimes i^{\ast} x),y) \simeq \Map_{\sX}(i_{\ast} z \otimes x, y) \\
 & \simeq \Map_{\sX}(i_{\ast} z, F(x,y)) \simeq \Map_{\sZ}(z,i^! F(x,y)).
\end{align*}
\end{proof}

\begin{cor} \label{cor:FractureSquareMonoidal} Suppose that $(\sU, \sZ)$ is a stable symmetric monoidal recollement of a closed symmetric monoidal stable $\infty$-category $\sX$. Then for all $x \in \sX$, we have a commutative diagram
\[ \begin{tikzcd}[row sep=4ex, column sep=4ex, text height=1.5ex, text depth=0.25ex]
x \otimes j_! (1_U) \ar{r} \ar{d}{\simeq} & x \ar{r} \ar{d} & x \otimes i_{\ast}(1_Z) \ar{d} \\
F(j_!(1_U), x) \otimes j_!(1_U) \ar{r} & F(j_!(1_U), x) \ar{r} & F(j_!(1_U), x) \otimes i_{\ast}(1_Z)
\end{tikzcd} \]
in which the righthand square is a pullback square.
\end{cor}

Finally, we record the following relation between stable symmetric monoidal recollements and smashing localizations.

% [\textbf{Relation to smashing localizations}]
\begin{obs} \label{SmashingLocalizationsAreStableMonoidalRecollements} Suppose $\sX$ is a symmetric monoidal stable $\infty$-category and $\sZ \subset \sX$ is a reflective and coreflective subcategory that determines a stable recollement $(\sU, \sZ)$ on $\sX$. Then this recollement is symmetric monoidal if and only if $i_{\ast} i^{\ast}$ is compatible with the symmetric monoidal structure on $\sX$ and the resulting projection formula for $i^{\ast} \dashv i_{\ast}$ holds, i.e., the natural map $i_{\ast} z \otimes x \to i_{\ast} ( z \otimes i^{\ast} x)$ is an equivalence for all $x \in \sX$ and $z \in \sZ$. Indeed, the `only if' direction hold by \cref{prp:ProjectionFormulasMonoidalRecollement}, and for the `if' direction, we only need to show that for every $x \in \sX$ such that $j^{\ast} x \simeq 0$, $j^{\ast} (x \otimes y) \simeq 0$ for every $y \in \sX$. But $j^{\ast} x \simeq 0$ if and only if $x \simeq i_{\ast} i^{\ast} x$, and then
\[ j^{\ast}( x \otimes y) \simeq j^{\ast}( i_{\ast} i^{\ast} x \otimes y) \simeq j^{\ast} (i_{\ast} (i^{\ast} x \otimes i^{\ast} y)) \simeq 0. \]

Suppose further that $\sX$ and $\sZ$ are presentable. In view of \cite[Prop.~5.29]{MATHEW2017994}, $\sZ$ is a \emph{smashing localization} of $\sX$ in the sense that $\sZ \simeq \Mod_{\sX}(A)$ for $A = i_{\ast}i^{\ast}1$ an idempotent $E_{\infty}$-algebra in $\sX$. We deduce that smashing localizations of $\sX$ are in bijective correspondence with stable symmetric monoidal recollements of $\sX$. Moreover, if $F: \sX \to \sX'$ is a morphism of symmetric monoidal recollements $(\sU,\sZ) \to (\sU', \sZ')$, then
\[ F i_{\ast} i^{\ast} 1 \simeq i'_{\ast} i'^{\ast} F(1) \simeq i'_{\ast} i'^{\ast} 1, \]
so $F$ preserves the defining idempotent $E_{\infty}$-algebras.
 % if and only if $F(A) \simeq A'$ as $E_{\infty}$-algebras for $A = i_{\ast} i^{\ast} 1$ and $A' = i'_{\ast} i'^{\ast} 1$. 
\end{obs}

%%% Example - family recollement from idempotent algebra

% T with a terminal object such that its finite coproduct completion $\FF_T$ admits pullbacks (e.g., $T = \Ar_G$)

\subsection{Families of recollements}

We conclude this section with a few extensions of recollement theory to the parametrized setting. Let $S$ be an $\infty$-category, let $\sX_{\bullet}: S \to \Recoll_{\str}^{\lex}$ be a functor, and let $\sX, \sU, \sZ \to S$ be the cocartesian fibrations obtained via the Grothendieck construction. Then in view of \cref{LaxVsStrictMorphismsOfRecollements} and the strictness assumption, we have $S$-adjunctions \cite[Def.~8.3]{Exp2}\footnote{Recall given two cocartesian fibrations $C, D \to S$ that a relative adjunction $\adjunct{F}{C}{D}{G}$ with respect to $S$ in the sense of Lurie \cite[Def.~7.3.2.2]{HA} is said to be an \emph{$S$-adjunction} if $F$ and $G$ both preserve cocartesian edges.}
\[ \begin{tikzcd}[row sep=4ex, column sep=4ex, text height=1.5ex, text depth=0.25ex]
\sU \ar[shift right=1,right hook->]{r}[swap]{j_{\ast}} & \sX \ar[shift right=2]{l}[swap]{j^{\ast}} \ar[shift left=2]{r}{i^{\ast}} & \sZ \ar[shift left=1,left hook->]{l}{i_{\ast}}.
\end{tikzcd} \]
In what follows, we use the following terminology from \cite{Exp2}:
\begin{enumerate}
\item An \emph{$S$-$\infty$-category} is a cocartesian fibration $C \to S$.
\item Given two $S$-$\infty$-categories $C, D \to S$, the \emph{$\infty$-category of $S$-functors} $\Fun_S(C,D)$ is notation for $\Fun^{\cocart}_{/S}(C,D)$.
\end{enumerate}

We first show that the procedure of taking $S$-functor categories yields a recollement.

\begin{lem} \label{lem:FunctorCategoryRecollement} For any $S$-$\infty$-category $K$, $(\Fun_S(K,\sU), \Fun_S(K, \sZ))$ is a recollement of $\Fun_S(K,\sX)$.
\end{lem}
\begin{proof} By \cite[Prop.~8.4]{Exp2}, we have induced adjunctions given by postcomposition
\[ \begin{tikzcd}[row sep=4ex, column sep=4ex, text height=1.5ex, text depth=0.25ex]
\Fun_S(K,\sU) \ar[shift right=1,right hook->]{r}[swap]{\overline{j}_{\ast}} & \Fun_S(K,\sX) \ar[shift right=2]{l}[swap]{\overline{j}^{\ast}} \ar[shift left=2]{r}{\overline{i}^{\ast}} & \Fun_S(K,\sZ) \ar[shift left=1,left hook->]{l}{\overline{i}_{\ast}},
\end{tikzcd} \]
where it is clear that $\overline{j}^{\ast} \overline{j}_{\ast} \simeq \id$ and $\overline{i}^{\ast} \overline{i}_{\ast} \simeq \id$, hence $\overline{j}_{\ast}$ and $\overline{i}_{\ast}$ are fully faithful. By \cite[Prop.~5.4.7.11]{HTT}, the hypothesis that for all $f: s \to t$ the restriction functors $f^{\ast}: \sX_t \to \sX_s$ preserve finite limits ensures that $\Fun_S(K, \sX)$ admits finite limits (which are computed fiberwise), and similarly the induced restriction functors $f_U^{\ast}$ and $f_Z^{\ast}$ preserve finite limits, so $\Fun_S(K, \sU)$, $\Fun_S(K, \sZ)$ admit finite limits and $\overline{j}^{\ast}, \overline{i}^{\ast}$ preserve finite limits. Since $j^{\ast} i_{\ast} \simeq 0$ and the terminal object $0 \in \Fun_S(K, \sU)$ is given by $K \to S \xto{\underline{0}} \sU$ for the cocartesian section $\underline{0}: S \to \sU$ that selects the terminal object in each fiber, we get that $\overline{j}^{\ast} \overline{i}_{\ast} \simeq 0$. Finally, since a morphism $f$ in $\Fun_S(K,\sX)$ is an equivalence if and only if $f(k)$ is an equivalence for all $k \in K$, we deduce that $\overline{j}^{\ast}$ and $\overline{i}^{\ast}$ are jointly conservative using the joint conservativity of $j^{\ast}$ and $i^{\ast}$.
\end{proof}

\begin{cor} \label{cor:LimitsOfRecollements} The forgetful functors $\Recoll_{\str}^{\lex} \to \Cat_{\infty}$ and $\Recoll_{\str}^{\stab} \to \Cat_{\infty}^{\stab}$ create limits.
\end{cor}
\begin{proof} The first statement follows from \cref{lem:FunctorCategoryRecollement} by taking $K=S$ and using that the $\infty$-category of cocartesian sections computes the limit of a diagram of $\infty$-categories \cite[\S 3.3.3]{HTT}. We note that the proof of \cref{lem:FunctorCategoryRecollement} shows that the evaluation functors at any $s \in S$ are left-exact and strict morphisms of recollements, so the limit resides in $\Recoll_{\str}^{\lex}$. Finally, because limits in $\Cat_{\infty}^{\stab}$ are created in $\Cat_{\infty}$, the second statement follows.
\end{proof}

We can also use \cref{lem:FunctorCategoryRecollement} to compute $S$-colimits in $\sX$. For clarity, let us revert to the non-parametrized case $S = \ast$ for the next two results; the $S$-analogues will also hold by the same reasoning.

\begin{lem} \label{lem:ColimitExistenceInRecollement} Let $(\sU, \sZ)$ be a recollement of $\sX$ and suppose that $\sU$ and $\sZ$ admit $K$-indexed colimits. Then $\sX$ admits $K$-indexed colimits.
\end{lem}
\begin{proof} With respect to the recollement of $\Fun(K, \sX)$ of \cref{lem:FunctorCategoryRecollement}, the constant diagram functor $\delta: \sX \to \Fun(K, \sX)$ is obviously a morphism of recollements. Passing to left adjoints, we obtain a right-lax commutative diagram
\[ \begin{tikzcd}[row sep=4ex, column sep=4ex, text height=1.5ex, text depth=0.25ex]
\Fun(K, \sU) \ar{r}{\overline{i}^{\ast} \overline{j}_{\ast}} \ar{d}[swap]{\colim} \ar[phantom]{rd}{\SWarrow} & \Fun(K, \sZ) \ar{d}{\colim} \\
\sU \ar{r}[swap]{i^{\ast} j_{\ast}} & \sZ,
\end{tikzcd} \]
which induces a morphism of recollements $\colim: \Fun(K, \sX) \to \sX$. We claim that $\colim$ is left adjoint to $\delta$. In fact, if $\cM, \cM^K \to \Delta^1$ are the cartesian fibrations classified by $i^{\ast} j_{\ast}$ and $\overline{i}^{\ast} \overline{j}_{\ast}$ respectively, then we have a map $\delta: \cM^K \to \cM$ of cartesian fibrations and by \cite[Prop.~7.3.2.6]{HA} a relative left adjoint $\colim: \cM^K \to \cM$. The formation of sections sends relative adjunctions to adjunctions, which proves the claim. We deduce that $\sX$ admits $K$-indexed colimits.
% e note that if $S \neq \ast$, for $S$-colimits we have the additional compatibility with restriction to check, but this follows from the explicit construction of the left adjoint.
\end{proof}

\begin{cor} Suppose $\sU$ and $\sZ$ are presentable $\infty$-categories and $\phi: \sU \to \sZ$ is a left-exact accessible functor. Then $\sX = \mathrm{\lim}^{\rlax} \phi$ is a presentable $\infty$-category.
\end{cor}
\begin{proof} By \cref{lem:ColimitExistenceInRecollement}, $\sX$ admits all small colimits. By \cite[Cor.~5.4.7.17]{HTT}, $\sX$ is accessible. We conclude that $\sX$ is presentable.
\end{proof}

% \begin{nul}[Existence of colimits] With respect to the recollement of \cref{lem:FunctorCategoryRecollement}, the constant diagram functor $\delta: \sX \to \Fun_S(K, \sX)$ is obviously a morphism of recollements. Suppose that $\sU$ and $\sZ$ admit $K$-indexed $S$-colimits. Passing to left adjoints, we obtain a right-lax commutative diagram
% \[ \begin{tikzcd}[row sep=4ex, column sep=4ex, text height=1.5ex, text depth=0.25ex]
% \Fun_{S}(K, \sU) \ar{r}{\overline{i}^{\ast} \overline{j}_{\ast}} \ar{d}[swap]{\colim} \ar[phantom]{rd}{\SWarrow} & \Fun_S(K, \sZ) \ar{d}{\colim} \\
% \sU \ar{r}[swap]{i^{\ast} j_{\ast}} & \sZ,
% \end{tikzcd} \]
% which induces a morphism of recollements $\colim: \Fun_S(K, \sX) \to \sX$. We claim that $\colim$ is left adjoint to $\delta$. In fact, if $\cM, \cM^K \to \Delta^1$ are the cartesian fibrations classified by $i^{\ast} j_{\ast}$ and $\overline{i}^{\ast} \overline{j}_{\ast}$ respectively, then we have a map $\delta: \cM^K \to \cM$ of cartesian fibrations and by \cite[Prop.~7.3.2.6]{HA} a relative left adjoint $\colim: \cM^K \to \cM$. The formation of sections sends relative adjunctions to adjunctions, which proves the claim. We deduce that $\sX$ admits $K$-indexed $S$-colimits.\footnote{If $S \neq \ast$, for $S$-colimits we have the additional compatibility with restriction to check, but this follows from the construction of the left adjoint as a morphism of recollements.}
% \end{nul}

Finally, we describe how recollements interact with an ambidextrous adjunction (e.g., the adjunction between restriction and induction for equivariant spectra). 

\begin{lem} \label{lem:AmbidexterityRecollement} Let $(\sU,\sZ)$ and $(\sU',\sZ')$ be stable recollements on $\sX$ and $\sX'$ and let $f^{\ast}: \sX \to \sX'$ be an exact functor such that $f^{\ast}|_{i_{\ast}(\sZ)} \subset i_{\ast}(\sZ')$ (so $f^{\ast}$ is not necessarily a morphism of recollements, but we still may define ${f_U}^{\ast} \coloneq j'^{\ast} f^{\ast} j_{\ast}$, ${f_Z}^{\ast} \coloneq i'^{\ast} f^{\ast} i_{\ast}$, and have ${f_U}^{\ast} j^{\ast} \simeq j'^{\ast} {f_U}^{\ast}$).
% be a morphism of stable recollements $(\sU, \sZ) \to (\sU', \sZ')$ and suppose that $f^{\ast}$ admits a right adjoint $f_{\ast}$.
\begin{enumerate}
    \item Suppose that $f^{\ast}|_{j_!( \sU)} \subset j'_!(\sU')$ and $f^{\ast}$ admits a right adjoint $f_{\ast}$. Then
    \begin{enumerate}
        \item The essential image of $f_{\ast} j'_{\ast}$ lies in $j_{\ast}(\sU)$, so $f^{\ast} \dashv f_{\ast}$ restricts to an adjunction
        \[ \adjunct{{f_U}^{\ast}}{\sU}{\sU'}{{f_U}_{\ast}} \]
        with $j_{\ast} {f_U}_{\ast} \simeq f_{\ast} j'_{\ast}$.
        \item The natural map $j^{\ast} f_{\ast} \to {f_U}_{\ast} j'^{\ast}$ is an equivalence. 
        \item The essential image of $f_{\ast} i'_{\ast}$ lies in $i_{\ast}(\sZ)$, so $f^{\ast} \dashv f_{\ast}$ restricts to an adjunction
        \[ \adjunct{{f_Z}^{\ast}}{\sZ}{\sZ'}{{f_Z}_{\ast}} \]
        with $i_{\ast} {f_Z}_{\ast} \simeq f_{\ast} i'_{\ast}$.
        % \item Suppose that the natural map $i^{\ast} f_{\ast} \to {f_Z}_{\ast} i'^{\ast}$ is an equivalence. Then $f_{\ast}$ is a morphism of stable recollements $(\sU',\sZ') \to (\sU,\sZ)$ that lies in $\Recoll_{\str}^{\stab}$.
    \end{enumerate}
    \item Suppose that $f^{\ast}|_{j_{\ast}( \sU)} \subset j'_{\ast}(\sU')$ and $f^{\ast}$ admits a left adjoint $f_!$. Then
    \begin{enumerate}
        \item The essential image of $f_{\ast} j'_!$ lies in $j_!(\sU)$, so $f_! \dashv f^{\ast}$ restricts to an adjunction
        \[ \adjunct{{f_U}_!}{\sU'}{\sU}{{f_U}^{\ast}} \]
        with $j_! {f_U}_! \simeq f_! j'_!$.
        \item The natural map ${f_U}_! j^{\ast} \to j'^{\ast} f_!$ is an equivalence. 
        \item The essential image of $f_! i'_{\ast}$ lies in $i_{\ast}(\sZ)$, so $f_! \dashv f^{\ast}$ restricts to an adjunction
        \[ \adjunct{{f_Z}_!}{\sZ'}{\sZ}{{f_Z}^{\ast}} \]
        with $i_{\ast} {f_Z}_! \simeq f_! i'_{\ast}$.
        \item The natural map $i^{\ast} {f_Z}_! \to {f_Z}_! i'^{\ast}$ is an equivalence.
    \end{enumerate}
    \item Suppose that $f^{\ast} \in \Recoll^{\stab}_{\str}$, $f^{\ast}$ admits left and right adjoints $f_!$ and $f_{\ast}$, and we have the ambidexterity equivalence $f_! \simeq f_{\ast}$. Then $f_{\ast} \in \Recoll^{\stab}_{\str}$ and we additionally have ambidexterity equivalences ${f_U}_! \simeq {f_U}_{\ast}$ and ${f_Z}_! \simeq {f_Z}_{\ast}$.
\end{enumerate}
\end{lem}
\begin{proof} We first prove the assertions of (1). For (1.1), for any $u' \in \sU'$ because we have for all $z \in \sZ$ that
\begin{align*} \Map_{\sX}(i_{\ast} z,f_{\ast} j'_{\ast} u') \simeq \Map_{\sU'}(j'^{\ast} f^{\ast} i_{\ast} z, u') \simeq \Map_{\sU'}(f_U^{\ast} j'^{\ast} i_{\ast} z, u') \simeq \ast,
\end{align*}
we get $f_{\ast} j'_{\ast} u' \in j_{\ast}(\sU)$. For (1.2), the assertion holds because the map is adjoint to the equivalence $f^{\ast} j_! \to j'_! {f_U}^{\ast}$. For (1.3), for any $z' \in \sZ'$ we have
\[ j^{\ast} f_{\ast} i'_{\ast} z' \simeq {f_U}_{\ast} j^{\ast} i'_{\ast} z' \simeq {f_U}_{\ast} 0 \simeq 0, \]
hence $f_{\ast} i'_{\ast} z' \in i_{\ast}(\sZ)$. Next, the assertions of (2) hold by a dual argument; we note that the extra assertion (2.4) holds because $f_!$ now commutes with $j_!$ instead of $j_{\ast}$. Finally, for (3) the functor $f_! \simeq f_{\ast}$ is in $\Recoll^{\stab}_{\str}$ by combining (1.1), (1.3), and (2.1). For the ambidexterity assertions, the equivalence ${f_Z}_! \simeq {f_Z}_{\ast}$ is clear because the embedding $i_{\ast}: \sZ \subset \sX$ is unambiguous, whereas for ${f_U}_! \simeq {f_U}_{\ast}$ we note that the sequence of equivalences
\begin{align*} \Map_{\sU}(u, {f_U}_! u') & \simeq \Map_{\sX}(j_! u, f_! j'_! u') \simeq \Map_{\sX} (j_! u, f_{\ast} j'_! u') \simeq \Map_{\sX'}(f^{\ast} j_! u, j'_! u') \\
& \simeq \Map_{\sX'} (j'_! {f_U}^{\ast} u, j'_! u') \simeq \Map_{\sU'} ({f_U}^{\ast} u, u')
\end{align*}
demonstrates that ${f_U}_!$ is right adjoint to ${f_U}^{\ast}$ and hence ${f_U}_! \simeq {f_U}_{\ast}$.
% Finally, (1.4) is obvious; we include the statement only to highlight that $f_{\ast}$ doesn't necessarily send $j_!(\sU')$ into $j_!(\sU)$.
\end{proof}

% For the subsequent corollary, we suppose that $T = S^{\op}$ is an atomic orbital $\infty$-category \cite[Def.~4.1]{Exp4}. For example, we may take $S = \Ar_G^{\op}$.

% \footnote{A $T$-functor $\phi: C \to D$ between $T$-stable $T$-$\infty$-categories is \emph{$T$-exact} if it is fiberwise exact and for all $f: s \to t$, the natural map $\phi_t (f_C)_{\ast} \to (f_D)_{\ast} \phi_s$ is an equivalence.}
\begin{cor} \label{cor:RecollementGivesParamStableSubcategories} Let $G$ be a finite group. Suppose that $\sX_{\bullet}: \sO^{\op}_G \to \Recoll^{\stab}_{\str}$ is a functor such that the underlying $G$-$\infty$-category $\sX$ is $G$-stable \cite[Def.~7.1]{Exp4}. Then $\sU$ and $\sZ$ are $G$-stable and all of the functors appearing in the diagram of $G$-adjunctions
\[ \begin{tikzcd}[row sep=4ex, column sep=4ex, text height=1.5ex, text depth=0.25ex]
\sU \ar[shift right=1,right hook->]{r}[swap]{j_{\ast}} & \sX \ar[shift right=2]{l}[swap]{j^{\ast}} \ar[shift left=2]{r}{i^{\ast}} & \sZ \ar[shift left=1,left hook->]{l}{i_{\ast}}
\end{tikzcd} \]
are $G$-exact.
\end{cor}
\begin{proof} By \cref{lem:AmbidexterityRecollement}, it only remains to check the Beck-Chevalley condition for $\sU$ and $\sZ$ to show the existence of finite $G$-products. But this follows from the same condition on $\sX$, since the restriction and induction functors $(f_{-})^{\ast}, (f_{-})_{\ast}$ commute with the inclusion functors $(j_{\bullet})_{\ast}$, $(j_{\bullet})_!$, and $(i_{\bullet})_{\ast}$.
\end{proof}

\begin{dfn} \label{dfn:ParamStableRecollement} In the situation of \cref{cor:RecollementGivesParamStableSubcategories}, we say that $(\sU,\sZ)$ is a \emph{$G$-stable $G$-recollement} of $\sX$.
\end{dfn}

\section{Recollements on lax limits of \texorpdfstring{$\infty$}{infinity}-categories}

Let $S$ be an $\infty$-category throughout this section. Suppose $p: C \to S$ is a locally cocartesian fibration classified by a $2$-functor $f: \fC[S] \to \Cat_{\infty}$ (\cite[Def.~1.1.5.1]{HTT} and \cite[\S 3]{G}), so for every $2$-simplex $\Delta^2 \to S$, we have a lax commutative diagram of $\infty$-categories
\[ \begin{tikzcd}[row sep=4ex, column sep=4ex, text height=1.5ex, text depth=0.25ex]
C_{0} \ar{rd}[swap]{F_{01}} \ar{rr}{F_{02}} & \ar[phantom]{d}{\Downarrow} & C_{2}, \\
& C_{1} \ar{ru}[swap]{F_{12}} &
\end{tikzcd} \]
and the higher-dimensional simplices of $S$ supply coherence data. Then the $2$-functoriality of $f$ yields two notions of lax limit corresponding to choosing two possible orientations for morphisms -- informally, the \emph{left-lax} limit of $f$ has objects given by tuples $(x_i \in C_i, \alpha_{ij}: F_{ij}(x_i) \to x_j)$, whereas the \emph{right-lax} limit of $f$ has objects given by tuples $(x_i \in C_i, \alpha_{ij}: x_j \to F_{ij}(x_i))$. To give rigorous meaning to these notions, we may circumvent giving a precise formulation of the lax universal property (for instance, as carried out in \cite{GHN}) and instead \emph{define} the left-lax limit to be the $\infty$-category of sections
\[ \mathrm{\lim}^{\llax}f = \mathrm{\lim}^{\llax} C \coloneq \Fun_{/S}(S,C) \]
and the right-lax limit to be the $\infty$-category
\[ \mathrm{\lim}^{\rlax}f = \mathrm{\lim}^{\rlax} C \coloneq \Fun^{\cocart}_{/S}(\sd(S),C), \]
where $\sd(S)$ is the \emph{barycentric subdivision} of $S$ (\cref{dfn:barycentricSubdivision}) that is locally cocartesian over $S$ via the $\mathit{max}$ functor (\cref{cnstr:MaxFunctorSubdivision}), and we let $\Fun^{\cocart}_{/S}(-,-)$ be the full subcategory on those functors over $S$ that preserve \emph{locally cocartesian} edges. Viewing $f$ itself as a \emph{left-lax diagram} in $\Cat_{\infty}$, we may thereby speak of left-lax and right-lax limits of left-lax diagrams of $\infty$-categories; dually, we may also speak of left-lax and right-lax limits of right-lax diagrams of $\infty$-categories encoded as locally cartesian fibrations. We refer to \cite[\S 1]{AMGR-NaiveApproach} or \cite[\S A]{AMGRb} for a more detailed discussion.\footnote{We follow \cite[\S 1]{AMGR-NaiveApproach} in referring to these two types of lax limits as `left' and `right', even though lax and oplax are more standard nomenclature. The terminology is consistent with the usage of left for cocartesian-type constructions and right for cartesian-type constructions (e.g., left and right fibrations).}

\begin{dfn} Let $S' \subset S$ be a full subcategory. Then $S'$ is a \emph{sieve} if for every morphism $x \to y$ in $S$, if $y \in S'$, then $x \in S'$. Dually, $S'$ is a \emph{cosieve} if $(S')^{\op}$ is a sieve in $S^{\op}$.

Given a sieve $S_0 \subset S$ and cosieve $S_1 \subset S$, we say that $S_0$ and $S_1$ form a \emph{sieve-cosieve decomposition} of $S$ if $S_0$ and $S_1$ are disjoint and any object $x \in S$ lies either in $S_0$ or $S_1$.
\end{dfn}

\begin{rem} Note that sieves and cosieves are necessarily stable under equivalences. Given a sieve-cosieve decomposition $(S_0,S_1)$ of $S$, we may define a functor $\pi: S \to \Delta^1$ that sends each object $x \in S$ to the integer $i \in \{0,1\}$ such that $x \in S_i$. Conversely, any functor $\pi: S \to \Delta^1$ determines a sieve-cosieve decomposition of $S$ by taking its fibers over $0$ and $1$.
\end{rem}

Our main goal in this section is to describe how sieve-cosieve decompositions of $S$ produces recollements on right-lax limits of locally cocartesian fibrations $p:C \to S$ (\cref{thm:RecollementRlaxLimitOfLlaxFunctor}).

\begin{rem} As we saw in \cref{recollEquivalenceToOplaxLim}, a recollement itself is an example of a right-lax limit over $\Delta^1$. Given a working theory of (pointwise) right-lax Kan extensions, our results should follow from the usual transitivity property of Kan extensions applied to the factorization $S \xto{\pi} \Delta^1 \to \ast$. However, we are not aware of such a theory that also affords the explicit description of the gluing functor given in \cref{thm:ExistenceLaxRightKanExtension}; indeed, \cref{thm:ExistenceLaxRightKanExtension} should precisely amount to a pointwise formula for the right-lax Kan extension along $\pi$. We refer the interested reader to the discussion in \cite[\S 2.2]{kondyrev2021dualizable} for more on this question.
\end{rem}

\subsection{Recollements on right-lax limits of strict diagrams}

Before entering into our study of left-lax diagrams, let us consider the simpler case of strict diagrams $f: S \to \Cat_{\infty}$. For this case, right-lax limits are modeled by sections of the \emph{cartesian} fibration that classifies $f$. Thus suppose that $p: C \to S$ is a cartesian fibration, $\pi: S \to \Delta^1$ is a functor, and let $p_0: C_0 \to S_0$, $p_1: C_1 \to S_1$ denote the pullbacks of $p$ to the fibers $S_0$, $S_1$ of $\pi$. Given a section $F: S \to C$ of $p$, let $j^{\ast} F: S_1 \to C_1$ be its restriction over $S_1$ and let $i^{\ast} F: S_0 \to C_0$ be its restriction over $S_0$. We obtain functors
\[ j^{\ast}: \Fun_{/S}(S,C) \to \Fun_{/S_1}(S_1,C_1), \quad i^{\ast}: \Fun_{/S}(S,C) \to \Fun_{/S_0}(S_0, C_0). \]
We first explain when $j^{\ast}$ and $i^{\ast}$ admit right adjoints. Suppose $G: S_1 \to C_1$ is a section of $p_1$. For every $x \in S$, let
\[ G_x: (S_1)_{x/} \coloneq S_1 \times_S S_{x/} \to S_1 \xto{G} C_1 \subset C \] 
be the composite functor and consider the commutative diagram
\[ \begin{tikzcd}[row sep=4ex, column sep=6ex, text height=1.5ex, text depth=0.25ex]
(S_1)_{x/} \ar{r}{G_x} \ar{d} & C \ar{d}{p} \\
(S_1)_{x/}^{\lhd} \ar{r} \ar[dotted]{ru}[swap]{\overline{G_x}} & S
\end{tikzcd} \]
where the cone point is sent to $x$. By \cite[Cor.~4.3.1.11]{HTT}, if for every $s \in S$, $C_s$ admits $(S_1)_{x/}$-indexed limits, and for every $f: s \to t$, the pullback functor $f^{\ast}: C_t \to C_s$ preserves $(S_1)_{x/}$-indexed limits, then there exists a dotted lift $\overline{G_x}$ which is a $p$-limit of $G_x$. If this holds for all $x \in S$,  then by the dual of \cite[Lem.~4.3.2.13]{HTT}, the $p$-right Kan extension $j_{\ast} G$ exists and is computed pointwise by these $p$-limits. Moreover, by \cite[Prop.~4.3.2.17]{HTT}, the right adjoint $j_{\ast}$ then exists and is computed objectwise by $j_{\ast} G$.

Now let $H: S_0 \to C_0$ be a section of $p_0$. The same results hold for computing $i_{\ast} H$. However, the slice $\infty$-categories $(S_0)_{x/}$ are empty when $x \in S_1$. Therefore, the hypotheses above amount to supposing that for all $s \in S$, $C_s$ admits a terminal object, and for all $f: s \to t$, the pullback functor $f^{\ast}$ preserves this terminal object.

Finally, let $\cK = \{K_{\alpha}\}_{\alpha \in A}$ be a class of simplicial sets and suppose that for all $K \in \cK$ and $s \in S$, the fiber $C_s$ admits $K$-indexed limits, and for all $f: s \to t$, the pullback functor $f^{\ast}$ preserves $K$-indexed limits. Then by the dual of \cite[Prop.~5.4.7.11]{HTT} and \cite[Rmk.~5.4.7.13]{HTT}, $\Fun_{/S}(S,C)$ admits $K$-indexed limits such that the evaluation functors $\ev_s: \Fun_{/S}(S,C) \to C_s$ preserve $K$-indexed limits -- in other words, the $K$-indexed limits in $\Fun_{/S}(S,C)$ are computed fiberwise. 
% The same also holds for $K$-indexed colimits. 

Let us now suppose that $p$ satisfies this condition for $\cK$ the class of finite simplicial sets and also satisfies the existence hypotheses for $j_{\ast}$.

\begin{prp} \label{prp:RecollementSectionCategoryCartesianFibration} The adjunctions 
\[ \begin{tikzcd}[row sep=4ex, column sep=4ex, text height=1.5ex, text depth=0.25ex]
\Fun_{/S_1}(S_1,C_1) \ar[shift right=1,right hook->]{r}[swap]{j_{\ast}} & \Fun_{/S}(S,C) \ar[shift right=2]{l}[swap]{j^{\ast}} \ar[shift left=2]{r}{i^{\ast}} & \Fun_{/S_0}(S_0,C_0) \ar[shift left=1,left hook->]{l}{i_{\ast}}
\end{tikzcd} \]
together exhibit $\Fun_{/S}(S,C)$ as a recollement of $\Fun_{/S_1}(S_1,C_1)$ and $\Fun_{/S_0}(S_0,C_0)$.
\end{prp}
\begin{proof} Note the functors $j^{\ast}$ and $i^{\ast}$ are left exact by the fiberwise computation of limits in section $\infty$-categories. Because $(S_0)_{x/} = \emptyset$ for all $x \in S_1$, we get that $j^{\ast} i_{\ast}$ is the constant functor at the terminal object of $\Fun_{/S_1}(S_1,C_1)$. Finally, $i^{\ast}$ and $j^{\ast}$ are jointly conservative because equivalences are detected objectwise in $\Fun_{/S}(S,C)$.
\end{proof}

\begin{rem} If the fibers of $p$ are moreover stable $\infty$-categories, then the left-exact pullback functors $f^{\ast}$ are necessarily exact and the recollement of \cref{prp:RecollementSectionCategoryCartesianFibration} is stable.
\end{rem}

\begin{exm} \label{exm:SieveCosieveRecollementOnFunctorCategory} Let $C \simeq D \times S$ and $p$ be the projection to $S$. Then the recollement of \cref{prp:RecollementSectionCategoryCartesianFibration} simplifies to 
\[ \begin{tikzcd}[row sep=4ex, column sep=4ex, text height=1.5ex, text depth=0.25ex]
\Fun(S_1,D) \ar[shift right=1,right hook->]{r}[swap]{j_{\ast}} & \Fun(S,D) \ar[shift right=2]{l}[swap]{j^{\ast}} \ar[shift left=2]{r}{i^{\ast}} & \Fun(S_0,D) \ar[shift left=1,left hook->]{l}{i_{\ast}}
\end{tikzcd} \]
where $j: S_1 \to S$ and $i: S_0 \to S$ now denote the inclusions. Recollement theory then gives a calculational technique for computing the right Kan extension $\phi_{\ast} F$ of a functor $F: S \to D$ along $\phi: S \to T$. Namely, if we let $\phi_0 = \phi \circ i$, $\phi_1 = \phi \circ j$, $F_0 = F|_{S_0}$, and $F_1 = F|_{S_1}$, the pullback square \cref{recollementFractureSquare} yields a pullback square
\[ \begin{tikzcd}[row sep=4ex, column sep=4ex, text height=1.5ex, text depth=0.25ex]
\phi_{\ast} F \ar{r} \ar{d} & (\phi_0)_{\ast} F_0 \ar{d} \\
(\phi_1)_{\ast} F_1 \ar{r} & (\phi_0)_{\ast} \left( \left(j_{\ast} F_1 \right)|_{S_0} \right).
\end{tikzcd} \]
\end{exm}

\subsection{Recollements on right-lax limits of left-lax diagrams} \label{subsection:recollRLaxLLax}

We now seek to establish the analogue of \cref{prp:RecollementSectionCategoryCartesianFibration} for right-lax limits of locally cocartesian fibrations. Although the ideas are straightforward, the categorical details turn out to be considerably more involved. We begin by proving some needed extensions to the theory of relative right Kan extensions initiated in \cite[\S 4.1-3]{HTT}, which play a technical role in our construction of the recollement adjunctions. We then construct the barycentric subdivision $\sd(S)$ (\cref{dfn:barycentricSubdivision}, but also see \cref{rem:barycentricSubdivisionOrdinaryCategory}), and extend the cocartesian pushforward of \cite[Lem.~2.23]{Exp2} to the locally cocartesian situation (\cref{prp:LocallyCocartesianPushforward} and \cref{thm:subdivisionExtension}). Finally, given a sieve-cosieve decomposition of $S$ and suitable hypotheses on the locally cocartesian fibration $p: C \to S$, we establish localizations in \cref{thm:ExistenceLaxRightKanExtension}, \cref{cor:openPartOfRecollement}, and \cref{prp:closedPartOfRecollement}, and show that these together constitute a recollement of the right-lax limit of $p$ in \cref{thm:RecollementRlaxLimitOfLlaxFunctor}.

\subsubsection{Relative right Kan extension}

In \cite[Prop.~4.3.1.10]{HTT}, Lurie gives a criterion for when a colimit diagram in a fiber of a locally cocartesian fibration is a relative colimit. In contrast, we will also need a separate understanding of when a \emph{limit} diagram in a fiber is a relative limit. As indicated in \cref{lem:LimitIsRelativeLimit}, in this situation we can give an unconditional statement.

\begin{lem} \label{lem:LimitIsRelativeLimit} Let $S$ be an $\infty$-category and let $f: C \to S$ be a locally cocartesian fibration. Let $s \in S$ be an object and $\overline{p}: K^{\lhd} \to C_s$ a limit diagram that extends $p$. Then, viewed as a diagram in $C$, $\overline{p}$ is a $f$-limit diagram \cite[4.3.1.1]{HTT}, i.e., the commutative square
\[ \begin{tikzcd}[row sep=4ex, column sep=4ex, text height=1.5ex, text depth=0.25ex]
C_{/\overline{p}} \ar{r} \ar{d} & C_{/p} \ar{d} \\
S_{/f \overline{p}} \ar{r} & S_{/f p}
\end{tikzcd} \]
is a homotopy pullback square.
\end{lem}
\begin{proof} It suffices to show that $C_{/\overline{p}} \to C_{/p} \times_{S_{/f p}} S_{/f \overline{p}}$ is a trivial Kan fibration. To this end, let $A \to B$ be a monomorphism of simplicial sets and consider the lifting problem
\[ \begin{tikzcd}[row sep=4ex, column sep=4ex, text height=1.5ex, text depth=0.25ex]
A \ar{r} \ar{d} & C_{/\overline{p}} \ar{d} \\
B \ar{r} \ar[dotted]{ru} &  C_{/p} \times_{S_{/f p}} S_{/f \overline{p}}.
\end{tikzcd} \]
This transposes to the lifting problem
\[ \begin{tikzcd}[row sep=4ex, column sep=4ex, text height=1.5ex, text depth=0.25ex]
A \star K^{\lhd} \bigcup_{A \star K} B \star K \ar{r}{\beta} \ar{d} & C \ar{d}{f} \\
B \star K^{\lhd} \ar{r}[swap]{\alpha} \ar[dotted]{ru}[swap]{\gamma} & S.
\end{tikzcd} \]
Our approach will be to first pushforward to the fiber $C_s$ using that $f$ is a locally cocartesian fibration and then solve the lifting problem in $C_s$ using that $\overline{p}$ is a limit diagram.

To begin, because $\overline{p}$ is a diagram in the fiber $C_s$, the map $\alpha$ factors as $B \star K^\rhd \to B \star \Delta^0 \xto{\alpha'} S$ with $\alpha'|_{\Delta^0} = \{s\}$. We may define a map $r: (B \star \Delta^0) \times \Delta^1 \to B \star \Delta^0$ such that $r_0 = \id$ and $r_1$ is constant at $\Delta^0$ in the following way: let $\pi: B \star \Delta^0 \to \Delta^1$ be the structure map of the join which sends $B$ to $\{0\}$ and $\Delta^0$ to $\{1\}$, and let $\rho$ be the composite $(B \star \Delta^0) \times \Delta^1 \xto{\pi \times \id} \Delta^1 \times \Delta^1 \xto{\text{max}} \Delta^1$, so the fiber of $\rho$ over $\{0\}$ is $B \times \{0\}$. Then, recalling that maps $L \to X \star Y$ of simplicial sets over $\Delta^1$ are equivalently specified by pairs of maps $(f_0:L_0 \to X, f_1: L_1 \to Y)$, $r$ is the map over $\Delta^1$ with respect to $\rho$ and $\pi$ given by $B \subset B \star \Delta^0$ and the constant map to $\Delta^0$. Now let
\[ h^\alpha: (B \star K^{\lhd}) \times \Delta^1 \to (B \star \Delta^0) \times \Delta^1 \xto{r} B \star \Delta^0 \xto{\alpha'} S,  \]
so $h^\alpha_0 = \alpha$ and $h^\alpha_1$ is constant at $\{s\}$. Also denote by $h^\alpha$ the restrictions of $h^\alpha$ to $(B \star K) \times \Delta^1$, $(A \star K^{\lhd}) \times \Delta^1$, and $(A \star K) \times \Delta^1$.

Let $\mathfrak{P} = (M_S,T,\emptyset)$ be the categorical pattern on $s\Set^+_{/S}$ that yields the locally cocartesian model structure, so $M_S$ consists of all the edges in $S$, $T$ consists of all the degenerate $2$-simplices in $S$, and the fibrant objects are the locally cocartesian fibrations. By the criterion of \cite[Lem.~B.1.10]{HA} applied to $K \to B \star K$ (with the degenerate edges marked) and $\{0\} \to (\Delta^1)^\sharp$, the inclusion map of marked simplicial sets
\[ (B \star K) \times \{0\} \cup_{(K \times \{0\})} K \times (\Delta^1)^\sharp \to (B \star K) \times (\Delta^1)^\sharp \]
is $\mathfrak{P}$-anodyne, and likewise replacing $K \to B \star K$ with $K^\lhd \to A \star K^\lhd$ and $K \to A \star K$. Using left properness of the locally cocartesian model structure, we deduce that the morphism
\[ \begin{tikzcd}[row sep=4ex, column sep=4ex, text height=1.5ex, text depth=0.25ex]
(A \star K^\lhd \cup_{A \star K} B \star K) \times \{0\} \cup_{K^\lhd \times \{0\}} K^\lhd \times (\Delta^1)^\sharp \ar{d} \\
(A \star K^\lhd \cup_{A \star K} B \star K) \times (\Delta^1)^\sharp
\end{tikzcd} \] 
is $\mathfrak{P}$-anodyne. Consider the commutative square
\[ \begin{tikzcd}[row sep=4ex, column sep=4ex, text height=1.5ex, text depth=0.25ex]
(A \star K^\lhd \cup_{A \star K} B \star K) \times \{0\} \cup_{K^\lhd \times \{0\}} K^\lhd \times (\Delta^1)^\sharp \ar{d} \ar{r} & \leftnat{C} \ar{d}{f} \\
(A \star K^\lhd \cup_{A \star K} B \star K) \times (\Delta^1)^\sharp \ar{r}[swap]{h^{\alpha}} \ar[dotted]{ru}[swap]{h^\beta} & S^\sharp
\end{tikzcd} \]
where $\leftnat{C}$ denotes the marking on $C$ given by the $f$-locally cocartesian edges and the top horizontal map restricted to the first factor is $\beta$ and to the second factor $K^\lhd \times (\Delta^1)^\sharp$ is the constant homotopy $K^{\lhd} \times \Delta^1 \xto{\pr} K^{\lhd} \xto{\overline{p}} C$. Then the dotted lift $h^\beta$ exists, and the image of $h^\beta_1$ is contained in the fiber $C_s$.

Now consider the commutative triangle
\[ \begin{tikzcd}[row sep=4ex, column sep=4ex, text height=1.5ex, text depth=0.25ex]
A \star K^\lhd \cup_{A \star K} B \star K \ar{r}{h^\beta_1} \ar{d} & C_s \\
B \star K^{\lhd} \ar[dotted]{ru}[swap]{\gamma_1}
\end{tikzcd} \]
Because $\overline{p}: K^\lhd \to C_s$ is a limit diagram, the map $(C_s)_{/\overline{p}} \to (C_s)_{/p}$ is a trivial Kan fibration. Therefore, the dotted lift $\gamma_1$ exists. 

% Next, define a map $r': (B \times \Delta^1) \star \Delta^0 \to B \star \Delta^0$ as follows: let $\rho': (B \times \Delta^1) \star \Delta^0 \xto{\pr \star \id} \Delta^1 \star \Delta^0 \cong \Delta^2 \xto{\sigma^1} \Delta^1$ where $\sigma^1(0) = 0$, $\sigma^1(1) = 1$, and $\sigma^1(2) = 1$, so the fiber of $\rho'$ over $\{0\}$ is $B \times \{0\}$, and then let $r'$ be the map over $\Delta^1$ with respect to $\rho'$ and the structure map $\pi$ specified by $\id_B$ over $\{0\}$ and the constant map to $\Delta^0$ over $\{1\}$. Let
% \[ (h')^{\alpha}: (B \times \Delta^1) \star K^{\lhd} \to (B \times \Delta^1) \star \Delta^0 \xto{r'} B \star \Delta^0 \xto{\alpha'} S, \]
% so $(h')^{\alpha}_0 = \alpha$ and $(h')^{\alpha}_1$ is constant at $\{s\}$.
Next, define a map
\[ \theta = (\theta', \theta''): (B \times \Delta^1) \star K^{\lhd} \to (B \star K^{\lhd}) \times \Delta^1 \]
by its factors
\begin{align*} \theta' &: (B \times \Delta^1) \star K^{\lhd} \xto{\pr \star \id} B \star K^{\lhd} \\
\theta'' &: (B \times \Delta^1) \star K^{\lhd} \xto{\pr \star \id} \Delta^1 \star K^\lhd \to \Delta^1 \star \Delta^0 \cong \Delta^2 \xto{\sigma^1} \Delta^1.
\end{align*}
Here $\sigma^1: \Delta^2 \to \Delta^1$ is the standard degeneracy map, so $\sigma^1(0) = 0$, $\sigma^1(1) = 1$, and $\sigma^1(2) = 1$. Also denote by $\theta$ the restriction to $(A \times \Delta^1) \star K^{\lhd}$, etc. Let
 \[ X = (A \times \Delta^1) \star K^{\lhd} \cup_{(A \times \Delta^1) \star K} (B \times \Delta^1) \star K  \bigcup\limits_{ (A \times \{1\}) \star K^{\lhd} \cup_{(A \times \{1\}) \star K} (B \times \{1 \}) \star K} B \star K^{\lhd} \]
 and consider the commutative diagram
\[ \begin{tikzcd}[row sep=4ex, column sep=4ex, text height=1.5ex, text depth=0.25ex]
 X \ar{r}{(h^\beta \circ \theta) \cup \gamma_1} \ar{d}[swap]{\lambda} & C \ar{d}{f} \\
 (B \times \Delta^1) \star K^{\lhd} \ar{r}[swap]{h^{\alpha} \circ \theta} \ar[dotted]{ru}[swap]{h^\gamma} & S
\end{tikzcd} \]
(where for commutativity, we use that $\theta_1: (B \times \{1\}) \star K^{\lhd} \to (B \star K^{\lhd}) \times \{1\}$ is an isomorphism). By the dual of \cite[Lem.~2.1.2.4]{HTT} applied to $A \to B$ and the right anodyne map $\{1\} \to \Delta^1$, the map
\[ \lambda': A \times \Delta^1 \cup_{A \times \{1\}} B \times \{1\} \to B \times \Delta^1 \]
is right anodyne. Then by \cite[Lem.~2.1.2.3]{HTT} applied to $\lambda'$ and the map $K \to K^{\lhd}$, $\lambda$ is inner anodyne. Thus the dotted lift $h^\gamma$ exists. Finally, let $\gamma = h^\gamma_0$ and observe that $\gamma$ is a solution to the original lifting problem of interest.
\end{proof}

% Remarks on relative Kan extensions
We briefly digress to complete the theory of Kan extensions by constructing relative Kan extensions along general functors (cf. Lurie's remark at the beginning of \cite[\S 4.3.3]{HTT}). Recall the relative join construction $- \star_{-} -$ of \cite[Def.~4.1]{Exp2} along with its bifibration property \cite[Lem.~4.8]{Exp2}.

%Shortcut definition of a relative left/right Kan extension
\begin{dfn} \label{Dfn:RelativeKanExtension} Consider the commutative diagram of $\infty$-categories
\[ \begin{tikzcd}[row sep=4ex, column sep=4ex, text height=1.5ex, text depth=0.25ex]
X \ar{r}{F} \ar{d}{\phi} & C \ar{d}{p} \\
Y \ar{r}{\alpha} & S
\end{tikzcd} \]
where $p: C \to S$ is a categorical fibration. Suppose given the data of a functor $G: Y \to C$ over $S$ and a homotopy $h: X \times \Delta^1 \to C$ over $S$ with $h_0 = G \circ \phi$ and $h_1 = F$. Let $\pi: Y \star_Y X \to Y$ be the structure map and let $\overline{G}: Y \star_Y X \xto{\pi} Y \xto{G} C$. Since $\Fun(Y \star_Y X, C) \to \Fun(Y, C) \times \Fun(X, C)$ is a bifibration, we may select an edge $\overline{G} \to \overline{F}$ that is cocartesian over $h: G \circ \phi \to F$ in $\Fun(X,C)$ with degenerate image $\id_{G}$ in $\Fun(Y,C)$. Then we say that $G$ is a \emph{$p$-right Kan extension of $F$} along $\phi$ (exhibited via $h$) if the commutative diagram
\[ \begin{tikzcd}[row sep=4ex, column sep=4ex, text height=1.5ex, text depth=0.25ex]
X \ar{r}{F} \ar[hookrightarrow]{d}{\iota_X} & C \ar{d}{p} \\
Y \star_Y X \ar{r}{\alpha \circ \pi} \ar{ru}{\overline{F}} & S
\end{tikzcd} \]
exhibits $\overline{F}$ as a $p$-right Kan extension of $F$ in the sense of \cite[Def.~4.3.2.2]{HTT}.
\end{dfn}

\begin{rem} \label{rem:ExistenceOfRelativeRKE} In the initial setup of \cref{Dfn:RelativeKanExtension}, given $\overline{F}: Y \star_Y X \to C$ a map over $S$ extending $F: X \to C$, let $G = \overline{F}|_Y: Y \to C$ and let $h: X \times \Delta^1 \xto{h'} Y \star_Y X \xto{\overline{F}} C$ with $h'$ specified by the pair $(\phi, \id_Y)$ (cf. the definition \cite[Def.~4.1]{Exp2} of $- \star_Y -$ as $j_\ast: s\Set_{/Y \times \partial \Delta^1} \to s\Set_{/Y \times \Delta^1}$ for the inclusion $j: Y \times \partial \Delta^1 \to Y \times \Delta^1$). Then $\overline{F}$ is a $p$-right Kan extension in the sense of \cite[Def.~4.3.2.2]{HTT} if and only if $G$ is a $p$-right Kan extension along $\phi$ in the sense of \cref{Dfn:RelativeKanExtension}. Moreover, we have an equivalence of $\infty$-categories $X \times_{Y \star_Y X} (Y \star_Y X)_{y/} \simeq X \times_Y Y_{y/}$ implemented by pulling back the functors $\iota_Y: Y \subset Y \star_Y X$ and $\pi: Y \star_Y X \to Y$ and the respective induced functors on the slice categories via $X \subset Y \star_Y X$. Because of this, Lurie's existence and uniqueness theorem \cite[Prop.~4.3.2.15]{HTT} for $p$-right Kan extensions applies to show that the $p$-right Kan extension $G$ of $F$ along $\phi$ exists if and only if for every $y \in Y$, the diagram $X \times_Y Y_{y/} \to X \xto{F} C$ extends to a $p$-limit diagram (which then computes the value of $G$ on $y$). Moreover, there is then a contractible space of choices for $G$.
\end{rem}

\begin{rem} \label{rem:AdjunctionForRKE} The situation of \cref{Dfn:RelativeKanExtension} globalizes in the following manner. Suppose every functor $F: X \to C$ admits a $p$-right Kan extension to $\overline{F}: Y \star_Y X \to C$. By \cite[Prop.~4.3.2.17]{HTT}, the restriction functor $(\iota_X)^\ast: \Fun_{/S}(Y \star_Y X, C) \to \Fun_{/S}(X,C)$ then admits a right adjoint $(\iota_X)_\ast$ which is computed on objects as $F \mapsto \overline{F}$. We also have a relative adjunction (\cite[Def.~7.3.2.2]{HA}) $$\adjunct{\iota_Y}{Y}{Y \star_Y X}{\pi}$$ over $Y$ (hence over $S$) where $\iota_Y$ is left adjoint to $\pi$. From this, we obtain an adjunction $$\adjunct{\pi^\ast}{\Fun_{/S}(Y,C)}{\Fun_{/S}(Y \star_Y X,C)}{(\iota_Y)^\ast}$$ where $\pi^\ast$ is left adjoint to $(\iota_Y)^\ast$. Composing these two adjunctions, we obtain the adjunction
\[ \adjunct{\phi^\ast}{\Fun_{/S}(Y,C)}{\Fun_{/S}(X,C)}{\phi_\ast} \]
where $\phi_{\ast}$ is given on objects by sending $F$ to its $p$-right Kan extension along $\phi$.
\end{rem}

% for all $y \in Y$, the inclusion of the fiber $X_y \to X \times_Y Y_{y/}$ is a right cofinal functor in the sense of \cite[Notation]{HA}\footnote{A map $q: K \to L$ is right cofinal if and only if $q^\op$ is cofinal. Other authors term right cofinal maps \emph{final} maps.} (e.g., $\phi$ is a cartesian fibration)
\begin{cor} \label{cor:RightKanExtensionComputedInFiber} Suppose we have a commutative diagram of $\infty$-categories
\[ \begin{tikzcd}[row sep=4ex, column sep=4ex, text height=1.5ex, text depth=0.25ex]
X \ar{r}{F} \ar{d}{\phi} & C \ar{d}{p} \\
Y \ar{r}{\alpha} & S
\end{tikzcd} \]
where $p$ is a locally cocartesian fibration and $\phi$ is a cartesian fibration. Suppose that for every $y \in Y$, the limit of $F|_{X_y}: X_y \to C_{\alpha(y)}$ exists. Then the $p$-right Kan extension $G: Y \to C$ of $F$ along $\phi$ exists and $G(y) \simeq \lim\limits_{\ot} F|_{X_y}$. If $G$ exists for all $F$, then we have an adjunction
\[ \adjunct{\phi^\ast}{\Fun_{/S}(Y,C)}{\Fun_{/S}(X,C)}{\phi_\ast} \]
where $\phi_{\ast}(F) \simeq G$.
% \item INSERT CORRECT HYPOTHESIS HERE.
% Suppose in addition that $\alpha$ and $\alpha \circ \phi$ are locally cocartesian fibrations, $F$ preserves locally cocartesian edges, and the pushforward functors encoded by $p$ preserve all limits indexed by the $\infty$-categories $X_y$. Then the $p$-right Kan extension $G: Y \to C$ of (1) preserves locally cocartesian edges, and we have an adjunction
% \[ \adjunct{\phi^\ast}{\Fun^{\cocart}_{/S}(Y,C)}{\Fun^{\cocart}_{/S}(X,C)}{\phi_\ast}. \]
\end{cor}
\begin{proof} We need to show that for every $y \in Y$, the $p$-limit of $F^y: X \times_Y Y_{y/} \to X \xto{F} C$ exists. By \cref{lem:LimitIsRelativeLimit}, the $p$-limit of $F|_{X_y}$ exists and is computed as the limit of $F|_{X_y}$ viewed as a diagram in $C_{\alpha(y)}$. Because $\phi$ is a cartesian fibration, we have a retraction $r: X \times_Y Y_{y/} \to X_y$ to the inclusion $i: X_y \to X \times_Y Y_{y/}$ such that $r$ is right adjoint to $i$ (on objects, $r$ is given by the formula $r(x,y \xto{e} \phi(x)) = e^\ast(x)$, where $e^\ast: X_{\phi(x)} \to X_y$ is the pullback functor encoded by the lifting property of the cartesian fibration $\phi$). As a left adjoint, $i$ is right cofinal.\footnote{We adopt Lurie's terminology in \cite{HA}: recall that a map $q: K \to L$ is right cofinal if and only if $q^{\op}$ is cofinal.} However, since $r \circ i = \id$, we moreover have that $r$ is right cofinal by the right cancellative property of right cofinal maps \cite[Prop.~4.1.1.3(2)]{HTT}. Hence, by \cite[Prop.~4.3.1.7]{HTT} applied to $r$ and a $p$-limit diagram $(X_y)^{\lhd} \to C$, the $p$-limit of $F^y$ exists and is computed as the limit of $F|_{X_y}$ in $C_{\alpha(y)}$. The claim now follows from \cref{rem:ExistenceOfRelativeRKE}.
\end{proof}

\subsubsection{Barycentric subdivision and locally cocartesian pushforward}

Our main goal in this subsection is to first define the barycentric subdivision $\sd(S)$ (\cref{dfn:barycentricSubdivision}) consisting of conservative functors $\sigma: [n] \to S$ (i.e., \emph{strings} in $S$) along with its \emph{maximum} functor $\max_S: \sd(S) \to S$, $[\sigma:[n] \to S] \mapsto \sigma(n)$, which is a locally cocartesian fibration (\cref{lm:subdivisionLocallyCocartesianByMaxFunctor}). This allows us to define the right-lax limit of a locally cocartesian fibration $p: C \to S$ as
$$\mathrm{lim}^{\rlax} C \coloneq \Fun^{\cocart}_{/S}(\sd(S), C).$$
We will then show that for any sieve $S_0 \subset S$, if we let $\sd(S)_0 \subset \sd(S)$ denote the full subcategory of strings that originate in $S_0$, then the inclusion $\into{\sd(S_0)}{\sd(S)_0}$ is a locally cocartesian equivalence over $S$,\footnote{Here we mark those edges that are locally cocartesian with respect to $\max_{S_0}$ resp. $\max_S$.} or equivalently, for any locally cocartesian fibration $p: C \to S$, the restriction functor
\[ \Fun^{\cocart}_{/S}(\sd(S)_0, C) \to \Fun^{\cocart}_{/S_0}(\sd(S_0), C|_{S_0}) \]
is a trivial fibration (\cref{thm:subdivisionExtension}(2)). A choice of inverse then amounts to a choice of \emph{locally cocartesian pushforward}. This will be the formal half of extending a object in $\mathrm{lim}^{\rlax} C|_{S_0}$ to one in $\mathrm{lim}^{\rlax} C$ itself, which we take up in the next subsection.

To set the stage for our work, we first introduce a few combinatorial constructions. Let $\Delta$ be the category with objects the finite ordinals $\{[n] = \{0<1< ... < n\} : n \in \NN \}$ and morphisms the order-preserving maps. Let $\xi: \cE \Delta \to \Delta$ denote the relative nerve \cite[Def.~3.2.5.2]{HTT} of the canonical inclusion $i:\into{\Delta}{s\Set}$. Then $\xi$ is a cocartesian fibration classified by $i$, which is an explicit model for the tautological cocartesian fibration over $\Delta$. Explicitly, an $n$-simplex $\Delta^n \to \cE \Delta$ is given by a sequence $[a_0] \xto{\alpha_0} [a_1] \xto{\alpha_1} ... \xto{g_{n-1}} [\alpha_n]$ of order-preserving maps in $\Delta$ together with morphisms $\kappa_i: \Delta^{\{0,...,i\}} \cong \Delta^i \to \Delta^{a_i}$ which fit into a commutative diagram
\[ \begin{tikzcd}[row sep=4ex, column sep=4ex, text height=1.5ex, text depth=0.25ex]
\Delta^{\{0\}} \ar[hookrightarrow]{r} \ar{d}{\kappa_0} & \Delta^{\{0,1\}} \ar[hookrightarrow]{r} \ar{d}{\kappa_1} & \cdots \ar[hookrightarrow]{r} & \Delta^{\{0,...,n-1\}} \ar[hookrightarrow]{r} \ar{d}{\kappa_{n-1}} & \Delta^n \ar{d}{\kappa_n} \\
\Delta^{a_0} \ar{r}{\alpha_0} & \Delta^{a_1} \ar{r}{\alpha_1} & \cdots \ar{r} & \Delta^{a_{n-1}} \ar{r}{\alpha_{n-1}} & \Delta^{a_n}.
\end{tikzcd} \]
Let $\cE \Delta^{\inj} \subset \cE \Delta$ denote the pullback over the subcategory $\Delta^{\inj} \subset \Delta$ of injective order-preserving maps and also denote the structure map of $\cE \Delta^{\inj}$ by $\xi$. Consider the span of marked simplicial sets
\[ \begin{tikzcd}[row sep=4ex, column sep=4ex, text height=1.5ex, text depth=0.25ex]
(\Delta^{\inj})^{\sharp} &  \leftnat{(\cE \Delta^{\inj})} \ar{r}{\xi} \ar{l}[swap]{\xi} & (\Delta^{\inj})^{\sharp}
\end{tikzcd} \]
where we mark the $\xi$-cocartesian edges in $\cE \Delta^{\inj}$. Similar to the definition in \cite[Exm.~2.25]{Exp2} (which considers the source input to be instead a cartesian fibration), let $$\widetilde{\Fun}_{\Delta^{\inj}}(\cE \Delta^{\inj},-) \coloneq \xi_{\ast} \xi^{\ast}(-): s\Set^+_{/\Delta^{\inj}} \to s\Set^+_{/\Delta^{\inj}}.$$

Note that with $\xi$ a cocartesian fibration, $\xi_{\ast} \xi^{\ast}$ is right Quillen with respect to the \emph{cartesian} model structure on $s\Set^+_{/ \Delta^{\inj}}$ by the dual of \cite[Thm.~2.24]{Exp2}.

\begin{dfn} The $\infty$-category of \emph{paths}\footnote{For us, a path in $C$ is any $n$-simplex $\Delta^n \to C$. In contrast, we reserve the term `string' for objects of the barycentric subdivision $\sd(C)$ (cf. \cref{dfn:barycentricSubdivision}).} in an $\infty$-category $C$ is
$$\widehat{\Ar}(C) \coloneq \widetilde{\Fun}_{\Delta^{\inj}}(\cE \Delta^{\inj}, C \times \Delta^{\inj}).$$
Let $\xi_C: \widehat{\Ar}(C) \to \Delta^{\inj}$ denote the structure map of the cartesian fibration and note that its fiber over $[n] \in \Delta^{\inj}$ is $\Fun(\Delta^n,C)$ and the functoriality is that of restriction in the source variable.

In addition, let $\widehat{\Ar}^{\simeq}(S) \subset \widehat{\Ar}(S)$ be the maximal sub-right fibration, i.e., the wide subcategory on the $\xi_S$-cartesian edges over $\Delta^{\inj}$ (so the fiber of $\widehat{\Ar}^{\simeq}(S)$ over $[n]$ is $\Map(\Delta^n,S)$), and for a functor $p: C \to S$, let
$$\widehat{\Ar}^{\simeq}_S(C) \coloneq \widehat{\Ar}^{\simeq}(S) \times_{\widehat{\Ar}(S)} \widehat{\Ar}(C).$$
\end{dfn}

\begin{rem} \label{rem:classifyingID}
By \cite[Prop.~7.3]{GHN}, the cartesian fibration $\xi_C: \widehat{\Ar}(C) \to \Delta^{\inj}$ is classified by the functor $(\Delta^{\inj})^{\op} \to \Cat_{\infty}$ that sends $[n]$ to $\Fun(\Delta^n,C)$ and is functorial with respect to precomposition in the first variable. It follows that we have an equivalence
$$\widehat{\Ar}^{\simeq}(C) \simeq \Delta^{\inj} \times_{\Cat_{\infty}} \Cat_{\infty}^{/C}$$
of right fibrations over $\Delta^{\inj}$.
\end{rem}

\begin{rem} If $C \to S$ is a categorical fibration, then $\widehat{\Ar}(C) \to \widehat{\Ar}(S)$ is also a categorical fibration by \cite[Prop.~B.2.7]{HA}.
\end{rem}

\begin{cnstr}[Variant associated to a sieve] \label{cnstr:sieveVariantsPathCategories} Let $\pi: S \to \Delta^1$ be a functor and $S_0$ the fiber over $0$. Let $\widehat{\Ar}(S)_0 \subset \widehat{\Ar}(S)$ be the full subcategory on those objects $\sigma: \Delta^n \to S$ such that $\pi \sigma(0) = 0$ (i.e., \emph{on those paths originating in $S_0$}), and let $\widehat{\Ar}^\simeq(S)_0 \coloneq \widehat{\Ar}(S)_0 \cap \widehat{\Ar}^{\simeq}(S)$. Define the `initial segment' functor $$\lambda_S: \widehat{\Ar}(S)_0 \to \widehat{\Ar}(S_0)$$ by the following rule: 

\begin{itemize}[leftmargin=2em]
\item[($\ast$)] Suppose $\sigma: \Delta^n \to \widehat{\Ar}(S)_0$ is a $n$-simplex, which corresponds to a sequence of inclusions
 \[ \begin{tikzcd}[row sep=4ex, column sep=4ex, text height=1.5ex, text depth=0.25ex]
\Delta^{a_0} \ar[hookrightarrow]{r}{\alpha_1} & \Delta^{a_1} \ar[hookrightarrow]{r}{\alpha_2} & \cdots \ar[hookrightarrow]{r}{\alpha_n} & \Delta^{a_n}
\end{tikzcd} \]
determining a map $a: \Delta^n \to \Delta^{\inj}$ and a functor $f: \Delta^n \times_{a, \Delta^{\inj}} \cE \Delta^{\inj} \to S$ such that for every $0 \leq i \leq n$, the restriction $f_i: \Delta^{a_i} \to S$ has $f_i(0) \in S_0$. Let $b_i \in \Delta^{a_i}$ be the maximum element such that $f_i(b_i) \in S_0$, and note that $a$ restricts to yield a sequence of inclusions  
\[ \begin{tikzcd}[row sep=4ex, column sep=4ex, text height=1.5ex, text depth=0.25ex]
\Delta^{b_0} \ar[hookrightarrow]{r}{\beta_1} \ar[hookrightarrow]{d} & \Delta^{b_1} \ar[hookrightarrow]{r}{\beta_2} \ar[hookrightarrow]{d} & \cdots \ar[hookrightarrow]{r}{\beta_n} & \Delta^{b_n} \ar[hookrightarrow]{d} \\
\Delta^{a_0} \ar[hookrightarrow]{r}{\alpha_1} & \Delta^{a_1} \ar[hookrightarrow]{r}{\alpha_2} & \cdots \ar[hookrightarrow]{r}{\alpha_n} & \Delta^{a_n}
\end{tikzcd} \]
because we always have that $\alpha_i(b_{i-1}) \leq b_i$ as $S_0$ is a sieve in $S$ stable under equivalences. Let $b: \Delta^n \to \Delta^{\inj}$ be the map determined by the sequence of upper horizontal inclusions. $f$ then restricts to yield a map $f_0$:
\[ \begin{tikzcd}[row sep=4ex, column sep=4ex, text height=1.5ex, text depth=0.25ex]
\Delta^n \times_{b, \Delta^{\inj}} \cE \Delta^{\inj} \ar{r}{f_0} \ar[hookrightarrow]{d} & C_0 \ar[hookrightarrow]{d} \\
\Delta^n \times_{a, \Delta^{\inj}} \cE \Delta^{\inj} \ar{r}{f} & C.
\end{tikzcd} \]
Define $\lambda_S(\sigma): \Delta^n \to \widehat{\Ar}(S_0)$ to be the $n$-simplex determined by $f_0$. Now observe that this assignment is natural in $\Delta^n$, hence defines a map of simplicial sets.
\end{itemize}

Observe that $\lambda_S$ is a retraction of the inclusion $\widehat{\Ar}(S_0) \to \widehat{\Ar}(S)_0$ induced by $S_0 \to S$.

An edge $e: \Delta^1 \to \widehat{\Ar}(S)_0$ is $\xi_S$-cartesian if and only if the corresponding functor $f: \Delta^1 \times_{a, \Delta^{\inj}} \cE \Delta^{\inj} \to S$ sends every edge $(i \in [a_0]) \to (\alpha_1(i) \in [a_1])$ to an equivalence, and similarly for $\xi_{S_0}$-cartesian edges in $\widehat{\Ar}(S_0)$. Therefore, $\lambda_S$ preserves cartesian edges and restricts to a map $$\lambda_S: \widehat{\Ar}^{\simeq}(S)_0 \to \widehat{\Ar}^{\simeq}(S_0).$$
\end{cnstr}

\begin{cnstr}[Variant associated to a sieve, relative version] \label{cnstr:sieveVariantPathCategoriesTwo}
Let $p: C \to S$ be a locally cocartesian fibration and let $p_0: C_0 \to S_0$ be its fiber over $0$. Note that $\widehat{\Ar}(C)_0 \cong \widehat{\Ar}(S)_0 \times_{\widehat{\Ar}(S)} \widehat{\Ar}(C)$. Let $$\widehat{\Ar}^{\simeq}_S(C)_0 \coloneq \widehat{\Ar}^{\simeq}(S)_0 \times_{\widehat{\Ar}(S)_0} \widehat{\Ar}(C)_0 \cong  \widehat{\Ar}^{\simeq}(S)_0 \times_{\widehat{\Ar}^{\simeq}(S)} \widehat{\Ar}^{\simeq}_S(C),$$
so $\widehat{\Ar}^{\simeq}_S(C)_0 \subset \widehat{\Ar}^{\simeq}_S(C)$ is the full subcategory on objects $c: \Delta^n \to C$ with $c(0) \in C_0$.  The initial segment functor $\lambda_{(-)}$ fits into a commutative diagram
\[ \begin{tikzcd}[row sep=4ex, column sep=4ex, text height=1.5ex, text depth=0.25ex]
\widehat{\Ar}^{\simeq}(S)_0 \ar[hookrightarrow]{r} \ar{d}{\lambda_S} & \widehat{\Ar}(S)_0 \ar{d}{\lambda_S} & \widehat{\Ar}(C)_0 \ar{l}[swap]{p} \ar{d}{\lambda_C} \\
\widehat{\Ar}^{\simeq}(S_0) \ar[hookrightarrow]{r} & \widehat{\Ar}(S_0) & \widehat{\Ar}(C_0) \ar{l}[swap]{p_0}
\end{tikzcd} \]

and therefore defines a functor $\lambda_{p}: \widehat{\Ar}_S^{\simeq}(C)_0 \to \widehat{\Ar}_{S_0}^{\simeq}(C_0)$.

Finally, let $\widehat{\Ar}^{\simeq}_S(C)_0^{\cocart} \subset \widehat{\Ar}^{\simeq}_S(C)_0$ be the full subcategory on those objects $c: \Delta^n \to C$ such that if $i \in \Delta^n$ is the maximum element with $c(i) \in C_0$, then $c$ sends every edge $\{j < j+1 \}$, $j \geq i$ to a locally-$p$ cocartesian edge (i.e., a cocartesian edge over $\Delta^1$ in the pullback $\Delta^1 \times_S C$).
\end{cnstr}

The next theorem implies that we can construct a \emph{locally cocartesian pushforward} extending from $C_0$ to $C$ along paths in the base $S$ that originate in $S_0$. This will amount to a section of the trivial fibration considered therein.

\begin{thm} \label{prp:LocallyCocartesianPushforward} The map $(\lambda_{p}, p): \widehat{\Ar}^{\simeq}_S(C)_0^{\cocart} \to \widehat{\Ar}^{\simeq}_{S_0}(C_0) \times_{p_0, \widehat{\Ar}^{\simeq}(S_0), \lambda_S} \widehat{\Ar}^{\simeq}(S)_0$ is a trivial fibration of simplicial sets.
\end{thm}
\begin{proof} We need to solve the lifting problem
\[ \begin{tikzcd}[row sep=4ex, column sep=4ex, text height=1.5ex, text depth=0.25ex]
\partial \Delta^n \ar[hookrightarrow]{d} \ar{r} & \widehat{\Ar}^{\simeq}_S(C)_0^{\cocart} \ar{d}{(\lambda_p,p)} \\
\Delta^n \ar{r} \ar[dotted]{ru} &  \widehat{\Ar}^{\simeq}_{S_0}(C_0) \times_{\widehat{\Ar}^{\simeq}(S_0)} \widehat{\Ar}^{\simeq}(S)_0.
\end{tikzcd} \]
Let $a: \Delta^n \to \widehat{\Ar}^{\simeq}(S)_0 \to \Delta^{\inj}$ and $b: \Delta^n \to \widehat{\Ar}^{\simeq}_{S_0}(C_0) \to \Delta^{\inj}$ be as discussed in the definition of $\lambda$. This lifting problem transposes to
\[ \begin{tikzcd}[row sep=4ex, column sep=4ex, text height=1.5ex, text depth=1ex]
\Delta^n \times_{b, \Delta^{\inj}} \cE \Delta^{\inj} \bigcup_{\partial \Delta^n \times_{b, \Delta^{\inj}} \cE \Delta^{\inj}} \partial \Delta^n \times_{a, \Delta^{\inj}} \cE \Delta^{\inj} \ar{r} \ar[hookrightarrow]{d}{f} & C \ar{d}{p} \\
\Delta^n \times_{a, \Delta^{\inj}} \cE \Delta^{\inj} \ar{r} \ar[dotted]{ru} & S.
\end{tikzcd} \]
Consider $\Delta^n \times_{a, \Delta^{\inj}} \cE \Delta^{\inj}$ as a marked simplicial set where an edge $(i \in \Delta^{a_k}) \to (j \in \Delta^{a_l})$, $\alpha: \Delta^{a_k} \to \Delta^{a_l}$, $\alpha(i) \leq j$ is marked if and only if $k = l$ (so $\alpha = \id$), $b_k \leq i$ and $j = i+1$, and let the domain of $f$ also inherit this marking. Then it suffices to show that $f$ is a trivial cofibration in the locally cocartesian model structure on $s\Set^+_{/S}$, defined by the categorical pattern $\mathfrak{P} = (M_S,T,\emptyset)$ with $M_S$ all of the edges in $S$ and $T$ consisting of the $2$-simplices $\tau$ in $S$ with the edge $\tau(\{1<2\})$ an equivalence. Proceeding by induction on $n$, by a two-out-of-three argument it suffices to show that the inclusion $f': \Delta^n \times_{b, \Delta^{\inj}} \cE \Delta^{\inj} \to \Delta^n \times_{a, \Delta^{\inj}} \cE \Delta^{\inj}$ is a trivial cofibration. We define a filtration of the poset inclusion $f'$ as follows:
\begin{itemize}
    \item[($\ast$)] Let $a_n - b_n = t$. For $0 \leq k \leq n$, let $\alpha_k: \Delta^{a_k} \to \Delta^{a_n}$ denote the inclusion. Let $P_r \subset \Delta^n \times_{a, \Delta^{\inj}} \cE \Delta^{\inj}$ be the subposet on those objects $(i \in \Delta^{a_k})$ such that $\alpha_k(i)-b_n \leq r$. Note that $P_0 = \Delta^n \times_{b, \Delta^{\inj}} \cE \Delta^{\inj}$, because if $(i \in \Delta^{a_k})$ is such that $i > b_k$, then necessarily $\alpha_k(i) > b_n$, and likewise if $i \leq b_k$, then $\alpha_k(i) \leq b_n$ (this follows from the definitions of the $b_i$ and that $S_0$ is a sieve stable under equivalences). Then we have that $f'$ factors as a sequence of poset sieve inclusions $\Delta^n \times_{b, \Delta^{\inj}} \cE \Delta^{\inj} = P_0 \subset P_1 \subset \cdots \subset P_t = \Delta^n \times_{a, \Delta^{\inj}} \cE \Delta^{\inj}$.
\end{itemize}
It now suffices to show that $P_i \subset P_{i+1}$ is a trivial cofibration for all $0 \leq i < t$. For simplicity, let us suppose $i=0$ (and $t>0$ for non-triviality), the other cases being proved similarly. Let $k \in [n]$ be the smallest element such that $b_n + 1 \in \Delta^{a_n}$ is in the image of $\alpha_k: \Delta^{a_k} \to \Delta^{a_n}$. Note then that for all $k \leq l \leq n$, $\alpha_l(b_l+1) = b_n+1$. View the poset $\Delta^{\{k,...,n\}} \times \Delta^1$ as a cosieve $U$ in $P_1$ via the inclusion which sends $(l,0)$ to $(b_l \in \Delta^{a_l})$ and $(l,1)$ to $(b_l + 1 \in \Delta^{a_l})$. Then as a marked simplicial set, we have $U = (\Delta^{\{k,...,n\}})^{\flat} \times (\Delta^1)^{\sharp}$. By \cite[B.1.10]{HA}, the inclusion
\[ U \cap P_0 = (\Delta^{\{k,...,n\}})^{\flat} \times \{0\} \to U = (\Delta^{\{k,...,n\}})^{\flat} \times (\Delta^1)^{\sharp} \]
is $\mathfrak{P}$-anodyne. Noting that $P_0$ and $U$ together cover $P_1$, it thus suffices to show that we have a homotopy pushout square of $\infty$-categories
\[ \begin{tikzcd}[row sep=4ex, column sep=4ex, text height=1.5ex, text depth=0.25ex]
U \cap P_0 \ar{r} \ar{d} & U \ar{d} \\
P_0 \ar{r} & P_1
\end{tikzcd} \]
as we would then deduce the lower horizontal map to be $\mathfrak{P}$-anodyne. For this, the criterion of \cref{lem:posetPushoutViaFlatness} is easily verified.
\end{proof}

\begin{lem} \label{lem:posetPushoutViaFlatness} Suppose $P$ is a poset, $Z \subset P$ is a sieve and $U \subset P$ is a cosieve such that $P = Z \cup U$. Then the commutative square
\[ \begin{tikzcd}[row sep=4ex, column sep=4ex, text height=1.5ex, text depth=0.25ex]
U \cap Z \ar{r} \ar{d} & U \ar{d} \\
Z \ar{r} & P
\end{tikzcd} \]
% with the map $U \cup_{U \cap Z} Z \to P$ inner anodyne
is a homotopy pushout square of $\infty$-categories if and only if for every $a \notin U$ and $c \notin Z$ such that $a \leq c$, the subposet $P_{a//c} = \{ b \in U \cap Z : a \leq b \leq c\}$ is weakly contractible.
\end{lem}
\begin{proof}  Define a map $\pi: P \to \Delta^2$ by
\begin{equation*} \pi(x) = \begin{cases} 0 & x \notin U \\
2 & x \notin Z \\
1 & x \in U \cap Z 
\end{cases}
\end{equation*}
Observe that $P \times_{\Delta^2} \Delta^{\{0,1\}} = Z$,  $P \times_{\Delta^2} \Delta^{\{1,2\}} = U$, and $P \times_{\Delta^2} \{1\} = U \cap Z$. We may therefore apply the flatness criterion of \cite[B.3.2]{HA} to $\pi$ in order to deduce the criterion in question.
\end{proof}

We now introduce the barycentric subdivision $\sd(S)$.

\begin{dfn} \label{dfn:barycentricSubdivision} An $n$-simplex $\sigma: \Delta^n \to S$ is a \emph{string} if $\sigma$ is a conservative functor, i.e., if for every $0 \leq i< j \leq n$, $\sigma(\{i<j\})$ is not an equivalence.\footnote{If every retract in $S$ is an equivalence, then it suffices to check that for every $0 \leq i < n$, $\sigma(\{i<i+1\})$ is not an equivalence.} The \emph{barycentric subdivision} (or \emph{subdivision})
\[ \sd(S) \subset \widehat{\Ar}^{\simeq}(S) \]
is the full subcategory of $\widehat{\Ar}^{\simeq}(S)$ on the strings in $S$. Note that the structure map $\xi_S: \widehat{\Ar}^{\simeq}(S) \to \Delta^{\inj}$ restricts to define a right fibration $\xi_S: \sd(S) \to \Delta^{\inj}$.

Given a functor $C \to S$, the \emph{$S$-relative subdivision} $\sd_S(C)$ is the pullback $$\sd_S(C) \coloneq \sd(S) \times_{\widehat{\Ar}^{\simeq}(S)} \widehat{\Ar}^{\simeq}_S(C) \cong \sd(S) \times_{\widehat{\Ar}(S)} \widehat{\Ar}(C).$$ Similarly, parallel to \cref{cnstr:sieveVariantsPathCategories} and \ref{cnstr:sieveVariantPathCategoriesTwo} we may define $\sd(S)_0$, $\sd_S(C)_0$, and $\sd_S(C)_0^{\cocart}$ for a locally cocartesian fibration $C \to S$ and a functor $S \to \Delta^1$. To be specific, let $\sd(S)_0 \subset \sd(S)$ be the full subcategory on those strings originating in the sieve $S_0$, let $\sd_S(C)_0 \coloneq \sd(S)_0 \times_{\sd(S)} \sd_S(C)$, and let $\sd_S(C)_0^{\cocart} \coloneq \sd_S(C)_0 \times_{\widehat{\Ar}_S^{\simeq}(C)_0} \widehat{\Ar}_S^{\simeq}(C)^{\cocart}_0$.
\end{dfn}

\begin{obs} \label{rem:barycentricSubdivisionOrdinaryCategory} Suppose that $S$ is the nerve of a category, which we also denote as $S$. Then $\sd(S)$ is the nerve of the category whose objects are conservative functors $\sigma: \Delta^n \to S$, and where a morphism $[\sigma: \Delta^n \to S] \to [\tau: \Delta^m \to S]$ is given by the data of a map $\alpha: \into{[n]}{[m]}$ in $\Delta^{\inj}$ and a natural transformation $\sigma \Rightarrow \alpha^{\ast} \tau$ through equivalences. In particular, if $S$ is the nerve of a poset $P$, then $\sd(P)$ is the nerve of the usual barycentric subdivision of $P$.

On the other hand, the usual definition of the subdivision of an $\infty$-category \cite[Def.~1.15]{AMGR-NaiveApproach} is as the left Kan extension of the functor $\sd: \Delta \to \Cat_{\infty}$ along the fully faithful inclusion $\Delta \subset \Cat_{\infty}$. By \cite[Lem.~A.3.7]{AMGRb}, this recovers $\sd(P)$ for $P$ a poset. In fact, we may transcribe over the proof there to show that $\sd(S) \xot{\simeq} \colim_{[n] \in \Delta_{/S}} \sd[n]$ for any $\infty$-category $S$. Here $\Delta_{/S} \coloneq \Delta \times_{\Cat_{\infty}} (\Cat_{\infty})^{/S}$ is the maximal sub-right fibration in $\widetilde{\Fun}_{\Delta}(\cE \Delta, S \times \Delta)$ (cf. \cref{rem:classifyingID}).\footnote{Beware that here $\Delta_{/S}$ does \emph{not} denote the nerve of the category of simplices of $S$ regarded as a simplicial set.} We sketch the argument, leaving routine details to the reader:
\begin{enumerate}
\item First note that for any two strings $\sigma, \tau \in \sd(S)$, every map $[\sigma \Rightarrow \tau] \in \Delta_{/S}$ necessarily lies over $\Delta^{\inj}$. Therefore, the inclusion $i: \sd(S) \subset \Delta_{/S}$ is full. Moreover, in view of the factorization system on $\Cat_{\infty}$ whose right class of maps is given by the conservative functors \cite[11.29]{Joyal}, $i$ admits a left adjoint. In particular, $i$ is cofinal, so
\[ \colim_{[n] \in \Delta_{/S}} \sd[n] \simeq \colim_{[n] \in \sd(S)} \sd[n]. \]
\item We next observe that the cocartesian fibration $\ev_1: \Ar(\sd(S)) \to \sd(S)$ is classified by the functor $\sd(S) \to \Delta^{\inj} \subset \Delta \xto{\sd} \Cat_{\infty}$. Therefore, $\colim_{[n] \in \sd(S)} \sd[n]$ identifies with the localization of $\Ar(\sd(S))$ at the class of $\ev_1$-cocartesian edges. But this localization also identifies with the source functor $\ev_0: \Ar(\sd(S)) \to \sd(S)$, yielding the desired equivalence $\colim_{[n] \in \sd(S)} \sd[n] \to \sd(S)$.
\end{enumerate}
\end{obs}

We now work towards constructing the `maximum' functor $\sd(S) \to S$. We first define this over $\widehat{\Ar}(S)$:

\begin{cnstr} \label{cnstr:MaxFunctorSubdivision} Define a \emph{last vertex} map $\max_S: \widehat{\Ar}(S) \to S$ by the following rule:
\begin{itemize} \item[($\ast$)] Suppose $\sigma: \Delta^n \to \widehat{\Ar}(S)$ is a $n$-simplex, which corresponds to a sequence of inclusions \[ \begin{tikzcd}[row sep=4ex, column sep=4ex, text height=1.5ex, text depth=0.25ex]
\Delta^{a_0} \ar[hookrightarrow]{r}{\alpha_1} & \Delta^{a_1} \ar[hookrightarrow]{r}{\alpha_2} & \cdots \ar[hookrightarrow]{r}{\alpha_n} & \Delta^{a_n}
\end{tikzcd} \]
determining a map $a: \Delta^n \to \Delta^{\inj}$ and a functor $f: \Delta^n \times_{a, \Delta^{\inj}} \cE \Delta^{\inj} \to S$. Define a functor $\chi: \Delta^n \to \Delta^n \times_{a, \Delta^{\inj}} \cE \Delta^{\inj}$ to be the identity on the first component and the unique $n$-simplex of $\cE \Delta^{\inj}$
\[ \begin{tikzcd}[row sep=4ex, column sep=4ex, text height=1.5ex, text depth=0.25ex]
\Delta^{\{0\}} \ar[hookrightarrow]{r} \ar{d}{\kappa_0} & \Delta^{\{0,1\}} \ar[hookrightarrow]{r} \ar{d}{\kappa_1} & \cdots \ar[hookrightarrow]{r} & \Delta^n \ar{d}{\kappa_n} \\
\Delta^{a_0} \ar[hookrightarrow]{r}{\alpha_1} & \Delta^{a_1} \ar[hookrightarrow]{r}{\alpha_2} & \cdots \ar[hookrightarrow]{r}{\alpha_n} & \Delta^{a_n}.
\end{tikzcd} \]
specified by $\kappa_i(i) = a_i$ on the second component. Then $\max_S(\sigma) = f \circ \chi: \Delta^n \to S$.
\end{itemize}

In other words, $\max_S$ is the functor induced by precomposing by the section $\Delta^{\inj} \to \cE \Delta^{\inj}$ which selects the maximal vertex in every fiber.
\end{cnstr}

The next lemma is obvious when $S$ is a poset, so the reader only interested in that case should feel free to skip its proof.

\begin{lem} \label{lm:subdivisionLocallyCocartesianByMaxFunctor} \begin{enumerate}[leftmargin=*] \item The functor $\max_S: \widehat{\Ar}(S) \to S$ is a categorical fibration.
\end{enumerate}
\begin{enumerate}
\setcounter{enumi}{1}
\item The restricted functor $\max_S: \widehat{\Ar}^{\simeq}(S) \to S$ is a locally cocartesian fibration.
\item The restricted functor $\max_S: \sd(S) \to S$ is a locally cocartesian fibration.
%true but don't need this right now
% \item For any locally cocartesian fibration $p: C \to S$, the $S$-relative last vertex functor
% \[ \max_{p}: \sd_S(C) \to \sd(S) \xto{\max_S} S \]
% is a locally cocartesian fibration.
\end{enumerate}
\end{lem}
\begin{proof} (1): We first verify that $\max_S$ is an inner fibration. For this, let $n \geq 2$, $0 < k < n$, and consider the lifting problem
\[ \begin{tikzcd}[row sep=4ex, column sep=4ex, text height=1.5ex, text depth=0.25ex]
\Lambda^n_k \ar{r} \ar{d} & \widehat{\Ar}(S) \ar{d}{\max_S} \\
\Delta^n \ar{r} \ar[dotted]{ru} & S. 
\end{tikzcd} \]
Let $a: \Delta^n \to \Delta^{\inj}$ be the unique extension of the given $\Lambda^n_k \to \Delta^{\inj}$. The lifting problem then transposes to
\[ \begin{tikzcd}[row sep=4ex, column sep=4ex, text height=1.5ex, text depth=0.25ex]
\Delta^n \bigcup_{\Lambda^n_k} \Lambda^n_k \times_{\Delta^{\inj}} \cE \Delta^{\inj} \ar{r} \ar{d} & S \\
\Delta^n \times_{\Delta^{\inj}} \cE \Delta^{\inj} \ar[dotted]{ru}
\end{tikzcd} \]
and it suffices to show the vertical arrow is inner anodyne. Since $\cE \Delta^{\inj} \to \Delta^{\inj}$ is a cocartesian fibration, it is in particular a flat inner fibration, and the desired result follows. 

We next show that $\max_S$ is a categorical fibration by lifting equivalences from the base. So suppose $e: \Delta^1 \to S$ is an equivalence and $\sigma: \Delta^n \to S$ is an object of $\widehat{\Ar}(S)$ such that $\max_S(\sigma) = \sigma(n) = e(0)$. The restriction of $\max_S$ to $\Fun(\Delta^n,S) \subset \widehat{\Ar}(S)$ is evaluation at $\{n\}$, which is a categorical fibration, so $e$ lifts to an equivalence in $\Fun(\Delta^n,S)$ and hence in $\widehat{\Ar}(S)$.

(2): First observe that since $\widehat{\Ar}^{\simeq}(S) \subset \widehat{\Ar}(S)$ is a subcategory stable under equivalences, the restricted $\max_S$ functor is a categorical fibration by (1). To prove that $\max_S$ is a locally cocartesian fibration, it then suffices to prove that for any edge $e: s \to t$ in $S$ that is \emph{not} an equivalence, the pullback $\max_S(e): \widehat{\Ar}^{\simeq}(S) \times_{S} \Delta^1 \to \Delta^1$ is a cocartesian fibration. To this end, we claim that an edge $\widetilde{e}: x \to y$ lifting $e$ is $\max_S(e)$-cocartesian if and only if the corresponding data of an inclusion $\alpha: \Delta^{a_0} \to \Delta^{a_1}$ and a functor $f: \Delta^1 \times_{\Delta^{\inj}} \cE \Delta^{\inj} \to S$ is such that in addition $a_1 = a_0 +1$ and $\alpha$ is the inclusion of the initial segment. Note that given an object $x: \Delta^{a_0} \to S$ with $s= x(a_0)$, such a lift $\widetilde{e}$ of $e$ may be defined by `appending' $e$ to $x$: indeed, let $y: \Delta^{a_0+1} \to S$ be an extension of $x \cup e: \Delta^{a_0} \cup_{a_0,\Delta^0,0} \Delta^1 \to S$, let
\[ r: \Delta^1 \times_{\alpha, \Delta^{\inj}} \cE \Delta^{\inj} \to \Delta^{a_{0}+1} \]
be the retraction functor which fixes $\Delta^{a_{0} + 1}$ and is given by $\alpha$ on $\Delta^{a_0}$, and define $\widetilde{e}$ as $y \circ r$. Hence, establishing the claim will complete the proof.
% Note that such a lift $\widetilde{e}$ always exists by appending the edge $e$ to the object $x: \Delta^{a_0} \to S$ with $x(a_0) = s$ (we suppress a routine lifting argument here).

The `only if' direction will follow from the `if' direction together with the stability of cocartesian edges under equivalence. For the `if' direction, fix such an edge $\widetilde{e}$. Recall from the definition that $\widetilde{e}: x \to y$ is $\max_S(e)$-cocartesian if and only if for all objects $z \in \widehat{\Ar}^{\simeq}(S)$ with $\max_S(z) = t$, the commutative square
\[ \begin{tikzcd}[row sep=4ex, column sep=4ex, text height=1.5ex, text depth=0.5ex]
\Map_{\widehat{\Ar}^{\simeq}(S)_{\max_S = t}}(y,z) \ar{r}{(\widetilde{e})^\ast} \ar{d} & \Map_{\widehat{\Ar}^{\simeq}(S)}(x,z) \ar{d}{\max_S} \\
\{ e \} \ar{r} & \Map_S(s,t)
\end{tikzcd} \]
is a homotopy pullback square.  Viewing $x$ as $x: \Delta^{a_0} \to S$, $y$ as $y: \Delta^{a_0+1} \to S$, and $z$ as $z: \Delta^{a_2} \to S$, and computing the mapping spaces in $\widehat{\Ar}^{\simeq}(S)$ as a cartesian fibration over $\Delta^{\inj}$, we see that
\[ \Map_{\widehat{\Ar}^{\simeq}(S)}(x,z) \simeq \bigsqcup_{\gamma: [a_0] \subset [a_2]} \Map_{\Map(\Delta^{a_0},S)}(x,\gamma^{\ast}z). \]
Therefore, it suffices to show that for any \emph{fixed} inclusion $\gamma: \into{ \Delta^{a_0}} {\Delta^{a_2}}$ with $\gamma(a_0) < a_2$, letting $\beta: \Delta^{a_0+1} \to \Delta^{a_2}$ be the unique extension of $\gamma$ with $\beta(a_0+1) = a_2$, we have that the square of mapping spaces
\[ \begin{tikzcd}[row sep=4ex, column sep=4ex, text height=1.5ex, text depth=0.5ex]
\Map_{\Map(\Delta^{a_{0}+1},S)}(y,\beta^{\ast}z) \ar{r}{\alpha^{\ast}} \ar{d} & \Map_{\Map(\Delta^{a_{0}},S)}(x,\gamma^{\ast}z) \ar{d} \\
\{e\} \ar{r} & \Map_{\iota S}(x(a_0),z(a_2))
\end{tikzcd} \]
is a homotopy pullback square (where the right vertical map sends $x \rightarrow \gamma^{\ast}z$ to the composite $x(a_0) \rightarrow z(\gamma(a_0)) \rightarrow z(a_2)$). (Here we implicitly use that maps in $\widehat{\Ar}^{\simeq}(S)$ are natural transformations through equivalences to account for the $\max_S = t$ condition for the upper-left mapping space.) But this follows since $\ev_{a_0+1}: \Fun(\Delta^{a_0+1},S) \to S$ is a cocartesian fibration with $\overline{x} \to y$ a cocartesian edge lifting $e$, where $\overline{x}$ is the degeneracy $s_{a_0}$ applied to $x$ (we note that $\Map_{\Map(\Delta^{a_{0}+1},S)}(\overline{x},\beta^{\ast}z) \simeq \Map_{\Map(\Delta^{a_{0}},S)}(x,\gamma^{\ast}z)$).

(3): This is clear from the description of the locally $\max_S$-cocartesian edges given in (2).
\end{proof}

Finally, we arrive at the main result of this subsection. \cref{lm:subdivisionLocallyCocartesianByMaxFunctor} ensures that the following theorem is well-formulated; also note that $\sd(S)_0 \subset \sd(S)$ is a sub-locally cocartesian fibration via $\max_S$ as it is the inclusion of a cosieve stable under equivalences.

\begin{thm} \label{thm:subdivisionExtension} Let $p: C \to S$ be a locally cocartesian fibration and $\pi: S \to \Delta^1$ a functor. Let $p_0: C_0 \to S_0$ be the fiber of $p$ over $0$. 
\begin{enumerate}
    \item Restricting the domain and codomain of the map of \cref{prp:LocallyCocartesianPushforward} yields the map
    \[ \sd_S(C)_0^{\cocart} \to \sd_{S_0}(C_0) \times_{\sd(S_0)} \sd(S)_0 \]
     which is also a trivial fibration of simplicial sets.
    \item Precomposition by the inclusion $\into{\sd(S_0)}{\sd(S)_0}$ defines a trivial fibration of simplicial sets
\[ \Fun^{\cocart}_{/S}(\sd(S)_0,C) \to \Fun^{\cocart}_{/S_0}(\sd(S_0),C_0). \]
\end{enumerate}
\end{thm}

For the proof, it will be convenient to introduce an auxiliary construction. Define a functor
$$\delta: \widehat{\Ar}(S) \to \widehat{\Ar}(\widehat{\Ar}(S))$$
by the following rule:
\begin{itemize}[leftmargin=*]
\item[($\ast$)] Suppose $\sigma: \Delta^n \to \widehat{\Ar}(S)$ is a $n$-simplex, which corresponds to a sequence of inclusions
 \[ \begin{tikzcd}[row sep=4ex, column sep=4ex, text height=1.5ex, text depth=0.25ex]
\Delta^{a_0} \ar[hookrightarrow]{r}{\alpha_1} & \Delta^{a_1} \ar[hookrightarrow]{r}{\alpha_2} & \cdots \ar[hookrightarrow]{r}{\alpha_n} & \Delta^{a_n}
\end{tikzcd} \]
determining a map $a: \Delta^n \to \Delta^{\inj}$ and a functor $f: \Delta^n \times_{a, \Delta^{\inj}} \cE \Delta^{\inj} \to S$. Define a map
$$\overline{a}: \Delta^n \times_{a, \Delta^{\inj}} \cE \Delta^{\inj} \to \Delta^{\inj}$$
on objects by $\overline{a}(i \in \Delta^{a_k}) = \Delta^{\{0,...,i\}}$ and on morphisms $(i \in \Delta^{a_k}) \to (j \in \Delta^{a_l})$, $\alpha_{k l}: \Delta^{a_k} \to \Delta^{a_l}$, $\alpha_{k l}(i) \leq j$ by restriction of $\alpha_{k l}$ to $\Delta^{\{0,...,i\}} \subset \Delta^{a_k}$ (which then is valued in $\Delta^{\{0,...,j\}} \subset \Delta^{a_l}$). Then define a functor of categories
\[ \phi: (\Delta^n \times_{a, \Delta^{\inj}} \cE \Delta^{\inj}) \times_{\overline{a}, \Delta^{\inj}} \cE \Delta^{\inj} \to \Delta^n \times_{a, \Delta^{\inj}} \cE \Delta^{\inj} \]
by sending objects $(i \in \Delta^{a_k}, i' \leq i)$ to $(i' \in \Delta^{a_k})$ and morphisms $(i \in \Delta^{a_k}, i' \leq i) \to (j \in \Delta^{a_l}, j' \leq j)$ (specified by the data of a map $\alpha_{k l}: \Delta^{a_k} \to \Delta^{a_l}$ such that $\alpha_{k l}(i) \leq j$ and $\alpha_{k l}(i') \leq j'$) to the morphism $(i' \in \Delta^{a_k}) \to (j' \in \Delta^{a_l})$ specified by the same data.

We may then specify a map
$$g: \Delta^n \times_{a, \Delta^{\inj}} \cE \Delta^{\inj} \to \widehat{\Ar}(S)$$
defined over $\Delta^{\inj}$ via $\overline{a}$ and the structure map $\xi_S$ as adjoint to the map
$$f \circ \phi: (\Delta^n \times_{a, \Delta^{\inj}} \cE \Delta^{\inj}) \times_{\overline{a}, \Delta^{\inj}} \cE \Delta^{\inj} \to S.$$
$g$ in turn defines the desired $n$-simplex $\delta(\sigma): \Delta^n \to \widehat{\Ar}(\widehat{\Ar}(S))$.
\end{itemize}
Informally, $\delta$ sends paths $s_0 \to s_1 \to ... \to s_n$ to their `initial segment parametrization'
\[ [s_0] \to [s_0 \to s_1] \to ... \to [s_0 \to s_1 \to ... \to s_n]. \]
Next, using the functor $\max_S$ to make sense of the next statement, we may use $\delta$ to define functors
\begin{align*}
 \delta &: \widehat{\Ar}^{\simeq}(S) \to \widehat{\Ar}^{\simeq}_S(\widehat{\Ar}^{\simeq}(S)) = \widehat{\Ar}^{\simeq}(S) \times_{\widehat{\Ar}(S)} \widehat{\Ar}(\widehat{\Ar}^{\simeq}(S)) \\
 \delta &: \sd(S) \to \sd_S(\sd(S)) = \sd(S) \times_{\widehat{\Ar}(S)} \widehat{\Ar}(\sd(S))
\end{align*}
as the identity on the first factor and a restriction of $\delta$ on the second factor.

\begin{proof}[Proof of \cref{thm:subdivisionExtension}] (1) follows from \cref{prp:LocallyCocartesianPushforward} in view of the pullback square
\[ \begin{tikzcd}[row sep=4ex, column sep=4ex, text height=1.5ex, text depth=0.25ex]
\sd_S(C)^{\cocart}_0 \ar{r} \ar{d} & \widehat{\Ar}^{\simeq}_S(C)_0^{\cocart} \ar{d} \\
\sd_{S_0}(C_0) \times_{\sd(S_0)} \sd(S)_0 \ar{r} & \widehat{\Ar}^{\simeq}_{S_0} (C_0) \times_{\widehat{\Ar}^{\simeq}(S_0)} \widehat{\Ar}^{\simeq}(S)_0.
\end{tikzcd} \]
For (2), we need to solve the lifting problem
\[ \begin{tikzcd}[row sep=4ex, column sep=4ex, text height=1.5ex, text depth=0.25ex]
A \ar[hookrightarrow]{d} \ar{r} & \Fun^{\cocart}_{/S}(\sd(S)_0,C) \ar{d} \\
B \ar{r} \ar[dotted]{ru} & \Fun^{\cocart}_{/S_0}(\sd(S_0),C_0).
\end{tikzcd} \]
This transposes to
\[ \begin{tikzcd}[row sep=4ex, column sep=4ex, text height=1.5ex, text depth=0.25ex]
A \times \sd(S)_0 \bigcup_{A \times \sd(S_0)} B \times \sd(S_0) \ar{r}{G \cup F} \ar[hookrightarrow]{d} & C \ar{d}{p} \\
B \times \sd(S)_0 \ar{r}{\max_S} \ar[dotted]{ru} & S.
\end{tikzcd} \]
The functoriality of $\sd_{S_0}(-)$ in its argument results in a functor
\[ \sd_{S_0}: \Fun_{/S_0}(\sd(S_0),C_0) \to \Fun_{/S_0}(\sd_{S_0}(\sd(S_0)),\sd_{S_0}(C_0)).\]
Given $F: B \times \sd(S_0) \to C_0$, let $\sd_{S_0}(F): B \times \sd_{S_0}(\sd(S_0)) \to \sd_{S_0}(C_0)$ denote the image. We then define $\overline{F}$ as the composite
\[ \begin{tikzcd}[row sep=4ex, column sep=6ex, text height=1.5ex, text depth=0.5ex]
B \times \sd(S_0) \ar{r}{\id \times \delta} & B \times \sd_{S_0}(\sd(S_0)) \ar{r}{\sd_{S_0}(F)} & \sd_{S_0}(C_0).
\end{tikzcd} \]
Also let $\overline{F}'$ denote $\overline{F}$ with codomain $\sd_S(C)^{\cocart}_0$ via the inclusion $\sd_{S_0}(C_0) \subset \sd_S(C)^{\cocart}_0$.

Similarly, given $G: A \times \sd(S)_0 \to C$, we may define $\overline{G}$ as the composite
\[ \begin{tikzcd}[row sep=4ex, column sep=6ex, text height=1.5ex, text depth=0.5ex]
A \times \sd(S)_0 \ar{r}{\id \times \delta} & A \times \sd_S(\sd(S)_0) \ar{r}{\sd_S(G)} & \sd_S(C)^{\cocart}_0
\end{tikzcd} \]
where we note that the codomain of $\sd_S(G)$ necessarily lies in $\sd_S(C)^{\cocart}_0$ by definition of the locally $\max_S$-cocartesian edges in $\sd(S)_0$ (here it is essential that we use $\sd(S)$ rather than $\widehat{\Ar}^{\simeq}(S)$). Clearly, $\overline{G}$ and $\overline{F}'$ are compatible on their common domain $A \times \sd(S_0)$ since $G$ and $F$ are. We thereby may factor the square above as
\[ \begin{tikzcd}[row sep=4ex, column sep=4ex, text height=1.5ex, text depth=0.25ex]
A \times \sd(S)_0 \bigcup_{A \times \sd(S_0)} B \times \sd(S_0) \ar{r}{\overline{G} \cup \overline{F}'} \ar[hookrightarrow]{d} & \sd_S(C)^{\cocart}_0 \ar{r}{\max_C} \ar[->>]{d}{\simeq} & C \ar{d}{p} \\
B \times \sd(S)_0 \ar{r}{(\overline{F} \lambda,\pr)} \ar[dotted]{ru} & \sd_{S_0}(C_0) \times_{\sd(S_0)} \sd(S)_0 \ar{r}{\max_S} & S
\end{tikzcd} \]
The dotted lift exists by (1), and postcomposition of such a lift by $\max_C$ defines the desired lift.
\end{proof}

% \begin{rem}
% Suppose $S = P$ is a poset. Given any $p \in P$ and applying \cref{thm:subdivisionExtension}(2) with $S= P^{\geq p}$ and $S_0 = \{ p \}$, this says that the inclusion $\{ p \} \to \sd(P^{\geq p})_0$ exhibits the latter as the free locally cocartesian fibration on $\{ p \}$ over $P$. This recovers \cite[Lem.~A.6.14]{AMGRb}.
% \end{rem}

\subsubsection{Main results}

We begin by constructing a factorization system \cite[Def.~5.2.8.8]{HTT} on $\sd(S)$ associated to a sieve-cosieve decomposition of $S$. To do this, we need a few preparatory lemmas.

\begin{lem} \label{lem:cartesianSlice} Let $p: X \to S$ be a cartesian fibration. Given a functor $\phi: K \to X$, let
\[ \overline{p}: X^{\phi/} = \Fun(K^{\rhd},X) \times_{\Fun(K, X)} \{\phi\} \to S^{p \phi/} = \Fun(K^{\rhd},X) \times_{\Fun(K,X)} \{ p \phi\} \]
be the functor induced by $p$. Then $\overline{p}$ is a cartesian fibration, and an edge $\overline{e}: \overline{x} \to \overline{y} \in X^{\phi/}$ is $\overline{p}$-cartesian if and only if the underlying edge $e: x \to y \in X$ is $p$-cartesian.
\end{lem}
\begin{proof} We may duplicate the proof of \cite[3.1.2.1]{HTT} to prove the lemma, the essential tool being \cite[3.1.2.3]{HTT}. In more detail, let $E$ be the described collection of edges in $X^{\phi/}$ and suppose given a lifting problem in marked simplicial sets of the form
\[ \begin{tikzcd}[row sep=4ex, column sep=4ex, text height=1.5ex, text depth=0.25ex]
\rightnat{\Lambda^n_n} \ar{r} \ar{d} & (X^{\phi/}, E) \ar{d}{\overline{p}} \\
\rightnat{\Delta^n} \ar{r} \ar[dotted]{ru} & (S^{p \phi/})^{\sharp} 
\end{tikzcd} \]
where we mark the edge $\{n-1,n\}$ of $\Lambda^n_n$ (if $n>1$) and of $\Delta^n$. This transposes to a lifting problem of the form
\[ \begin{tikzcd}[row sep=4ex, column sep=4ex, text height=1.5ex, text depth=0.25ex]
\rightnat{\Lambda^n_n} \times K^{\rhd} \bigcup_{\rightnat{\Lambda^n_n} \times K} \rightnat{\Delta^n} \times K \ar{r}{f} \ar{d}{i} & \rightnat{X} \ar{d}{p} \\
\rightnat{\Delta^n} \times K^{\rhd} \ar[dotted]{ru} \ar{r} & S^{\sharp}
\end{tikzcd} \]
where we mark the $p$-cartesian edges in $X$. Note that $f$ is indeed a map of marked simplicial sets: this is by definition of $E$ for $f$ on the edge $\{n-1,n\} \times \{v\}$ ($v \in K^{\rhd}$ the cone point), and by definition of $f$ on $\Delta^n \times K$ as given by $\phi \circ \pr_K$ for the other marked edges. Applying \cite[3.1.2.3]{HTT}, we deduce that $i$ is marked right anodyne, so the dotted lift exists.
% use this? together with the bifibration property of the functor $\Fun(K^{\rhd},X) \to \Fun(K,X) \times X$: given a map $s \to t$ in $S$
\end{proof}

\begin{lem} \label{lem:mappingSpacesCartesianFibration} Let $p: X \to S$ be a cartesian fibration. Suppose we have a commutative square in $X$
\[ \begin{tikzcd}[row sep=4ex, column sep=4ex, text height=1.5ex, text depth=0.25ex]
x \ar{r}{h} \ar{d}{f} & z \ar{d}{g} \\
y \ar{r}{k}  & w.
\end{tikzcd} \]
If the edge $g$ is $p$-cartesian, then we have an equivalence
\[ \Map_{x//w}(y,z) \xto{\simeq} \Map_{px//pw}(py,pz). \]
\end{lem}
\begin{proof} By \cref{lem:cartesianSlice}, $\overline{p}: X^{x/} \to S^{px/}$ is a cartesian fibration and $g$, viewed as an edge $h \to kf$, is a $\overline{p}$-cartesian edge. Therefore, we have a homotopy pullback square of spaces
\[ \begin{tikzcd}[row sep=4ex, column sep=4ex, text height=1.5ex, text depth=0.25ex]
\Map_{x/}(y,z) \ar{r}{g_{\ast}} \ar{d}{p} & \Map_{x/}(y,w) \ar{d}{p} \\
\Map_{px/}(py,pz) \ar{r}{pg_{\ast}} & \Map_{px/}(py,pw)
\end{tikzcd} \]
Taking fibers over $k \in \Map_{x/}(y,w)$ and $pk \in \Map_{px/}(py,pw)$ yields the claimed equivalence.
\end{proof}

Fix a functor $\pi: S \to \Delta^1$ and let $S_i$ denote the fiber over $i \in \{0,1\}$. We now define a factorization system on $\widehat{\Ar}^{\simeq}(S)$ that will restrict to a factorization system on the full subcategory $\sd(S)$. Recall that the data of a morphism $e: x \to y$ in $\widehat{\Ar}^{\simeq}(S)$ is given by an inclusion $\alpha: \into{\Delta^{a_0}}{\Delta^{a_1}}$ and a map $f: \Delta^1 \times_{\Delta^{\inj}} \cE \Delta^{\inj} \to S$ that restricts to $x: \Delta^{a_0} \to S$ and $y: \Delta^{a_1} \to S$, such that $f$ sends morphisms $(i \in \Delta^{a_0}) \to (\alpha(i) \in \Delta^{a_1})$ to equivalences in $S$.

\begin{dfn} Let $\cL$ be the subclass of morphisms $(\alpha,f): x \to y$ such that for every $i \notin \im \alpha$, we have that $y(i) \in S_0$, and let $\cR$ be the subclass of morphisms $(\alpha,f): x \to y$ such that for every $i \notin \im \alpha$, we have that $y(i) \in S_1$.
\end{dfn}

\begin{prp} $(\cL, \cR)$ defines a factorization system on $\widehat{\Ar}^{\simeq}(S)$ and on $\sd(S)$.
\end{prp}
\begin{proof} We will check the assertion concerning $\widehat{\Ar}^{\simeq}(S)$; the second assertion will then be an obvious consequence. We first explain how to factor morphisms. Suppose that $\gamma: \into{\Delta^{a_0}}{\Delta^{a_2}}$, $h: \Delta^1 \times_{a, \Delta^{\inj}} \cE \Delta^{\inj} \to S$ is the data of a morphism in $\widehat{\Ar}^{\simeq}(S)$ from $x$ to $z$. Let $\Delta^{a_1} \subset \Delta^{a_2}$ be the subset on those $i \in \Delta^{a_2}$ such that $i \in \im \gamma$ or $z(i) \in S_0$. We then obtain a factorization of $\gamma$ as
\[ \begin{tikzcd}[row sep=4ex, column sep=4ex, text height=1.5ex, text depth=0.25ex]
\Delta^{a_0} \ar[hookrightarrow]{r}{\alpha} & \Delta^{a_1} \ar[hookrightarrow]{r}{\beta} & \Delta^{a_2}.
\end{tikzcd} \]
defining $\overline{a}: \Delta^2 \to \Delta^{\inj}$ extending the given $a: \Delta^{\{0,2\}} \to \Delta^{\inj}$.
Let $r: \Delta^2 \times_{\overline{a}, \Delta^{\inj}} \cE \Delta^{\inj} \to \Delta^1 \times_{a, \Delta^{\inj}} \cE \Delta^{\inj}$ be the unique retraction which is the identity on $\Delta^{a_0}$ and $\Delta^{a_2}$ and is given by $\beta$ on $\Delta^{a_1}$. Let $\overline{h} = h \circ r$. Then $\overline{h}$ is the desired factorization of $h$, as it corresponds to a factorization
\[ \begin{tikzcd}[row sep=4ex, column sep=4ex, text height=1.5ex, text depth=0.25ex]
x \ar{r}{f} \ar[bend right]{rr}[swap]{h} & y \ar{r}{g} & z \\
\end{tikzcd} \]
with $y = z \circ \beta: \Delta^{a_1} \to S$ defined so that $y(i) \in S_0$ for all $i \notin \im \alpha$ and $z(j) \in S_1$ for all $j \notin \im \beta$, hence $f$ in $\cL$, and $g$ in $\cR$.

Next, observe that because $S_0$ and $S_1$ are closed under retracts, so are $\cL$ and $\cR$. It only remains to check that $\cL$ is left orthogonal to $\cR$. For this, suppose given a commutative square in $\widehat{\Ar}^{\simeq}(S)$ on the left with $f \in \cL$ and $g \in \cR$ covering the square in $\Delta^{\inj}$ on the right
\[ \begin{tikzcd}[row sep=4ex, column sep=4ex, text height=1.5ex, text depth=0.25ex]
x \ar{r}{h} \ar{d}{f} & z \ar{d}{g} \\
y \ar{r}{k} \ar[dotted]{ru} & w
\end{tikzcd} \qquad, \qquad
\begin{tikzcd}[row sep=4ex, column sep=4ex, text height=1.5ex, text depth=0.25ex]
\Delta^a \ar{r}{\delta} \ar{d}{\alpha} & \Delta^c \ar{d}{\beta} \\
\Delta^b \ar{r}{\kappa} \ar[dotted]{ru}{\gamma} & \Delta^d.
\end{tikzcd} \]
Because $\xi_S: \widehat{\Ar}^{\simeq}(S) \to \Delta^{\inj}$ is a right fibration, by \cref{lem:mappingSpacesCartesianFibration} it suffices to show that $\Map_{\Delta^a//\Delta^d}(\Delta^b,\Delta^c)$ is contractible. This holds if and only if $\Delta^b \subset \Delta^c$ when viewed as subsets of $\Delta^d$, so that the mapping space is non-empty. Our hypothesis ensures that if $i \notin \im \beta$, then $w(i) \in S_1$, and if $i \in \Delta^b$, either $i \in \im \alpha$ or $y(i) \in S_0$. Therefore, we must have that for every $i \in \Delta^b$ with $i \notin \im \alpha$ that $w(\kappa(i)) \in S_0$, and hence $\kappa(i) \in \im \beta$. We conclude that the dotted lift $\gamma$ exists, which completes the proof.
\end{proof}

Let $\Ar^L(\sd(S)) \subset \Ar(\sd(S))$ denote the full subcategory on those morphisms $x \to y$ in the class $\cL$.

%Relative adjunction
\begin{lem} \label{lem:relativeAdjSubdivision} \begin{enumerate}[leftmargin=*] \item The inclusion $i: \Ar^L(\sd(S)) \subset \Ar(\sd(S))$ admits a right adjoint $r$ that on objects sends $h: x \to y$ to $f: x \to z$ where $h$ factors as $g \circ f$ according to the $(\cL, \cR)$ factorization system.
\end{enumerate}
\begin{enumerate}
\setcounter{enumi}{1}
\item $i \dashv r$ defines a relative adjunction with respect to evaluation $\ev_0$ at the source, and therefore for every $x \in \sd(S)$ we obtain an adjunction
\[ \adjunctb{\{x\} \times_{\sd(S)} \Ar^L(\sd(S)) }{ \sd(S)^{x/}  }. \]
\item The relative adjunction $i \dashv r$ restricts to a relative adjunction
\[ \adjunct{i}{\Ar^L(\sd(S)) \times_{\ev_1, \sd(S)} \sd(S)_0 }{\Ar(\sd(S)) \times_{\ev_1, \sd(S)} \sd(S)_0 }{r} \]
and therefore for every $x \in \sd(S)$ we obtain an adjunction
\[ \adjunctb{\{x\} \times_{\sd(S)} \Ar^L(\sd(S)) \times_{\sd(S)} \sd(S)_0 }{ \sd(S)_0^{x/}  }. \]
\end{enumerate}
\end{lem}
\begin{proof} Claim (1) is the dual formulation of \cite[5.2.8.19]{HTT}. Claims (2) and (3) then follow by the definition of relative adjunction \cite[7.3.2.1]{HA} and its pullback property \cite[7.3.2.5]{HA}.
\end{proof}

We are now prepared to construct the recollement adjunctions. Note that the hypotheses of the following theorem are satisfied if $S$ is equivalent to a finite poset and $p: C \to S$ is a locally cocartesian fibration such that the fibers admit finite limits and the pushforward functors preserve finite limits.
 % (for example, if $p$ is classified by a $2$-functor valued in stable $\infty$-categories and exact functors between them).

\begin{thm} \label{thm:ExistenceLaxRightKanExtension} Let $p: C \to S$ be a locally cocartesian fibration, let $\pi: S \to \Delta^1$ be a functor, and suppose we have a commutative diagram
\[ \begin{tikzcd}[row sep=4ex, column sep=4ex, text height=1.5ex, text depth=0.25ex]
\sd(S)_0 \ar[hookrightarrow]{d}{\phi} \ar{r}{F} & C \ar{d}{p} \\
\sd(S) \ar{r}{\max_S} & S
\end{tikzcd} \]
where $F$ preserves locally cocartesian edges. Given $x \in \sd(S_1)$, let
\[ J_x = \{x\} \times_{\sd(S)} \Ar^L(\sd(S)) \times_{\sd(S)} \sd(S)_0 .\]
Note that $(\max_S \circ \ev_1)|_{J_x}$ is constant at $\max_S(x)$.
\begin{enumerate} \item If for every $x \in \sd(S_1)$, the limit of $(F \ev_1)|_{J_x}: J_x \to C_{\max_S(x)}$ exists, then the $p$-right Kan extension $G$ of $F$ along $\phi$ exists and $G(x) \simeq \lim\limits_{\ot} F|_{J_x}$.
\item If for every $f: s \to t$ in $S$, the pushforward functor $f_!: C_s \to C_t$ preserves all limits appearing in (1), then $G$ preserves all locally cocartesian edges.
\item If the hypotheses of (1) and (2) hold for all $F$, then we have an adjunction
\[ \adjunct{\phi^{\ast}}{\Fun^{\cocart}_{/S}(\sd(S),C)}{\Fun^{\cocart}_{/S}(\sd(S)_0,C)}{\phi_{\ast}}. \]
\end{enumerate}
\end{thm}
\begin{proof} Note that $\sd(S_1) \subset \sd(S)$ is the complementary \emph{sieve} inclusion to the \emph{cosieve} $\sd(S)_0 \subset \sd(S)$. For (1), to show existence of the $p$-right Kan extension it suffices for every $x \in \sd(S_1)$ to show that the $p$-limit of $F \circ \pr_1: \sd(S)_0^{x/} \to \sd(S)_0 \to C$ exists. But by the argument of \cref{cor:RightKanExtensionComputedInFiber} applied to the adjunction $\adjunctb{J_x}{\sd(S)_0^{x/}}$ of \cref{lem:relativeAdjSubdivision}, this follows from the given hypothesis.

For (2), first note that there are no locally $\max_S$-cocartesian edges $e: x \to y$ such that $x \in \sd(S_1)$ and $y \in \sd(S)_0$, or vice-versa, so it suffices to handle the case where $e: x \to y$ is a locally $\max_S$-cocartesian edge in $\sd(S_1)$ only. Let $f: \max_S(x) = s \to \max_S(y) = t$ be the edge in $S_1 \subset S$. If $f$ is an equivalence, then $e$ is an equivalence and $G(e)$ is an equivalence, so we may suppose $f$ is not an equivalence. Then by the description of the locally $\max_S$-cocartesian edges in \cref{lm:subdivisionLocallyCocartesianByMaxFunctor}, $y$ is obtained from $e$ by appending the edge $f$. Correspondingly, the functor $J_y \xto{\simeq} J_x$ defined via sending $y \to z$ to $x \to z$ by precomposing is an equivalence, using that such edges are constrained to only add objects in $S_0$. Examining how the functoriality of $G$ is obtained from the pointwise existence criterion for Kan extensions, we see that the comparison morphism in $C_t$
\[ \psi: f_! G(x) \simeq f_! (\lim\limits_{\ot} F \ev_1|_{J_x}) \to G(y) \simeq \lim_{\ot} F \ev_1|_{J_y} \]
is induced via the functoriality of limits (contravariant in the diagram, covariant in the target) from the commutative diagram
\[ \begin{tikzcd}[row sep=4ex, column sep=4ex, text height=1.5ex, text depth=0.25ex]
J_x \ar{r}{F \ev_1} & C_s \ar{d}{f_!} \\
J_y \ar{u}{\simeq} \ar{r}{F \ev_1} & C_t.
\end{tikzcd} \]
The hypothesis that $f_!$ preserve limits indexed by $J_x$ together with $J_y \simeq J_x$ then proves that $\psi$ is an equivalence.

Finally, for (3) it is clear that if $G: \sd(S) \to C$ preserves locally cocartesian edges, then the restriction $\phi^{\ast} G$ of $G$ to $\sd(S)_0$ does as well. (1) and (2) establish the same fact for $\phi_{\ast} F$. Hence, the characteristic adjunction
\[ \adjunct{\phi^{\ast}}{\Fun_{/S}(\sd(S),C)}{\Fun_{/S}(\sd(S)_0,C)}{\phi_{\ast}} \]
of the $p$-right Kan extension along $\phi$ restricts to the full subcategories of functors preserving locally cocartesian edges in order to yield the desired adjunction.
\end{proof}

\begin{rem} \label{rem:IdentifyCategoryForLimitWhenPoset} Suppose that $S$ is a poset and $x \in S_1 \subset \sd(S_1)$. Then the $\infty$-category $J_x$ that appears in \cref{thm:ExistenceLaxRightKanExtension} is the poset whose objects are strings $[a_0 < \cdots < a_n < x]$, $n \geq 0$ with $a_i \in S_0$ and whose morphisms are string inclusions.
\end{rem}

\begin{cor} \label{cor:openPartOfRecollement} Suppose the hypotheses of \cref{thm:ExistenceLaxRightKanExtension} are satisfied. Let $j: \sd(S_0) \to \sd(S)$ denote the inclusion. Then the functor $j^{\ast}$ of restriction along $j$ participates in an adjunction
\[\adjunct{j^{\ast}}{\Fun^{\cocart}_{/S}(\sd(S),C)}{\Fun^{\cocart}_{/S_0}(\sd(S_0),C_0)}{j_{\ast}} \]
with fully faithful right adjoint $j_{\ast}$.
\end{cor}
\begin{proof} Combine \cref{thm:ExistenceLaxRightKanExtension} and \cref{thm:subdivisionExtension}(2).
\end{proof}

We also have a far simpler result concerning the calculation of the left adjoint $j_!$ of $j^{\ast}$ (but see \cref{rem:leftadjointspecial}).

\begin{prp} \label{prp:openPartOfRecollementLeftAdjoint} Let $p: C \to S$ be a locally cocartesian fibration, let $\pi: S \to \Delta^1$ be a functor, and suppose that for every $s \in S_1$, the fiber $C_s$ admits an initial object $\emptyset$, and for every $[f:s \rightarrow t] \in S_1$ the pushforward functors $f_!$ all preserve initial objects. Then $j^{\ast}$ admits a fully faithful left adjoint $j_!$ such that for $F: \sd(S_0) \to C_0$, $j_! F(x) \simeq \emptyset$ for all $x \in \sd(S_1)$.
\end{prp}
\begin{proof} Suppose we have a commutative diagram
\[ \begin{tikzcd}[row sep=4ex, column sep=4ex, text height=1.5ex, text depth=0.25ex]
\sd(S)_0 \ar[hookrightarrow]{d}{\phi} \ar{r}{F} & C \ar{d}{p} \\
\sd(S) \ar{r}{\max_S} & S.
\end{tikzcd} \]
For all $x \in \sd(S_1)$, the fiber product $\sd(S)^{/x} \times_{\sd(S)} \sd(S)_0$ is the empty category. Therefore, under our assumption the $p$-left Kan extension $\phi_! F$ of $F$ along $\phi$ exists and is computed by $\phi_! F(x) = \emptyset$ on $\sd(S_1)$. Combining this observation with \cref{thm:subdivisionExtension}(2), we obtain the desired adjunction
\[ \adjunct{j_!}{\Fun^{\cocart}_{/S_0}(\sd(S_0),C_0)}{\Fun^{\cocart}_{/S}(\sd(S),C)}{j^{\ast}}. \]
\end{proof}

We next turn to the cosieve inclusion $S_1 \subset S$. Note that the inclusion $i: \into{\sd(S_1)}{\sd(S)}$ is a sub-locally cocartesian fibration with respect to $\max_S: \sd(S) \to S$, and is in addition a \emph{sieve} inclusion, and hence $i$ is a cartesian fibration. In fact, the cosieve inclusion $j: \into{\sd(S)_0}{\sd(S)}$ is complementary to $i$.

\begin{prp} \label{prp:closedPartOfRecollement} Let $p: C \to S$ be a locally cocartesian fibration, let $\pi: S \to \Delta^1$ be a functor, and suppose the fibers of $p$ admit terminal objects and the pushforward functors preserve terminal objects. Then we have the adjunction
\[ \adjunct{i^{\ast}}{\Fun^{\cocart}_{/S}(\sd(S),C)}{\Fun^{\cocart}_{/S_1}(\sd(S_1),C_1)}{i_{\ast}} \]
with $i_{\ast}$ fully faithful, where $i^{\ast}$ is given by restriction along $i$ and $i_{\ast}$ is $p$-right Kan extension along $i$. Moreover, for a functor $G: \sd(S_1) \to C_1$, we have $(i_{\ast} G)(x) \simeq \ast \in C_{\max_S(x)}$ for all $x \in \sd(S)_0$.
\end{prp}
\begin{proof} By \cref{cor:RightKanExtensionComputedInFiber}, using the hypothesis that the fibers of $p$ admit terminal objects we have the adjunction
\[ \adjunct{i^{\ast}}{\Fun_{/S}(\sd(S),C)}{\Fun_{/S_1}(\sd(S_1),C_1)}{i_{\ast}}  \]
with $i^{\ast}$ and $i_{\ast}$ as described. Then using that the pushforward functors preserve terminal objects, we see that this adjunction restricts to the one of the proposition.
\end{proof}

\begin{lem} \label{lem:rlaxLimitAdmitsFiniteLimitsAndStable} Let $p: C \to S$ be a locally cocartesian fibration and suppose that the fibers $C_s$ admit $K$-(co)limits and the pushforward functors preserve $K$-(co)limits. Then $\Fun^{\cocart}_{/S}(\sd(S), C)$ admits $K$-indexed (co)limits, and for all $\sigma \in \sd(S)$ over $s = \max_S(\sigma)$, the evaluation functor $\ev_{\sigma}: \Fun^{\cocart}_{/S}(\sd(S), C) \to C_s$ preserves $K$-indexed (co)limits. Moreover, if the fibers $C_s$ are stable $\infty$-categories and the pushforward functors are exact, then $\Fun^{\cocart}_{/S}(\sd(S), C)$ is a stable $\infty$-category.
\end{lem}
\begin{proof} Apply \cite[Prop.~5.4.7.11]{HTT} to the locally cocartesian fibration $\sd(S) \times_S C \to \sd(S)$, with the subcategory of $\widehat{\Cat}_{\infty}$ either taken to be those $\infty$-categories that admit $K$-indexed (co)limits and functor that preserve $K$-indexed (co)limits, or the subcategory $\Cat_{\infty}^{\stab}$ of stable $\infty$-categories and exact functors thereof.
\end{proof}

Finally, putting everything together, we get:

\begin{thm} \label{thm:RecollementRlaxLimitOfLlaxFunctor} Let $p: C \to S$ be a locally cocartesian fibration whose fibers admit finite limits and whose pushforward functors preserve finite limits. Let $\pi: S \to [1]$ be a functor and suppose moreover that the hypotheses of \cref{thm:ExistenceLaxRightKanExtension} hold so that the adjunction \ref{thm:ExistenceLaxRightKanExtension}(3) exists. Then the two adjunctions of \cref{cor:openPartOfRecollement} and \cref{prp:closedPartOfRecollement} combine to exhibit $\Fun^{\cocart}_{/S}(\sd(S),C)$ as a recollement of $\Fun^{\cocart}_{/S_0}(\sd(S_0),C_0)$ and $\Fun^{\cocart}_{/S_1}(\sd(S_1),C_1)$.
\end{thm}
\begin{proof} We verify the conditions to be a recollement. By our assumption on $p$ and \cref{lem:rlaxLimitAdmitsFiniteLimitsAndStable}, finite limits in $\Fun^{\cocart}_{/S}(\sd(S),C)$ exist and are computed fiberwise. Therefore, the restriction functors $j^{\ast}$ and $i^{\ast}$ are left exact. By the formula for $i_{\ast}$ given in \cref{prp:closedPartOfRecollement}, it is clear that $j^{\ast} i_{\ast}$ is constant at the terminal object. Finally, we check that $j^{\ast}$ and $i^{\ast}$ are jointly conservative. Suppose given a morphism $\alpha: F \to F'$ in $\Fun^{\cocart}_{/S}(\sd(S),C)$ such that $j^{\ast} \alpha$ and $i^{\ast} \alpha$ are equivalences. Observe that $\alpha$ is an equivalence if and only if for all $x \in S$, $\alpha_x: F(x) \to F'(x)$ is an equivalence (viewing $x$ as an object in $\sd(S)$). Because any object of $S$ lies in either $S_0$ or $S_1$, we deduce that $\alpha$ is an equivalence.
\end{proof}

\begin{rem} \label{rem:downfinite}
Suppose that $S$ is a \emph{down-finite} poset $P$. Let $C \to P$ be a locally cocartesian fibration such that its fiber admit finite limits and its pushforward functors preserve finite limits. Then the hypotheses of \cref{thm:ExistenceLaxRightKanExtension} automatically hold for every sieve-cosieve decomposition of $P$. Indeed, the categories $J_x$ that appear there are all finite (cf. \cref{rem:IdentifyCategoryForLimitWhenPoset}).
\end{rem}

Let us now return to the question of the existence of $j_!$.

\begin{rem} \label{rem:leftadjointspecial}
In fact, the left adjoint $j_!$ in \cref{prp:openPartOfRecollementLeftAdjoint} should exist even if we only suppose that the fibers of $C$ admits initial objects (i.e., we need not suppose that the pushforward functors preserve initial objects). However, in that case $j_!$ will not generally be the $p$-left Kan extension along the inclusion $\phi$, and relatedly, a direct proof of this would appear to be overly cumbersome in our framework. Rather, we can say the following (which covers most cases of practical relevance):
\begin{itemize}[leftmargin=*]
\item Suppose that the hypotheses of \cref{thm:RecollementRlaxLimitOfLlaxFunctor} are satisfied and we have also shown that $\Fun^{\cocart}_{/S_1}(\sd(S_1),C)$ admits an initial object. Then as in any recollement situation, $j_!$ exists and is computed by $j_!(u)=[u,\emptyset \rightarrow i^*j_*(u)]$.
\item To exhibit the initial object of $\Fun^{\cocart}_{/S_1}(\sd(S_1),C)$, suppose also that $S_1$ is a finite poset $P$. Then using \cref{thm:RecollementRlaxLimitOfLlaxFunctor} in conjunction with \cref{lem:ColimitExistenceInRecollement}, we may proceed by induction on the cardinality of $P$ and repeatedly invoke our assumption that the fibers of $C$ admit an initial object to conclude that $\Fun^{\cocart}_{/S_1}(\sd(S_1),C)$ admits an initial object whose evaluation at every singleton string is also initial.
\end{itemize}
\end{rem}

We conclude this subsection by giving an application of \cref{thm:RecollementRlaxLimitOfLlaxFunctor} to the presentability of the right-lax limit $\Fun^{\cocart}_{/S}(\sd(S),C)$. First suppose that $S$ is equivalent to a \emph{finite} poset and write $P = S$. 

% \begin{dfn} Given an object $s \in S$, its \emph{dimension} $\dim_S(s)$ is the supremum over all $n$ such that there exists a functor $\sigma: \Delta^n \to S$, $\sigma(n) = s$ with $\sigma(\{i,i+1 \})$ not an equivalence for all $0 \leq i < n$. The \emph{dimension} $\dim(S)$ of $S$ is the supremum of $\{ \dim_S(s): s \in S \}$.
% \end{dfn}

\begin{prp} \label{prp:rightLaxLimitPresentable} Suppose that the fibers $C_s$ of $p: C \to P$ are presentable and the pushforward functors are left-exact and accessible. Then $\Fun^{\cocart}_{/P}(\sd(P),C)$ is presentable, and for all $s \in P$, the evaluation functor $\ev_s: \Fun^{\cocart}_{/P}(\sd(P),C) \to C_{s}$ preserves (small) colimits and is accessible.
\end{prp}
\begin{proof} The accessibility statements follow from \cite[Prop.~5.4.7.11]{HTT} as in \cref{lem:rlaxLimitAdmitsFiniteLimitsAndStable}, so we only need to show the existence and preservation of small colimits. Our strategy is to proceed by induction on the cardinality of $P$. If $|P| \leq 1$, then the statement is clear. Suppose for the inductive hypothesis that we have established the statement for all posets $Q$ such that $|Q|< |P|$. Let $b \in P$ be a maximal object and let $\pi: P \to \Delta^1$ be the functor determined by the sieve-cosieve decomposition $P_0 = P\setminus \{b\}$ and $P_1 = \{b \}$. Because the diagrams that appear in \cref{thm:ExistenceLaxRightKanExtension} are finite, we may apply \cref{thm:RecollementRlaxLimitOfLlaxFunctor} to decompose $\Fun^{\cocart}_{/P}(\sd(P),C)$ as a recollement of $\Fun^{\cocart}_{/P_0}(\sd(P_0),C_0)$ and $C_b$. By the inductive hypothesis, both these $\infty$-categories admit all small colimits such that the evaluation functors at objects in $P_0$ and $P_1$ are colimit-preserving. By \cref{lem:ColimitExistenceInRecollement}, we conclude that $\Fun^{\cocart}_{/P}(\sd(P),C)$ admits all small colimits such that the evaluation functors for objects $s \in P$ are colimit-preserving.
\end{proof}

Next, we may use the equivalence (cf. \cref{rem:barycentricSubdivisionOrdinaryCategory})
\begin{equation} \label{hypothesis} \tag{$\star$}
\sd(S) \xot{\simeq} \colim_{[n] \in \Delta_{/S}} \sd([n])    
\end{equation}
to promote \cref{prp:rightLaxLimitPresentable} to a statement involving arbitrary $S$.

\begin{cor}
Suppose the fibers $C_s$ of $p: C \to S$ are presentable and the pushforward functors are left-exact and accessible. Then $\Fun^{\cocart}_{/S}(\sd(S), C)$ is presentable.
\end{cor}
\begin{proof}
We may simply copy over the proof strategy used to establish \cite[Prop.~6.1.6(1)]{AMGRb}. By \eqref{hypothesis}, we have that
\[ \Fun^{\cocart}_{/S}(\sd(S), C) \xto{\simeq} \lim_{[n] \in (\Delta_{/S})^{\op}} \Fun^{\cocart}_{/[n]}(\sd[n], C|_{[n]}). \]
By \cref{prp:rightLaxLimitPresentable} and \cref{thm:RecollementRlaxLimitOfLlaxFunctor}, for every $[\sigma: [n] \to S] \in \Delta_{/S}$, $\mathrm{lim}^{\rlax} \sigma^* C$ is presentable and the evaluation functors $\{ \ev_i: \mathrm{lim}^{\rlax} \sigma^* C \to C_{\sigma(i)} \}_{i=0}^n$ are colimit-preserving and jointly conservative. Note then that for any map $\alpha:[m]\to [n]$, the restriction functor
$$\alpha^*: \mathrm{lim}^{\rlax} \sigma^* C \to \mathrm{lim}^{\rlax} \alpha^* \sigma^* C$$
preserves colimits. Then since $\mathrm{lim}^{\rlax} C$ is a limit of presentable $\infty$-categories along colimit-preserving functors, it is presentable.
\end{proof}

\begin{rem} \label{rem:amgrCompare}
We explain a subtle difference between our general approach and the one of \cite[\S 6]{AMGRb}, which is adapted to the case of locally cocartesian fibrations $p: C \to P$ over a poset $P$ whose fibers are presentable stable $\infty$-categories and whose pushforward functors are exact and accessible. Suppose one could prove directly that $\Fun^{\cocart}_{/P}(\sd(P), C)$ is presentable (for any poset) and that the restriction functor $j^*: \Fun^{\cocart}_{/P}(\sd(P), C) \to \Fun^{\cocart}_{/P_0}(\sd(P_0), C)$ preserves colimits, so that it admits a right adjoint $j_*$. Then without a pointwise formula for $j_*$, it is generally difficult to show that $j_*$ is fully faithful. However, this would follow if we could also exhibit a fully faithful \emph{left} adjoint $j_!$ to $j^*$, and this turns out to be easier to analyze (cf. \cref{prp:openPartOfRecollementLeftAdjoint}). This is the strategy adopted in the proof of \cite[Prop.~6.1.6]{AMGRb}.

Therefore, if we were only interested in the existence of the recollement on $\Fun^{\cocart}_{/P}(\sd(P), C)$ in the stable presentable case, then we could bypass the work that goes into establishing the pointwise formula of \cref{thm:ExistenceLaxRightKanExtension}. However, our primary motivation for undertaking this work lay precisely in having this pointwise formula. Note also that in the presentable case, the right adjoint $j_*$ exists unconditionally even if it is not describable as a relative right Kan extension.

On the other hand, such tricks aren't available in the absence of presentability (though for idempotent-complete small stable $\infty$-categories, one can pass to their $\Ind$-completions as is done in \cite[\S 7.2]{AMGRb}). Over a down-finite poset $P$ (cf. \cref{rem:downfinite}), our \cref{thm:RecollementRlaxLimitOfLlaxFunctor} thus allows one to strengthen \cite[Thm.~A]{AMGRb} by removing all of the presentability hypotheses therein.
\end{rem}

\subsubsection{Symmetric monoidal structure} \label{sec:symmMon}

We briefly explain how to promote \cref{thm:RecollementRlaxLimitOfLlaxFunctor} to a statement involving symmetric monoidal recollements. First recall the notions of left-lax and right-lax morphisms of locally cocartesian fibrations from \cite[\S A]{AMGRb}:

\begin{rec} \label{rec:amgr}
Let $\lambda, \xi: \sC, \sD \to S$ be locally cocartesian fibrations. A \emph{left-lax} morphism $\lambda \to \xi$ is a functor $F: \sC \to \sD$ over $S$ (which need not preserve locally cocartesian edges). In contrast, a \emph{right-lax} morphism $\lambda \to \xi$ is defined as in \cite[\S A.5]{AMGRb} as the ``unstraightened'' counterpart to a right-lax natural transformation of left-lax functors.

The collection of locally cocartesian fibrations over $S$ and right-lax morphisms thereof assemble into an $\infty$-category $\LocCocart^{\rlax}_{S}$ which contains $\LocCocart_{S}$ as a wide subcategory. Moreover, $\mathrm{lim}^{\rlax}$ extends to a functor over $\LocCocart^{\rlax}_{S}$ that is right adjoint to the constant functor $\const: \sE \mapsto \sE \times S$. See \cite[Obs.~A.5.8]{AMGRb}.
\end{rec}

% First, let $\LocCocart^{\rlax}_S$ denote the $\infty$-category whose objects are locally cocartesian fibrations over $S$ and whose morphisms are \emph{right-lax} morphisms of locally cocartesian fibrations in the sense of \cite[Def.~A.5.4]{AMGRb}. Then by \cite[Obs.~A.5.8]{AMGRb}, $\mathrm{lim}^{\rlax}$ participates in an adjunction
% \[ \adjunct{\const}{\Cat}{\LocCocart^{\rlax}_S}{\mathrm{\lim}^{\rlax}} \]
% where the left adjoint sends $C$ to $C \times S$.
In view of the adjunction $\const \dashv \mathrm{lim}^{\rlax}$, $\mathrm{\lim}^{\rlax}$ sends commutative monoids in $\LocCocart^{\rlax}_S$ to symmetric monoidal $\infty$-categories. Moreover, a diagram chase shows that given a commutative monoid structure on $[p:C \to S]$, for any $\alpha:T \to S$ the pullback $[\alpha^* C \to T]$ is a commutative monoid in $\LocCocart^{\rlax}_T$ and the restriction functor $\mathrm{lim}^{\rlax} C \to \mathrm{lim}^{\rlax} \alpha^* C$ is symmetric monoidal. It follows that if the recollement of \cref{thm:RecollementRlaxLimitOfLlaxFunctor} exists in this situation, then it is symmetric monoidal.

\begin{rem}
If $S = \Delta^1$, then a commutative monoid in $\LocCocart^{\rlax}_{\Delta^1}$ is the data of a lax symmetric monoidal functor of symmetric monoidal $\infty$-categories (cf. \cite[Prop.~2.6]{kondyrev2021dualizable}). In general, to endow $p: C \to S$ with the structure of a commutative monoid entails endowing its fibers with symmetric monoidal structures and its pushforward functors and natural transformations thereof with lax symmetric monoidal structures in a coherent fashion. See \cite[\S 4]{AMGRb} for how to produce examples from simpler input.
\end{rem}

\section{\texorpdfstring{$1$}{1}-generated and extendable objects} \label{sec:extendable}

Suppose $S = \Delta^2$ and $p: C \to \Delta^2$ is a locally cocartesian fibration classified by a $2$-functor
\[ \begin{tikzcd}[row sep=4ex, column sep=6ex, text height=1.5ex, text depth=0.25ex]
C_{0} \ar{rd}[swap]{F} \ar{rr}{H} & \ar[phantom]{d}{\Downarrow} & C_{2}. \\
& C_{1} \ar{ru}[swap]{G} &
\end{tikzcd} \]
Then the data of a functor $\sd(\Delta^2) \to C$ over $\Delta^2$ that preserves locally cocartesian edges can be summarized as follows:
\begin{itemize}
    \item Objects $c_i \in C_i$ for $i = 0,1,2$.
    \item Morphisms $f: c_1 \to F(c_0)$, $g: c_2 \to G(c_1)$, and $h: c_2 \to H(c_0)$.
    \item A commutative square
\[ \begin{tikzcd}[row sep=4ex, column sep=6ex, text height=1.5ex, text depth=0.25ex]
c_2 \ar{r}{h} \ar{d}{g} & H(c_0) \ar{d}{\can} \\
G(c_1) \ar{r}{G(f)} & G F(c_0).
\end{tikzcd} \]
\end{itemize}

Furthermore, if the map $\mit{can}$ is an equivalence, then the data of the commutative square and the morphism $h$ is redundant, since then $h \simeq G(f) \circ g$ and compositions in an $\infty$-category are unique up to contractible choice. More precisely, if we let $\gamma_2: \sd_1(\Delta^2) \subset \sd(\Delta^2)$ be the subposet on $\{[0], [1], [2], [0<1], [1<2] \}$, then the functor
\[ \gamma_2^{\ast}: \Fun^{\cocart}_{/\Delta^2}(\sd(\Delta^2), C) \to \Fun^{\cocart}_{/\Delta^2}(\sd_1(\Delta^2), C) \]
is a trivial fibration onto its image when restricted to objects for which $\mit{can}$ is an equivalence.

Our goal in this section is to generalize this observation to the case where $S=\Delta^n$. We introduce subcategories of $1$-generated and extendable objects (\cref{dfn:OneGenerated} and \cref{dfn:extendability}) and show their equivalence under the restriction functor $\gamma_n^{\ast}$ (\cref{thm:OneGenerationAndExtension}), given a stability hypothesis on $C \xto{p} \Delta^n$. This material will play an important role in \cite{QS21b}.

\begin{ntn} \label{ntn:convexStringsLengthOne} Let $\gamma_n: \sd_1(\Delta^n) \subset \sd(\Delta^n)$ be the subposet on strings $[k]$ and $[k<k+1]$.
\end{ntn}

We also introduce convenient notation for convex subposets of $\Delta^n$.

\begin{ntn} Let $[i:j] \subset \Delta^n$ denote the subposet on $i \leq k \leq j$.
\end{ntn}

Via its inclusion into $\sd(\Delta^n)$, we regard $\sd_1(\Delta^n)$ as a simplicial set over $\Delta^n$ (i.e., by the functor that takes the maximum) and as a marked simplicial set (so that each edge $[k] \to [k<k+1]$ is marked). We first state the analogue of \cref{thm:RecollementRlaxLimitOfLlaxFunctor} for $\sd_1$, whose proof is far simpler.

\begin{prp} \label{prp:StaircaseRecollement} Let $p: C \to \Delta^n$ be a locally cocartesian fibration such that the fibers admit finite limits and the pushforward functors preserve finite limits. Let $0 \leq k < n$, so the subcategories $[0:k] \cong \Delta^k$ and $[k+1:n] \cong \Delta^{n-k-1}$ of $\Delta^n$ give a sieve-cosieve decomposition. Then we have adjunctions
\[ \begin{tikzcd}[row sep=4ex, column sep=4ex, text height=1.5ex, text depth=0.25ex]
\Fun^{\cocart}_{/[0:k]}(\sd_1([0:k]), C_{[0:k]}) \ar[shift right=1,right hook->]{r}[swap]{j_{\ast}} & \Fun^{\cocart}_{/\Delta^n}(\sd_1(\Delta^n), C) \ar[shift right=2]{l}[swap]{j^{\ast}} \ar[shift left=2]{r}{i^{\ast}} & \Fun^{\cocart}_{/[k+1:n]}(\sd_1([k+1:n]), C_{[k+1:n]}) \ar[shift left=1,left hook->]{l}{i_{\ast}}
\end{tikzcd} \]
that exhibit $\Fun^{\cocart}_{/\Delta^n}(\sd_1(\Delta^n), C)$ as a recollement.
\end{prp}
\begin{proof} Let $j: \sd_1([0:k]) \to \sd_1(\Delta^n)$ and $i: \sd_1([k+1:n]) \to \sd_1(\Delta^n)$ be the inclusions, so $j^{\ast}$ and $i^{\ast}$ are defined by restriction along $j$ and $i$. As in the proof of \cref{lem:rlaxLimitAdmitsFiniteLimitsAndStable}, our hypotheses on $p$ ensure that the three $\infty$-categories admit finite limits and the functors $j^{\ast}$ and $i^{\ast}$ are left-exact. Moreover, since equivalences are detected on strings $[k]$, $j^{\ast}$ and $i^{\ast}$ are jointly conservative. The functor $i_{\ast}$ is obtained by $p$-right Kan extension as in the proof of \cref{prp:closedPartOfRecollement}, and its essential image consists of functors $F: \sd_1(\Delta^n) \to C$ such that $F(i)$ is a terminal object in $C_i$ for all $0 \leq i \leq k$, so $j^{\ast} i_{\ast}$ is the constant functor at the terminal object.

Finally, we show existence of $j_{\ast}$. Let $\sd_1([0:k])^+$ be the subposet of $\sd_1([0:n])$ on all objects in $\sd_1([0:k])$ and $\{[k<k+1]\}$, with marking inherited from $\sd(\Delta^n)$. Then we have a pushout square of marked simplicial sets
\[ \begin{tikzcd}[row sep=4ex, column sep=4ex, text height=1.5ex, text depth=0.25ex]
\Delta^0 \ar{r} \ar{d} & (\Delta^1)^{\sharp} \ar{d} \\
\sd_1([0:k]) \ar{r} & \sd_1([0:k])^+ 
\end{tikzcd} \]
so the inclusion $\sd_1([0:k]) \subset \sd_1([0:k])^+$ is $\mathfrak{P}$-anodyne for the categorical pattern $\mathfrak{P}$ defining the locally cocartesian model structure on $s\Set^+_{/\Delta^n}$. We thus obtain a trivial fibration
\[ \Fun^{\cocart}_{/[0:k+1]}(\sd_1([0:k])^+, C_{[0:k+1]}) \to \Fun^{\cocart}_{/[0:k]}(\sd_1([0:k]), C_{[0:k]}). \]
On the other hand, given a commutative diagram
\[ \begin{tikzcd}[row sep=4ex, column sep=4ex, text height=1.5ex, text depth=0.25ex]
\sd_1([0:k])^+ \ar{r}{F} \ar{d} & C \ar{d}{p} \\
\sd_1([0:k+1]) \ar{r} \ar[dotted]{ru}[swap]{G} & \Delta^n,
\end{tikzcd} \]
since $\sd_1([0:k])^+ \times_{\sd_1([0:k+1])} \sd_1([0:k+1])_{[k+1]/} \cong \{[k < k+1]\}$, $F$ admits a $p$-right Kan extension along $\sd_1([0:k])^+ \subset \sd_1([0:k+1])$ and $G$ is a $p$-right Kan extension of $F$ if and only if $G$ sends the edge $[k+1] \to [k<k+1]$ to an equivalence. Therefore, we may alternate between anodyne extension and $p$-right Kan extension along the filtration
\[ \sd_1([0:k]) \subset \sd_1([0:k])^+ \subset \sd_1([0:k+1]) \subset \cdots \sd_1([0:n-1])^+ \subset \sd_1(\Delta^n) \]
to define the functor $j_{\ast}$. Moreover, we see that the essential image of $j_{\ast}$ consists of those functors $\sd_1(\Delta^n) \to C$ that send the edges $[l+1] \to [l < l+1]$ to equivalences for all $l \geq k$.
\end{proof}

We next wish to introduce a condition on objects of $\Fun^{\cocart}_{/\Delta^n}(\sd(\Delta^n), C)$, which we term \emph{$1$-generated}, that indicates that the data of such objects is essentially determined by their restriction to $\sd_1(\Delta^n)$.

\begin{ntn} Given a string $\sigma = [i<i+k]$ in $\sd(\Delta^n)$, let $Q_{\sigma} \subset \sd(\Delta^n)$ be the subposet on all strings $[i<\cdots<i+k]$. Note that $Q_{\sigma}$ is a $(k-1)$-dimensional cube lying in the fiber $\sd(\Delta^n)_{\max=i+k}$ with $\sigma$ as its minimal element.
\end{ntn}

\begin{dfn} \label{dfn:OneGenerated} Let $C \to \Delta^n$ be a locally cocartesian fibration and $F: \sd(\Delta^n) \to C$ be a functor that preserves locally cocartesian edges. We say that $F$ is \emph{$1$-generated} if for all strings $\sigma = [i<i+k]$ in $\sd(\Delta^n)$, $F|_{Q_{\sigma}}$ is a limit diagram in $C_{i+k}$.

Let $\Fun^{\cocart}_{/\Delta^n}(\sd(\Delta^n), C)_{\gen{1}}$ be the full subcategory on the $1$-generated objects.
\end{dfn}

% \begin{dfn} Let $C \to \Delta^n$ be a locally cocartesian fibration and $F: \sd(\Delta^n) \to C$ be a functor that preserves locally cocartesian edges. We say that $F$ is \emph{$1$-generated} if for all string inclusions $e: [i<i+k] \to [i < i+1 < i+k]$ in $\sd(\Delta^n)$, $F(e)$ is an equivalence in $C_{i+k}$. Let $\Fun^{\cocart}_{/\Delta^n}(\sd(\Delta^n), C)_{\gen{1}}$ be the full subcategory of $1$-generated objects.
% \end{dfn}

\begin{lem} \label{lm:equivalentOneGenerationConditions} Let $C \to \Delta^n$ be a locally cocartesian fibration whose fibers are stable $\infty$-categories and whose pushforward functors are exact. Then $F: \sd(\Delta^n) \to C$ is $1$-generated if and only if for all string inclusions $e: [i<i+k] \to [i < i+1 < i+k]$ in $\sd(\Delta^n)$, $F(e)$ is an equivalence in $C_{i+k}$.
\end{lem}
\begin{proof} We will prove the stronger claim that for fixed $k \geq 2$ and all string inclusions $e_{ij}: \sigma_{ij} = [i<i+j] \to [i < i+1 < i+j]$ with $2 \leq j \leq k$, $F|_{Q_{\sigma_{ij}}}$ is a limit diagram for all $Q_{\sigma_{ij}}$ if and only if $F(e_{ij})$ is an equivalence for all $e_{ij}$.

We proceed by induction on $k$. For the base case $k=2$, given a string inclusion $\sigma = [i<i+2] \to [i<i+1<i+2]$, the edge is the $1$-dimensional cube $Q_{\sigma}$, so $F|_{Q_{\sigma}}$ is a limit diagram if and only if $F(e)$ is an equivalence. Now let $k>2$ and suppose we have proven the statement for all $l<k$. Note that in proving either direction of the `if and only if' statement, we may suppose that $F|_{Q_{\sigma_{ij}}}$ is a limit diagram \emph{and} $F(e_{ij})$ for all $2 \leq j < k$, so let us do so.

Consider an edge $e: \sigma = [i<i+k] \to [i < i+1 < i+k]$. For $1<j<k$, let $Q_{\sigma,j} \subset Q_{\sigma}$ be the subposet on strings excluding vertices $i+j, ..., i+k-1$. Then we have a descending filtration of sieve inclusions
\[ Q_{\sigma} \coloneq Q_{\sigma,k} \supset Q_{\sigma, k-1} \supset Q_{\sigma, k-2} \supset \cdots \supset Q_{\sigma,2} \]
where $Q_{\sigma, j}$ is a $(j-1)$-dimensional cube and $Q_{\sigma,2}$ consists only of the edge $e$. Note that if we let $Q_{\sigma,j}' = Q_{\sigma,j+1} \setminus Q_{\sigma,j}$ for $1<j<k$, then the minimal element of $Q_{\sigma,j}'$ is given by $\sigma_j = [i<i+j<i+k]$, and if we let $\sigma'_j = [i<i+j]$, then $Q_{\sigma,j}'$ is obtained from $Q_{\sigma_j'}$ by concatenating $i+k$. By the inductive hypothesis and using that the pushforward functors are exact, we get that $F|_{Q_{\sigma,j}'}$ is a limit diagram. Taking total fibers of cubes then shows that $F|_{Q_{\sigma,j}}$ is a limit diagram if and only if $F|_{Q_{\sigma,j-1}}$ is a limit diagram. Traversing the filtration, we conclude that $F|_{Q_{\sigma}}$ is a limit diagram if and only if $F(e)$ is an equivalence.
\end{proof}

% Let us also record an observation made in the proof of \cref{lm:equivalentOneGenerationConditions} as a separate lemma.

\begin{lem} \label{lm:limitCubeDegenerateIfOneGenerated} Let $Q = \sd(\Delta^n)_{\max = n}$, $D$ a stable $\infty$-category, and $f: Q \to D$ a functor. Suppose the following condition holds:
\begin{itemize} \item[($\ast$)] For all string inclusions $e: \sigma \to \sigma'$ in $Q$ obtained by concatenating $[i<k] \to [i<i+1<k] $ by a (possibly empty) suffix $\tau$, $f(e)$ is an equivalence.
\end{itemize}
Then $f$ is a limit diagram if and only if $f([n] \to [n-1<n])$ is an equivalence.
\end{lem}
\begin{proof} The proof is similar to that of \cref{lm:equivalentOneGenerationConditions}. For $0 \leq j < n$, let $Q_{\geq j}$, $Q_{=j}$ be the subposet on strings $\sigma$ with minimum $\geq j$, resp $=j$. Then $Q_{\geq j}$ is a $(n-j)$-dimensional cube, $Q_{=j} = Q_{\geq j} \setminus Q_{\geq j+1}$ is a $(n-j-1)$-dimensional cube, and we have a descending filtration
\[ Q = Q_{\geq 0} \supset Q_{\geq 1} \supset Q_{\geq 2} \supset \cdots \supset Q_{\geq n-1}. \] 
Observe that $Q_{=j} = Q_{[j<n]}$, so $f|_{Q_{=j}}$ is a limit diagram under our hypotheses by the proof of \cref{lm:equivalentOneGenerationConditions}. Therefore, taking total fibers shows that $f|_{ Q_{\geq j}}$ is a limit diagram if and only if $f|_{ Q_{\geq j+1}}$ is a limit diagram. Traversing the filtration then proves the claim.
\end{proof}

We continue to assume $C \to \Delta^n$ is a locally cocartesian fibration whose fibers are stable $\infty$-categories and whose pushforward functors are exact. Observe that we have a commutative diagram
\[ \begin{tikzcd}[row sep=4ex, column sep=6ex, text height=1.5ex, text depth=0.5ex]
\Fun^{\cocart}_{/[0:n-1]}(\sd([0:n-1]), C_{[0:n-1]}) \ar{r}{\gamma_{n-1}^{\ast}} & \Fun^{\cocart}_{/[0:n-1]}(\sd_1([0:n-1]), C_{[0:n-1]}) \\
\Fun^{\cocart}_{/\Delta^n}(\sd(\Delta^n), C) \ar{r}{\gamma_n^{\ast}} \ar{u}{j^{\ast}} \ar{d}[swap]{i^{\ast}} & \Fun^{\cocart}_{/\Delta^n}(\sd_1(\Delta^n), C) \ar{u}{j^{\ast}} \ar{d}[swap]{i^{\ast}}  \\
C_n \ar{r}{\id} & C_n,
\end{tikzcd} \]
so in particular $\gamma_n^{\ast}$ is a morphism of stable recollements. However $\gamma_n$ generally fails to be a \emph{strict} morphism of stable recollements, i.e., the natural transformation
\[ i^{\ast} j_{\ast} \to i^{\ast} j_{\ast} \gamma_{n-1}^{\ast} \]
is typically not an equivalence.

\begin{lem} \label{lem:OneGeneratedStrictMorphismOfRecollements} Suppose $F: \sd(\Delta^n) \to C$ is $1$-generated. Then the comparison map
\[  i^{\ast} j_{\ast} j^{\ast} F = (j_{\ast} j^{\ast} F)(n) \to i^{\ast} j_{\ast} \gamma_{n-1}^{\ast} j^{\ast} F =  (j_{\ast}( F|_{\sd_1([0:n-1])}))(n) \]
is an equivalence.
\end{lem}
\begin{proof} Let $K \subset \sd(\Delta^n)$ be the subposet on strings $\sigma$ with $\max(\sigma) = n$ and $\sigma \neq n$. By the formulas computing $j_{\ast}$ given in \cref{thm:ExistenceLaxRightKanExtension} and \cref{prp:StaircaseRecollement}, we see that the comparison map is given by the canonical map from the limit of $F|_{K}$ to $F([n-1<n])$. Since $F$ is $1$-generated, by \cref{lm:equivalentOneGenerationConditions} the conditions of \cref{lm:limitCubeDegenerateIfOneGenerated} are satisfied, so this canonical map is an equivalence.
\end{proof}

\begin{dfn} For the functor $j_{\ast}$ defined as in \cref{thm:ExistenceLaxRightKanExtension} with respect to $[0:n-1]$ and $\{ n \}$, we say that a functor $F: \sd([0:n-1]) \to C_{[0:n-1]}$ is \emph{$+$-1-generated} if both $F$ and $j_{\ast} F$ are $1$-generated. Let
\[ \Fun^{\cocart}_{/[0:n-1]}(\sd([0:n-1]), C_{[0:n-1]})_{\gen{1}}^+ \]
be the full subcategory on the $+$-$1$-generated objects.
\end{dfn}

\begin{lem} \label{lem:OneGeneratedRecollement} We have adjunctions
\[ \begin{tikzcd}[row sep=4ex, column sep=4ex, text height=1.5ex, text depth=0.25ex]
\Fun^{\cocart}_{/[0:n-1]}(\sd([0:n-1]), C_{[0:n-1]})_{\gen{1}}^+ \ar[shift right=1,right hook->]{r}[swap]{j_{\ast}} & \Fun^{\cocart}_{/\Delta^n}(\sd(\Delta^n), C)_{\gen{1}} \ar[shift right=2]{l}[swap]{j^{\ast}} \ar[shift left=2]{r}{i^{\ast}} & C_n \ar[shift left=1,left hook->]{l}{i_{\ast}}
\end{tikzcd} \]
that exhibit $\Fun^{\cocart}_{/\Delta^n}(\sd(\Delta^n), C)_{\gen{1}}$ as a stable recollement.
\end{lem}
\begin{proof} Clearly, we may define $j_{\ast}$, $i^{\ast}$, and $i_{\ast}$ to be the restrictions of the corresponding functors for the adjunctions of \cref{thm:RecollementRlaxLimitOfLlaxFunctor}. The only subtle point is that given $F: \sd(\Delta^n) \to C$ which is $1$-generated, we require that the localization $j_{\ast} j^{\ast} F$ is also $1$-generated. But this holds, since $F \simeq j_{\ast} j^{\ast} F$ except possibly at $n \in \sd(\Delta^n)$ and the $1$-generated condition ignores $n$. Therefore, we may also define $j^{\ast}$ as the restricted functor, and the recollement conditions are then immediate. 
\end{proof}

\begin{cor} \label{cor:OneGeneratedStrictMorphismOfRecollements} The restriction $\gamma_n^{\ast}: \Fun^{\cocart}_{/\Delta^n}(\sd(\Delta^n), C)_{\gen{1}} \to \Fun^{\cocart}_{/\Delta^n}(\sd_1(\Delta^n), C)$ is a strict morphism of stable recollements with respect to \cref{lem:OneGeneratedRecollement} and \cref{prp:StaircaseRecollement}.
\end{cor}
\begin{proof} This follows immediately from \cref{lem:OneGeneratedStrictMorphismOfRecollements}.
\end{proof}

We want to apply \cref{cor:OneGeneratedStrictMorphismOfRecollements} to show that $\gamma_n^{\ast}$ is an equivalence (in fact, a trivial fibration) onto its essential image. To understand this image as a condition on objects in the codomain, we introduce the following definition. For $0 \leq i < j \leq n$, let $\tau^j_i: C_i \to C_j$ denote the pushforward functor encoded by the locally cocartesian fibration.

\begin{dfn} \label{dfn:extendability} We say that a functor $f: \sd_1(\Delta^n) \to C$ is \emph{extendable} if for every string $[i<i+1<i+k]$ in $\sd(\Delta^n)$, the canonical map in $C_{i+k}$
\[ \tau^{i+k}_i f(i) \to (\tau^k_{i+1} \circ \tau^{i+1}_i) f(i) \]
encoded by the locally cocartesian fibration is an equivalence. Let
\[ \Fun^{\cocart}_{/\Delta^n}(\sd_1(\Delta^n),C)_{\ext} \]
denote the full subcategory on the extendable objects.
\end{dfn}

\begin{dfn} For the functor $j_{\ast}$ defined as in \cref{prp:StaircaseRecollement} with respect to $[0:n-1]$ and $\{ n \}$, we say that a functor $f: \sd_1([0:n-1]) \to C$ is \emph{$+$-extendable} if both $f$ and $j_{\ast} f$ are extendable. Let
\[ \Fun^{\cocart}_{/[0:n-1]}(\sd_1([0:n-1]), C_{[0:n-1]})_{\ext}^+ \]
be the full subcategory on the $+$-extendable objects.
\end{dfn}

Note that the extendability condition becomes stronger through considering the additional strings in $\sd(\Delta^n)$; for example, extendability is no condition on $f: \sd_1([0:1]) \to C_{[0:1]}$, but we acquire the condition that the map $\tau^2_0 f(0) \to \tau^2_1 \tau^1_0 f(0)$ is an equivalence upon enlarging to $\Delta^2$. Let us first state the evident counterpart to \cref{lem:OneGeneratedRecollement}.

\begin{lem} \label{lem:ExtendableObjectsRecollement} We have adjunctions
\[ \begin{tikzcd}[row sep=4ex, column sep=4ex, text height=1.5ex, text depth=0.25ex]
\Fun^{\cocart}_{/[0:n-1]}(\sd_1([0:n-1]), C_{[0:n-1]})_{\ext}^+ \ar[shift right=1,right hook->]{r}[swap]{j_{\ast}} & \Fun^{\cocart}_{/\Delta^n}(\sd_1(\Delta^n), C)_{\ext} \ar[shift right=2]{l}[swap]{j^{\ast}} \ar[shift left=2]{r}{i^{\ast}} & C_n \ar[shift left=1,left hook->]{l}{i_{\ast}}
\end{tikzcd} \]
that exhibit $\Fun^{\cocart}_{/\Delta^n}(\sd_1(\Delta^n), C)_{\ext}$ as a stable recollement.
\end{lem}
\begin{proof} This is immediate from restricting the recollement of \cref{prp:StaircaseRecollement}.
\end{proof}

We have assembled all the ingredients needed to prove \cref{thm:OneGenerationAndExtension}. Note that by \cref{lm:limitCubeDegenerateIfOneGenerated}, $\gamma_n^{\ast}$ of a $1$-generated object is extendable, so the functor of \cref{thm:OneGenerationAndExtension} is well-defined.

\begin{thm} \label{thm:OneGenerationAndExtension} Suppose $C \to \Delta^n$ is a locally cocartesian fibration whose fibers are stable $\infty$-categories and whose pushforward functors are exact. Then the functor
\[ \gamma_n^{\ast}: \Fun^{\cocart}_{/\Delta^n}(\sd(\Delta^n), C)_{\gen{1}} \to \Fun^{\cocart}_{/\Delta^n}(\sd_1(\Delta^n),C)_{\ext} \]
is an equivalence of $\infty$-categories.
\end{thm}
\begin{proof} We proceed by induction on $n$. For the base cases $n = 0$ and $n=1$, the result is trivial. Let $n>1$ and suppose we have proven the theorem for all $k<n$. By the inductive hypothesis, $\gamma_{n-1}^{\ast}$ is an equivalence. Observe that $\gamma_{n-1}^{\ast}$ restricts to a functor
\[ (\gamma_{n-1}^{\ast})^+: \Fun^{\cocart}_{/[0:n-1]}(\sd([0:n-1]), C_{[0:n-1]})_{\gen{1}}^+ \to \Fun^{\cocart}_{/[0:n-1]}(\sd_1([0:n-1]), C_{[0:n-1]})_{\ext}^+. \]
If we let $(\gamma_{n-1}^{\ast})^{-1}$ be an inverse functor, then by \cref{lm:equivalentOneGenerationConditions}, if $f: \sd_1([0:n-1]) \to C_{[0:n-1]}$ is $+$-extendable, then $(\gamma_{n-1}^{\ast})^{-1}(f)$ is $+$-1-generated. Therefore, $(\gamma_{n-1}^{\ast})^+$ is also an equivalence. By \cref{cor:OneGeneratedStrictMorphismOfRecollements} (but replacing the codomain there with the recollement of \cref{lem:ExtendableObjectsRecollement}) and the two-out-of-three property of equivalences for a strict morphism of stable recollements (\cref{rem:TwoOutOfThreePropertyEquivalencesStrictMorphismRecoll}), we deduce that $\gamma_n^{\ast}$ is an equivalence.
\end{proof}

\begin{obs} \label{dualizedDescriptionOfSpine} To make better use of \cref{thm:OneGenerationAndExtension}, let us further unpack $\Fun_{/\Delta^n}^{\cocart}(\sd_1(\Delta^n),C)$. Note that we may write $\sd_1(\Delta^n)$ as the union of marked simplicial sets
\[ \sd([0:1]) \cup_{1} \sd([1:2]) \cup_{2} \cdots \cup_{n} \sd([n-1:n]), \]
so we obtain a fiber product decomposition
\[ \Fun_{/\Delta^n}^{\cocart}(\sd_1(\Delta^n),C) \simeq \Fun_{/[0:1]}^{\cocart}(\sd([0:1]), C_{[0:1]}) \times_{C_1} \cdots \times_{C_{n-1}} \Fun_{/[n-1:n]}^{\cocart}(\sd([n-1:n]), C_{[n-1:n]}). \]

Let $\tau^{i+1}_i: C_i \to C_{i+1}$ be the pushforward functors as before, and with respect to the trivial fibration (induced by the inner anodyne spine inclusion $[0:1] \cup_1 \cdots \cup_{n-1} [n-1:n] \to \Delta^n$)
\[ \Fun(\Delta^n, \Cat_{\infty}) \xto{\simeq} \Fun([0:1], \Cat_{\infty}) \times_{1} \cdots \times_{n-1} \Fun([n-1:n], \Cat_{\infty}), \]
let $\tau_{\bullet}: \Delta^n \to \Cat_{\infty}$ be a functor lifting the $\tau^{i+1}_i$. Let $C^{\vee} \to (\Delta^n)^{\op}$ be a cartesian fibration classified by $\tau_{\bullet}$. Then if we let $[i+1:i] = [i:i+1]^{\op}$, we have that $(C^{\vee})_{[i+1:i]} \simeq (C_{[i:i+1]})^{\vee}$ where the righthand $(-)^{\vee}$ denotes the dual cartesian fibration of the cocartesian fibration $C_{[i:i+1]} \to [i:i+1]$. Then by \cref{dualizingOneSimplex}, we have an equivalences of $\infty$-categories
\[ \Fun^{\cocart}_{/[i:i+1]}(\sd([i:i+1]), C_{[i:i+1]}) \simeq \Fun_{/[i+1:i]} ([i+1:i], C^{\vee}_{[i+1:i]}) \simeq \Ar(C_{i+1}) \times_{\ev_1,C_{i+1}, \tau^{i+1}_i} C_i. \]
\end{obs}

Again using that the spine inclusion is inner anodyne, we obtain the following proposition.

\begin{prp} \label{prp:DualDescriptionOfSections} We have equivalences of $\infty$-categories
\begin{align*} \Fun_{/\Delta^n}^{\cocart}(\sd_1(\Delta^n),C) & \simeq \Fun_{/(\Delta^n)^{\op}}((\Delta^n)^{\op},C^{\vee}) \\
 & \simeq \Ar(C_n) \times_{C_n} \Ar(C_{n-1}) \times_{C_{n-1}} \cdots \times_{C_2} \Ar(C_1) \times_{C_1} C_0,
\end{align*}
where in the fiber product, the maps $\Ar(C_k) \to C_k$ are given by evaluation at the target, and the maps $\Ar(C_k) \to C_{k+1}$ are given by composing evaluation at the source with $\tau^{k+1}_k: C_k \to C_{k+1}$.
\end{prp}

\begin{ntn} Under the equivalence of \cref{prp:DualDescriptionOfSections}, let $\Fun_{/(\Delta^n)^{\op}}((\Delta^n)^{\op},C^{\vee})_{\ext}$ denote the extendable objects. Then we will also write (abusing notation)
\[ \begin{tikzcd}[row sep=4ex, column sep=6ex, text height=1.5ex, text depth=0.5ex]
\Fun^{\cocart}_{/\Delta^n}(\sd(\Delta^n), C)_{\gen{1}} \ar{r}{\gamma_n^{\ast}}[swap]{\simeq} \ar[hook]{d} & \Fun_{/(\Delta^n)^{\op}}((\Delta^n)^{\op},C^{\vee})_{\ext} \ar[hook]{d} \\
\Fun^{\cocart}_{/\Delta^n}(\sd(\Delta^n), C) \ar{r}{\gamma_n^{\ast}} & \Ar(C_n) \times_{C_n} \cdots \times_{C_1} C_0.
\end{tikzcd} \]
\end{ntn}

\begin{rem} \label{rem:NikolausScholze} The type of iterated fiber product occuring in \cref{prp:DualDescriptionOfSections} appears in the work of Nikolaus and Scholze when they describe the data of a (genuine) $C_{p^n}$-spectrum $X$ whose geometric fixed points (except possibly $\Phi^{C_{p^n}} X$ and $\Phi^{C_{p^{n-1}}} X$) are all bounded below; cf. \cite[Rem.~II.4.8]{NS18}.\footnote{Nikolaus and Scholze elide the subtlety involving the lack of bounded-below hypotheses needed on $\Phi^{C_{p^n}} X$ and $\Phi^{C_{p^{n-1}}} X$.} In fact, \cref{thm:OneGenerationAndExtension} together with \cite[Thm.~E]{AMGRb} applies to give a proof of \cite[Rem.~II.4.8]{NS18} that is independent of the machinery of ``coalgebras for endofunctors'' developed in \cite[\S II.5]{NS18}. We will explain this in more detail in \cite[\S 3.2]{QS21b} as well as prove a dihedral refinement of this assertion. For now, we give an overview of the argument:

By \cite[Thm.~E]{AMGRb}, for any finite group $G$ with subconjugacy poset $P$ there exists a locally cocartesian fibration $\Sp^{G}_{\phi\text{-locus}} \to P$ whose right-lax limit is canonically equivalent\footnote{The comparison functor is defined analogously to the functor \eqref{eqn:comparison_map} in \cref{thmB}; cf. \cite[Constr.~2.43]{QS21a}.} to the $\infty$-category $\Sp^{G}$ of (genuine) $G$-spectra. Furthermore, for every subgroup $H \leq G$, $(\Sp^{G}_{\phi\text{-locus}})_{H} \simeq \Sp^{h W_G H} = \Fun(B W_G H, \Sp)$ where $W_G H = N_G H/H$ is the Weyl group, and the equivalence transports a $G$-spectrum $X$ to its associated diagram of geometric fixed points $\{ \Phi^H X \in \Sp^{h W_G H} \}$. If $G = C_{p^n}$, then we may identify the pushforward functor associated to $[C_{p^k} \leq C_{p^m}]$ with the proper Tate construction $(-)^{\tau C_{p^{m-k}}}$ endowed with residual action; in particular, when $m = k+1$, this is the ordinary Tate construction $(-)^{t C_p}$. In addition, under the equivalence $\Sp^{C_{p^n}} \simeq \mathrm{lim}^{\rlax} \Sp^{C_{p^n}}_{\phi\text{-locus}}$ and the isomorphism $P \cong [n]$, the map $\gamma_n^*$ identifies with the forgetful functor
\[ \Sp^{C_{p^n}} \to \Sp^{h C_{p^n}} \times_{\scriptscriptstyle (-)^{tC_p}, \Sp^{h C_{p^{n-1}}}, \ev_1} \Ar(\Sp^{h C_{p^{n-1}}}) \times_{\scriptscriptstyle (-)^{tC_p} \ev_0, \Sp^{h C_{p^{n-2}}}, \ev_1} \Ar(\Sp^{h C_{p^{n-2}}}) \times \cdots \times \Ar(\Sp) \] 
that sends $X$ to $[\Phi^e X, \Phi^{C_p}X \rightarrow (\Phi^e X)^{t C_p}, ..., \Phi^{C_{p^n}} X \to (\Phi^{C_{p^{n-1}}} X)^{t C_p}]$ where the maps are the usual ones. The assertion made in \cite[Rem.~II.4.8]{NS18} is that $\gamma_n^*$ restricts to an equivalence
\[ \Sp^{C_{p^n}}_{+} \xto{\simeq} \Sp^{h C_{p^n}}_{+} \times_{\scriptscriptstyle (-)^{tC_p}, \Sp^{h C_{p^{n-1}}}, \ev_1} \Ar'(\Sp^{h C_{p^{n-1}}}) \times_{\scriptscriptstyle (-)^{tC_p} \ev_0, \Sp^{h C_{p^{n-2}}}, \ev_1} \Ar'(\Sp^{h C_{p^{n-2}}}) \times \cdots \times \Ar(\Sp) \]
where:
\begin{itemize}[leftmargin=*]
\item $\Sp^{C_{p^n}}_{+} \subset \Sp^{C_{p^n}}$ denotes the full subcategory of $C_{p^n}$-spectra spanned by those objects whose geometric fixed points (except possibly $\Phi^{C_{p^n}}$ and $\Phi^{C_{p^{n-1}}}$) are all bounded-below;
\item $\Sp^{h C_{p^n}}_{+} \subset \Sp^{h C_{p^n}}$ denotes the full subcategory of Borel $C_{p^n}$-spectra spanned by those objects whose underlying spectrum is bounded-below;
\item $\Ar'$ denotes the full subcategory on arrows whose source is bounded-below. 
\end{itemize} 

To invoke \cref{thm:OneGenerationAndExtension} to deduce this, we need to show that for every $X \in \Sp^{C_{p^n}}_{+}$, $X$ is $1$-generated as an object in $\mathrm{lim}^{\rlax} \Sp^{C_{p^n}}_{\phi\text{-locus}}$. If $n=2$, then this is precisely the content of the \emph{Tate orbit lemma} of \cite[Lem.~I.2.1]{NS18} once one identifies the fiber of the natural transformation $\can: (-)^{\tau C_{p^2}} \Rightarrow ((-)^{t C_p})^{t C_{p^2}/ C_p}$ encoded by $\Sp^{C_{p^2}}_{\phi\text{-locus}}$ with $((-)_{h C_p})^{t C_{p^2}/ C_p}$. Proceeding by induction on $n$, it is then not difficult to verify that the condition of \cref{lm:equivalentOneGenerationConditions} holds for all $X \in \Sp^{C_{p^n}}_{+}$; we record this as \cite[Cor.~3.40]{QS21b}.
\end{rem}

\section{Reconstruction of sheaves on stratified \texorpdfstring{$\infty$}{infinity}-topoi}

We explain how to apply \cref{thm:RecollementRlaxLimitOfLlaxFunctor} to prove a reconstruction theorem (\cref{thm:ReconstructionTopoi}) for sheaves in an $\infty$-topos stratified by a finite poset $P$ in the sense of Barwick--Glasman--Haine (\cref{dfn:strattopos}). We then prove a conjecture of Barwick--Glasman--Haine by establishing an equivalence (\cref{thm:AdjunctionTopos}) between the $\infty$-category of $P$-stratified $\infty$-topoi and that of \emph{toposic} locally cocartesian fibrations over $P^{\op}$ (\cref{dfn:toposicfibration}). To begin with, we recall the basic structure theory of recollements of $\infty$-topoi.

\begin{exm} \label{exm:RecollementTopoi}
Let $\sX$ be an $\infty$-topos and $U$ a $(-1)$-truncated object. The slice $\infty$-topos $\sX_{/U}$ is said to be an \emph{open subtopos} of $\sX$ \cite[\S 6.3.5]{HTT}.\footnote{Lurie uses the terminology ``\'etale geometric morphism''.} Let $\sX_{\setminus U} = \{ x \in \sX: x \times U \xto{\simeq} U \} \subset \sX$. $\sX_{\setminus U}$ is the \emph{closed subtopos of $\sX$ complementary to $U$} \cite[Def.~7.3.2.6]{HTT}. We then have a diagram of adjunctions
\[ \begin{tikzcd}[column sep=4em]
\sX_{/U} \ar[hookrightarrow, shift left=2]{r}{j_!} \ar[hookrightarrow, shift right=4]{r}[swap]{j_*} & \sX \ar[shift left=1]{l}[description]{j^*} \ar[shift left=2]{r}{i^*} \ar[shift right=1, hookleftarrow]{r}[swap]{i_*} & \sX_{\setminus U} 
\end{tikzcd} \]
that exhibits $(\sX_{/U}, \sX_{\setminus U})$ as a recollement of $\sX$. Conversely, by \cite[Prop.~A.8.15]{HA} given a left-exact accessible functor $\phi: \sU \to \sZ$ between $\infty$-topoi, the fiber product $\sX \coloneq \Ar(\sZ) \times_{\ev_1, \sZ, \phi} \sU$ is an $\infty$-topos and there exists a uniquely determined $(-1)$-truncated object $U$ such that $\sU \simeq \sX_{/U}$ and $\sZ \simeq \sX_{\setminus U}$ compatibly with the adjunctions to $\sX$.

In what follows, we will generically use the notation $j_! \dashv j^* \dashv j_*$ and $i^* \dashv i_*$ for these functors arising from a recollement on an $\infty$-topos.
\end{exm}

\begin{dfn}
A \emph{locale} is a $0$-topos, i.e. a poset $L$ such that $L$ admits infinite joins $\bigvee_{\alpha} x_{\alpha}$ (so that $L$ is presentable) and infinite joins distribute over finite meets.
\end{dfn}

\begin{exm}
Let $\sX$ be an $\infty$-topos. Then its full subcategory $\Open(\sX)$ of $(-1)$-truncated objects is a locale. Note that $\Open(\sX)$ is isomorphic to the poset of open subtopoi of $\sX$ (embedded in $\sX$ via $j_!$) via the assignment $U \mapsto \sX_{/U}$. Note also that if $X$ is a topological space, then $\Open(\Shv(X))$ is isomorphic to the poset $\Open(X)$ of open sets in $X$. If $P$ is a poset equipped with the Alexandroff topology, then these are precisely the cosieves in $P$.
\end{exm}

\begin{exm} \label{exm:balmerlocale}
Let $\sC$ be a presentably symmetric monoidal stable $\infty$-category and suppose there is some regular cardinal $\kappa$ such that the unit and tensor product restrict to define a symmetric monoidal structure on the full subcategory $\sC^{\kappa}$ of $\kappa$-compact objects in $\sC$. Then the set of radical thick $\otimes$-ideals in $\sC^{\kappa}$ forms a coherent locale \cite[Thm.~3.1.9]{Kock2016}.
\end{exm}

\begin{dfn}[{\cite[Def.~8.2.1]{Exodromy}}] \label{dfn:strattopos}
Let $P$ be a poset and $\sX$ an $\infty$-topos. A \emph{$P$-stratification of $\sX$} is a geometric morphism $\pi_*: \sX \to \Shv(P)$ of $\infty$-topoi, or equivalently a geometric morphism $\pi_*: \Open(\sX) \to \Open(P)$ of locales. We also say that the data $(\sX, \pi_*)$ comprises that of a \emph{$P$-stratified $\infty$-topos}.
\end{dfn}

In the next remark, we consider $P^{\op} \subset \Open(P)$ as a subposet via the map $p \mapsto P^{\geq p}$.

\begin{rem} \label{rem:AMGRdefinition}
Via the assignment $\pi_* \mapsto \pi^*|_{P^{\op}}$, geometric morphisms $\pi_*: \Open(\sX) \to \Open(P)$ are in bijective correspondence with maps of posets $f: P^{\op} \to \Open(\sX)$ such that
\begin{enumerate}
\item $\bigvee_{p \in P} f(p) =  \mathds{1}$.
\item For every $p,q \in P$, $\bigvee_{r \geq p,q} f(r) \xto{\cong} f(p) \times f(q)$.
\end{enumerate}
Indeed, given any map of posets $f: P^{\op} \to \Open(\sX)$, its left Kan extension $F: \Open(P) \to \Open(\sX)$ admits a right adjoint $G$ defined by $G(U) = \{ p \in P: f(p) \leq U \}$, and $F$ is then left-exact if and only if $f$ satisfies conditions (1) and (2).

Furthermore, (2) is also equivalent to the following factorization property: for every $p,q \in P$, the square
\[ \begin{tikzcd}
\sX_{/\bigvee_{r \geq p,q} f(r)} \ar[hookrightarrow]{r}{j_!} & \sX_{/f(p)} \\ 
\sX_{/f(q)} \ar[hookrightarrow]{r}{j_!} \ar[->>]{u}{j^*} & \sX \ar[->>]{u}{j^*}
\end{tikzcd} \]
commutes. We thus see that the notion of a $P$-stratification of $\sX$ is the evident toposic analogue of the notion of a $P$-stratification of a presentable stable $\infty$-category in the sense of \cite[Def.~2.4.3]{AMGRb}. Conversely, in view of \cref{exm:balmerlocale} one can sometimes give a `localic' reformulation of \cite[Def.~2.4.3]{AMGRb} (or rather, its symmetric monoidal refinement \cite[Def.~4.3.2]{AMGRb}).
\end{rem}

We now proceed to notate various subtopoi associated to a $P$-stratified $\infty$-topos.

\begin{ntn}[{\cite[Notn.~8.2.3]{Exodromy}}] \label{ntn:ExodromySubtopoi}
Let $\pi_*: \sX \to \Shv(P)$ be a $P$-stratification of $\sX$. In what follows, all fiber products are computed in $\Top_{\infty}$. For any open subset $O \subset P$, we let
$$\sX_O \coloneq \sX_{/\pi^* O} \simeq \sX \times_{\Shv(P)} \Shv(O).$$
Dually, for any closed subset $Z \subset P$, we let
$$ \sX_Z \coloneq \sX_{\setminus \pi^*(P \setminus Z)} \simeq \sX \times_{\Shv(P)} \Shv(Z). $$
For any $p \in P$, we define the \emph{$p$th stratum} of $(\sX, \pi_*)$ to be
$$ \sX_p \coloneq \sX \times_{\Shv(P)} \Shv(\{p\}) .$$
\end{ntn}

\begin{ntn} \label{ntn:GeometricFixedPoints}
In \cref{ntn:ExodromySubtopoi}, the $p$th stratum $\sX_p$ is the closed complement of $\sX_{P^{>p}}$ in $\sX_{P^{\geq p}} = \sX_{/\pi^*(p)}$, or alternatively the open complement of $\sX_{P^{<p}}$ in $\sX_{P^{\leq p}}$. We then have the adjunction
\[ \begin{tikzcd}
\Phi^p: \sX \ar[shift left=2]{r}{j^*} \ar[hookleftarrow, shift right=1]{r}[swap]{j_*} & \sX_{/\pi^*(p)} \ar[shift left=2]{r}{i^*} \ar[hookleftarrow, shift right=1]{r}[swap]{i_*}  & \sX_p: \rho_p,
\end{tikzcd} \]
in which $\rho_p$ is a geometric morphism.
\end{ntn}

\begin{rem} \label{rem:OutOfPosition}
Let $\pi_*: \sX \to \Shv(P)$ be a $P$-stratification of $\sX$ and suppose $p,q \in P$ such that $p \ngeq q$. Then $\Phi^q \rho_p$ is homotopic to the constant map at the final object. Indeed, by \cref{rem:AMGRdefinition} we have a factorization of $\Phi^q \rho_p$ as
\[ \begin{tikzcd}
\sX_p \ar[hookrightarrow]{r}{i_*} & \sX_{/\pi^*(p)} \ar[hookrightarrow]{r}{j_*} \ar[->>]{d}{j^*} & \sX \ar[->>]{d}{j^*} \\ 
& \sX_{/\pi^*(P^{\geq p,q})} \ar[hookrightarrow]{r}{j_*} & \sX_{/\pi^*(q)} \ar[->>]{d}{i^*} \\ 
& & \sX_q
\end{tikzcd} \]
and since $p \notin P^{\geq p,q}$, the composite $j^* i_*: \sX_p \to \sX_{/\pi^*(P^{\geq p,q})}$ is homotopic to the constant map at the final object.
\end{rem}

Given a $P$-stratified $\infty$-topos $(\sX, \pi_*)$, we may construct its associated \emph{gluing diagram} in the same manner as \cite[Def.~2.5.7]{AMGRb}.

\begin{cnstr} \label{con:GluingDiagram}
Let $\sG(\sX) = \{ (x,p): x \in \sX_p \} \subset \sX \times P^{\op}$, where $\sX_p \subset \sX$ via $\rho_p$. The projection
$$\lambda: \sG(\sX) \to P^{\op}$$
is then a locally cocartesian fibration with fibers $\sX_p$ such that for all $q \leq p$, the corresponding pushforward functor $\Gamma^q_p: \sX_p \to \sX_q$ is given by $\Phi^q \circ \rho_p$ (cf. [AMGR, Obs.~2.5.5]). 
\end{cnstr}

We codify the structure of $\lambda: \sG(\sX) \to P^{\op}$ by means of the following definition.

\begin{dfn} \label{dfn:toposicfibration}
We call a locally cocartesian fibration $\lambda: \widehat{\sX} \to P^{\op}$ \emph{toposic} if its fibers are $\infty$-topoi and its pushforward functors are left-exact and accessible.
\end{dfn}

If $P$ is finite, we will show that taking the limit in $\sX$ furnishes an equivalence $\Theta_P: \mathrm{lim}^{\rlax} \sG(\sX) \xto{\simeq} \sX$, thereby proving a \emph{reconstruction theorem} for $(\sX, \pi_*)$. First, we note:

\begin{lem} \label{lem:ToposicFibrationYieldsTopos}
Let $P$ be a finite poset and $\lambda: \widehat{\sX} \to P^{\op}$ a toposic locally cocartesian fibration. Then the right-lax limit $\sX = \Fun^{\cocart}_{/P^{\op}}(\sd(P^{\op}), \widehat{\sX})$ is an $\infty$-topos. Moreover, any cosieve $O \subset P$ determines a recollement of $\sX$ with open subtopos given by the right-lax limit of $\lambda|_{O^{\op}}$ and complementary closed subtopos given by the right-lax limit of $\lambda|_{(P \setminus O)^{\op}}$.
 % $\Fun^{\cocart}_{/O^{\op}}(\sd(O^{\op}), \widehat{\sX}|_{O^{\op}})$ and complementary closed subtopos $\Fun^{\cocart}_{/(P \setminus O)^{\op}}(\sd((P \setminus O)^{\op}), \widehat{\sX}|_{(P \setminus O)^{\op}})$.
\end{lem}
\begin{proof}
Given \cref{thm:RecollementRlaxLimitOfLlaxFunctor} and proceeding by induction on the cardinality of $P$, the first part follows from the known statement for recollements of $\infty$-topoi recalled in \cref{exm:RecollementTopoi}. The second statement then follows by \cref{thm:RecollementRlaxLimitOfLlaxFunctor} again.
\end{proof}

Consider now the functor $\Theta_P: \Fun^{\cocart}_{/P^{\op}}(\sd(P^{\op}), \sG(\sX)) \to \sX$ that sends a functor $f: \sd(P^{\op}) \to \sG(\sX)$ to $\lim_{\sd(P^{\op})} (\pr_{\sX} \circ f)$.

\begin{thm} \label{thm:ReconstructionTopoi}
Suppose $P$ is a finite poset and let $(\sX, \pi_*)$ be a $P$-stratified $\infty$-topos. Then
$$\Theta_P: \Fun^{\cocart}_{/P^{\op}}(\sd(P^{\op}), \sG(\sX)) \to \sX$$
is an equivalence.
\end{thm}
\begin{proof} To ease notation, let $\sX' \coloneq \Fun^{\cocart}_{/P^{\op}}(\sd(P^{\op}), \sG(\sX))$. We proceed by induction on the cardinality of $P$. We may suppose that $P$ is nonempty. Choose a minimal element $b \in P$ and let $O = P \setminus \{ b \}$. Let
$$(\pi_O)_*: \Open(\sX_{/\pi^*(O)}) \to \Open(O)$$
denote the $O$-stratification of the open subtopos $\sX_{/\pi^*(O)}$ restricted from that of $\sX$. Note that $\sG(\sX)|_{O^{\op}} \simeq \sG(\sX_{/\pi^*(O)})$ as locally cocartesian fibrations over $O^{\op}$. Indeed, one observes that for all $p \in O$, the fully faithful inclusion $\rho_p: \into{\sX_p}{\sX}$ factors through $\sX_{/\pi^*(O)}$ and identifies $\sX_p$ with $(\sX_{/\pi^*(O)})_p$ embedded via $(\rho_O)_p$, so the inclusion $\sG(\sX)|_{O^{\op}} \subset \sX \times O^{\op}$ factors through $\sX_{/\pi^*(O)}$ (embedded via $j_*$ in $\sX$) and identifies with $\sG(\sX_{/\pi^*(O)})$.

Let $(\sX_{/\pi^*(O)})' \coloneq \Fun^{\cocart}_{/O^{\op}}(\sd(O^{\op}), \sG(\sX_{/\pi^*(O)}))$ and write
$$ \Theta_O: (\sX_{/\pi^*(O)})' \xto{\pr_*} \Fun(\sd(O^{\op}),\sX_{/\pi^*(O)}) \xto{\lim} \sX_{/\pi^*(O)}.$$
We now show that $\Theta_P: \sX' \to \sX$ is a morphism of recollements from $((\sX_{/\pi^*(O)})', \sX_b)$ to $(\sX_{/\pi^*(O)}, \sX_b)$:
\begin{enumerate}[leftmargin=*]
\item We have a distinguished homotopy making the diagram
\[ \begin{tikzcd}[column sep=4em]
\sX' \ar{r}{j^* = \res} \ar{d}{\Theta_P} & (\sX_{/\pi^*(O)})' \ar{d}{\Theta_O} \\ 
\sX \ar{r}{j^*} & \sX_{/\pi^*(O)}
\end{tikzcd} \]
commute as follows: given $[f: \sd(P^{\op}) \to \sG(\sX)] \in \sX'$, consider the composite
\[ g: \sd(P^{\op}) \xto{f} \sG(\sX) \xto{\pr} \sX \xto{j^*} \sX_{/\pi^*(O)} \]
whose limit is $j^* \Theta_P(f)$. Then since $\sX_b \xtolong{i_* = \rho_b}{1.2} \sX \xto{j^*} \sX_{/\pi^*(O)}$ is homotopic to the constant map at the final object, $g$ is a right Kan extension of its restriction $g_0$ to $\sd(O^{\op})$. But since the limit of $g_0$ is $\Theta_O j^*(f)$, this supplies an equivalence $j^* \Theta_P(f) \simeq \Theta_O j^*(f)$ that is natural in $f$.
\item Likewise, we may construct an equivalence
$$i^* \Theta_P = \Phi^b \Theta_P \simeq \ev_b: \sX' \to \sX_b$$
as follows: let $[f: \sd(P^{\op}) \to \sG(\sX)] \in \sX'$ and consider the composite
\[ g: \sd(P^{\op}) \xto{f} \sG(\sX) \xto{\pr} \sX \xto{i^*} \sX_b. \]
If $a \ngeq b$, then the composite $\begin{tikzcd} \sX_a \ar[hookrightarrow]{r}{\rho_a} & \sX \ar[->>]{r}{\Phi^b} & \sX_b \end{tikzcd}$ is homotopic to the constant map at the final object by \cref{rem:OutOfPosition}. Consequently, $g$ is the right Kan extension of its restriction to $\sd((P^{\geq b})^{\op})$. Let $\sd^+((P^{> b})^{\op})$ be the subposet on strings ending at $b$ (in $P^{\op}$) and note that $ \sd((P^{> b})^{\op}) \cong \sd^+((P^{> b})^{\op})$ via the ``append $b$'' map. We then have a pullback square
\[ \begin{tikzcd}
\lim g|_{\sd((P^{\geq b})^{\op})} \ar{r} \ar{d}{\gamma'} & \lim g|_{\sd((P^{> b})^{\op})} \ar{d}{\gamma} \\ 
g(b) \ar{r} & \lim g|_{\sd^+((P^{> b})^{\op})}
\end{tikzcd} \]
in which $\gamma$ is induced by the ``append $b$'' homotopy $\into{\sd((P^{> b})^{\op}) \times [1]}{\sd((P^{\geq b})^{\op})}$. For all strings $\sigma = [p_1 > ... > p_n]$ in $(P^{>b})^{\op}$, letting $\sigma^+ \coloneq [p_1 > ... > p_n > b]$ we note that $g(\sigma \subset \sigma^+)$ is an equivalence. Therefore, $\gamma$ and hence $\gamma'$ is an equivalence, and this is clearly natural in the input $f$.
\end{enumerate}
We conclude that we have a morphism of recollements
\[ \begin{tikzcd}[column sep=4em]
(\sX_{/\pi^*(O)})' \ar{d}{\Theta_O} & \sX' \ar{d}{\Theta_P} \ar{r}{i^* = \ev_b} \ar{l}[swap]{j^* = \res} & \sX_b \ar{d}{=} \\
\sX_{/\pi^*(O)} & \sX \ar{l}[swap]{j^*} \ar{r}{i^* = \Phi^b} & \sX_b.
\end{tikzcd} \]
By the inductive hypothesis, $\Theta_O$ is an equivalence. To then deduce that $\Theta_P$ is an equivalence, by \cref{rem:TwoOutOfThreePropertyEquivalencesStrictMorphismRecoll} it remains to observe that we have a \emph{strict} morphism of recollements, i.e., that the adjoint square
\[ \begin{tikzcd}
(\sX_{/\pi^*(O)})' \ar{r}{i^*j_*} \ar{d}{\Theta_O} & \sX_b \ar{d}{=} \\ 
\sX_{/\pi^*(O)} \ar{r}{i^*j_*} & \sX_b
\end{tikzcd} \]
commutes. But using that the lower $i^* j_*: \sX_{/\pi^*(O)} \to \sX_b$ is left-exact, this precisely amounts to our formula for the gluing functor $i^* j_*: (\sX_{/\pi^*(O)})' \to \sX_b$ of the recollement on $\sX'$ that we gave in \cref{thm:ExistenceLaxRightKanExtension}.
\end{proof}

In fact, we can elaborate upon \cref{thm:ReconstructionTopoi} to also reconstruct the $P$-stratification of $\sX$.

\begin{cnstr} \label{cnstr:StratOfFibration}
Let $P$ be a finite poset, $\lambda: \widehat{\sX} \to P^{\op}$ a toposic locally cocartesian fibration, and $\sX = \Fun^{\cocart}_{/P^{\op}}(\sd(P^{\op}), \widehat{\sX})$ its right-lax limit, which is an $\infty$-topos by \cref{lem:ToposicFibrationYieldsTopos}. Given a cosieve $O \subset P$, let $\pi^*(O) \in \sX$ be the uniquely determined $(-1)$-truncated object such that $\Fun^{\cocart}_{/O^{\op}}(\sd(O^{\op}), \widehat{\sX}|_{O^{\op}}) \simeq \sX_{/\pi^*(O)}$. Then we may define a $P$-stratification of $\sX$ by the map of posets
\[ \pi^*: \Open(P) \to \Open(\sX), \]
as it is clear that $\pi^*$ preserves joins and meets (e.g., in view of \cref{rem:AMGRdefinition}).
\end{cnstr}

\begin{cor} \label{cor:ReconstructionOfStrat}
Let $P$ be a finite poset and $(\sX, \pi_*)$ a $P$-stratified $\infty$-topos. The $P$-stratification of $\sX' \coloneq \Fun^{\cocart}_{/P^{\op}}(\sd(P^{\op}), \sG(\sX))$ given by \cref{cnstr:StratOfFibration} coincides with that of $\sX$ under the equivalence $\Theta_P$ of \cref{thm:ReconstructionTopoi}.
\end{cor}
\begin{proof}
For every cosieve $O \subset P$, let $(\sX_{/\pi^*(O)})' \coloneq \Fun^{\cocart}_{/O^{\op}}(\sd(O^{\op}), \sG(\sX)|_{O^{\op}})$ and note that as in the proof of \cref{thm:ReconstructionTopoi} that $ \sG(\sX)|_{O^{\op}} \simeq \sG(\sX_{/\pi^*(O)})$. By \cref{thm:ReconstructionTopoi}, we have that $\Theta_O: (\sX_{/\pi^*(O)})' \to \sX_{/\pi^*(O)}$ is an equivalence. To then see that $(\sX_{/\pi^*(O)})'$ identifies with the open subtopos $\sX_{/\pi^*(O)}$ under the equivalence $\Theta_P$, it remains to observe that the square
\[ \begin{tikzcd}
\sX' \ar{r}{j^*} \ar{d}{\Theta_P} & (\sX_{/\pi^*(O)})' \ar{d}{\Theta_O} \\ 
\sX \ar{r}{j^*} & \sX_{/\pi^*(O)}
\end{tikzcd} \]
commutes. We may proceed by induction on the cardinality of $P \setminus O$.\footnote{Of course, we could also adapt the proof of \cref{thm:ReconstructionTopoi} to show this directly.} If $O = P$ or $O = P \setminus \{ b\}$, then we are done by the proof of \cref{thm:ReconstructionTopoi}. If not, let $b \in P \setminus O$ be a minimal element. We have a factorization
\[ \begin{tikzcd}
\sX' \ar{r}{j^*} \ar{d}{\Theta_P} & (\sX_{/\pi^*(P\setminus \{ b\})})' \ar{r}{j^*} \ar{d}{\Theta_{P\setminus \{ b\}}} & (\sX_{/\pi^*(O)})' \ar{d}{\Theta_O} \\ 
\sX \ar{r}{j^*} & \sX_{/\pi^*(P\setminus \{ b\})} \ar{r}{j^*} & \sX_{/\pi^*(O)}
\end{tikzcd} \]
By the inductive hypothesis, both the inner squares commute, hence the outer square commutes.
\end{proof}

\begin{rem} \label{rem:transportingSheaf}
By \cref{cor:ReconstructionOfStrat}, it follows that given a sheaf $x \in \sX$, under the equivalence of \cref{thm:ReconstructionTopoi} $x$ corresponds to a functor $f_x: \sd(P^{\op}) \to \sG(\sX)$ that sends $[p]$ to $\Phi^p(x)$. The equivalence $x \simeq \Theta_P(f_x)$ then ``reconstructs'' $x$ from its stratumwise values $\Phi^p(x)$ and gluing data thereof.
\end{rem}

We next turn to questions of functoriality in the $P$-stratified $\infty$-topos.

\begin{obs} \label{obs:FunctorialityForToposicRecollements}
Continuing from \cref{exm:RecollementTopoi}, we explain how recollements of topoi are functorial in geometric morphisms. In one direction, suppose given a commutative square
\[ \begin{tikzcd}
\sU \ar{r}{\phi} \ar{d}[swap]{(f_U)_*} & \sZ \ar{d}{(f_Z)_*} \\ 
\sU' \ar{r}{\phi'} & \sZ'
\end{tikzcd} \]
of $\infty$-topoi, where $(f_U)_*, (f_Z)_*$ are geometric morphisms and $\phi, \phi'$ are left-exact accessible functors. Let $\sX$ and $\sX'$ be the $\infty$-topoi $\Ar(\sZ) \times_{\ev_1, \sZ, \phi} \sU$ and $\Ar(\sZ') \times_{\ev_1, \sZ', \phi'} \sU'$. Then the induced functor $f_*: \sX \to \sX'$ admits a left adjoint $f^*$ induced by the mate $(f_Z)^* \phi' \Rightarrow \phi (f_U)^*$; explicitly,
$$f^*[u', z' \rightarrow \phi'(u')] = [(f_U)^*(u'), (f_Z)^*(z') \rightarrow (f_Z)^* \phi'(u') \rightarrow \phi (f_U)^*(u') ].$$
Moreover, since $(f_U)^*, (f_Z)^*, \phi, \phi'$ are left-exact and $(j^*, i^*): \sX \to \sU \times \sZ$ creates finite limits, we see that $f^*$ is left-exact. We conclude that $f_*$ is a geometric morphism. Moreover, $f_*$ is a strict morphism of recollements whose left adjoint $f^*$ is a (not necessarily strict) morphism of recollements. Note also that if we identify $\sU \simeq \sX_{/U}$ and $\sU' \simeq \sX_{/U'}$ for $(-1)$-truncated objects $U, U'$, then $f^*(U') \simeq U$.

Conversely, let $\sX, \sX'$ be $\infty$-topoi decomposed by recollements $(\sU, \sZ), (\sU', \sZ')$ with gluing functors $\phi, \phi'$ and suppose $f_*: \sX \to \sX'$ is a geometric morphism such that both $f^*$ and $f_*$ are morphisms of recollements. Then $f_*$ is necessarily a strict morphism of recollements, and we obtain a commutative square $(f_Z)_* \phi \simeq \phi' (f_U)_*$ as above.

Finally, the theory of recollements implies that these constructions are mutually inverse.
\end{obs}

\begin{dfn}[{\cite[8.2.2]{Exodromy}}]
A \emph{geometric morphism of $P$-stratified $\infty$-topoi} $(\sX, \pi_*) \to (\sY, \rho_*)$ is a geometric morphism $f_*: \sX \to \sY$ subject to the condition that the induced diagram of posets
\[ \begin{tikzcd}
\Open(\sX) \ar{rr}{f_*} \ar{rd}[swap]{\pi_*} & & \Open(\sY) \ar{ld}{\rho_*} \\ 
& \Open(P)
\end{tikzcd} \]
commutes, i.e., for all cosieves $O \subset P$, $f^* \rho^* (O) \cong \pi^*(O)$.

The collection of $P$-stratified $\infty$-topoi and geometric morphisms thereof assembles into an $\infty$-category $\StrTop_{\infty,P}$. Note also that $\StrTop_{\infty,P} \simeq \Top_{\infty} \times_{\Top_0} (\Top_0)_{/\Open(P)}$.
\end{dfn}

\begin{dfn}[{\cite[8.2.7]{Exodromy}}] A \emph{geometric morphism of toposic locally cocartesian fibrations} from $[\lambda: \widehat{\sX} \to P^{\op}]$ to $[\xi: \widehat{\sY} \to P^{\op}]$ is a functor $F: \widehat{\sX} \to \widehat{\sY}$ over $P^{\op}$ such that
\begin{enumerate}
\item $F$ preserves locally cocartesian edges.
\item For all $p \in P$, the fiber $F_p: \widehat{\sX}_p \to \widehat{\sY}_p$ is a geometric morphism of $\infty$-topoi.
\end{enumerate}
The collection of toposic locally cocartesian fibrations and geometric morphisms thereof assembles into an $\infty$-category $\LocCocart^{\mathrm{top}}_{P^{\op}}$.\footnote{Barwick--Glasman--Haine label this $\infty$-category as $\LocCocart^{\lex,\mathrm{top}}_{P^{\op}}$.}
\end{dfn}

\begin{obs} \label{obs:StratGeometricMorphisms}
Let $f_*: (\sX, \pi_*) \to (\sY, \rho_*)$ be a geometric morphism of $P$-stratified $\infty$-topoi. Then for all cosieves $O \subset P$, $f_*$ is a strict morphism of recollements with respect to $(\sX_{/\pi^*(O)}, \sX_{\setminus \pi^*(O)} )$ and $(\sY_{/\rho^*(O)}, \sY_{\setminus \rho^*(O)} )$. Moreover, for all maps of posets $Q \to P$, restriction along $\Shv(Q) \to \Shv(P)$ (in $\Top_{\infty}$) defines a geometric morphism $f'_*: \Shv(Q) \times_{\Shv(P)} \sX \to \Shv(Q) \times_{\Shv(P)} \sY$ of $Q$-stratified $\infty$-topoi. Consequently, for all $p \in P$, $f_*$ sends the stratum $\sX_p$ into $\sY_p$ (with respect to the embeddings $\rho_p$ of \cref{ntn:GeometricFixedPoints}) and we may thus restrict $f_* \times \id: \sX \times P^{\op} \to \sY \times P^{\op}$ to obtain a functor
\[ \sG(f_*): \sG(\sX) \to \sG(\sY) \]
over $P^{\op}$ that preserves locally cocartesian edges. We may thereby promote \cref{con:GluingDiagram} to a functor
$$\sG: \StrTop_{\infty,P} \to \LocCocart^{\lex,\mathrm{top}}_{P^{\op}}.$$

Conversely, suppose $P$ is a finite poset and let $F: \widehat{\sX} \to \widehat{\sY}$ be a geometric morphism of toposic locally cocartesian fibrations. Let $\sX = \mathrm{lim}^{\rlax} \widehat{\sX}$ and $\sY = \mathrm{lim}^{\rlax} \widehat{\sY}$. Let $f_*: \sX \to \sY$ denote the functor induced by $F$. Then by \cref{obs:FunctorialityForToposicRecollements}, \cref{thm:RecollementRlaxLimitOfLlaxFunctor}, and proceeding by induction on the cardinality of $P$, we see that $f_*$ is a geometric morphism such that for every cosieve $O \subset P$, $f_*$ is a strict morphism of recollements from $(\mathrm{lim}^{\rlax} \widehat{\sX}|_{O^{\op}}, \mathrm{lim}^{\rlax} \widehat{\sX}|_{(P \setminus O)^{\op}})$ to $(\mathrm{lim}^{\rlax} \widehat{\sY}|_{O^{\op}}, \mathrm{lim}^{\rlax} \widehat{\sY}|_{(P \setminus O)^{\op}})$. It follows that $f_*$ is a geometric morphism of $P$-stratified $\infty$-topoi with respect to the $P$-stratifications of \cref{cnstr:StratOfFibration}. Therefore, $\mathrm{lim}^{\rlax}$ promotes to a functor
\[ \mathrm{lim}^{\rlax}: \LocCocart^{\mathrm{top}}_{P^{\op}} \to \StrTop_{\infty,P}. \]
\end{obs}

Our remaining goal is to prove that $\sG$ and $\mathrm{lim}^{\rlax}$ define an adjoint equivalence of $\infty$-categories. For the proof, we will need to use the following piece of $(\infty,2)$-category theory from \cite[\S A.8]{AMGRb}:

\begin{obs} \label{leftlaxTorightlax}
Let $\sC, \sD \to P^{\op}$ be locally cocartesian fibrations and recall our discussion of left-lax and right-lax morphisms of locally cocartesian fibrations from \cref{rec:amgr}. Then the space $\Map^{\llax, R}_{/P^{\op}}(\sC, \sD)$ of left-lax morphisms whose fibers are right adjoints is naturally equivalent to the space $\Map^{\rlax, L}_{/P^{\op}}(\sD, \sC)$ of right-lax morphisms whose fibers are left adjoints, with the equivalence implemented fiberwise by passage to adjoints.
\end{obs}

\begin{thm} \label{thm:AdjunctionTopos}
Let $P$ be a finite poset. $\sG$ and $\mathrm{lim}^{\rlax}$ participate in an adjoint equivalence
\[ \adjunct{\mathrm{lim}^{\rlax}}{\LocCocart^{\mathrm{top}}_{P^{\op}}}{\StrTop_{\infty,P}}{\sG}. \]
\end{thm}
\begin{proof}
We proceed as in the proof of \cite[Thm.~6.2.6]{AMGRb}. Suppose $[\lambda: \widehat{\sX} \to P^{\op}]$ is a toposic locally cocartesian fibration and $(\sY, \rho_*)$ is a $P$-stratified $\infty$-topos. In view of the adjunction $\const \dashv \mathrm{lim}^{\rlax}$, we first note that we have a natural equivalence\footnote{Here, $\Cat$ refers to the $\infty$-category of large $\infty$-categories, so that $\Pr^L$ and $\Pr^R$ are subcategories of $\Cat$.}
\[ \psi: \Map_{\Cat}(\sY, \mathrm{lim}^{\rlax} \widehat{\sX}) \xto{\simeq} \Map^{\rlax}_{/P^{\op}}(\sY \times P^{\op}, \widehat{\sX}). \]
Since the evaluation functors $\mathrm{lim}^{\rlax} \widehat{\sX} \to \widehat{\sX}_p$ at each $p \in P$ are all left adjoints, $\psi$ restricts to the equivalence $\psi'$ in the diagram
\[ \begin{tikzcd}
\Map_{\Pr^L}(\sY, \mathrm{lim}^{\rlax} \widehat{\sX}) \ar{r}{\psi'}[swap]{\simeq} \ar{d}{\simeq} & \Map^{\rlax, L}_{/P^{\op}}(\sY \times P^{\op}, \widehat{\sX}) \ar{d}{\simeq} \\ 
\Map_{\Pr^R}(\mathrm{lim}^{\rlax} \widehat{\sX}, \sY) \ar{r}{\psi''}[swap]{\simeq} & \Map^{\llax, R}_{/P^{\op}}(\widehat{\sX}, \sY \times P^{\op}).
\end{tikzcd} \]
We then have the vertical equivalences (with the righthand one explained in \cref{leftlaxTorightlax}), yielding the equivalence $\psi''$ in which a right-adjoint functor $f_*: \mathrm{lim}^{\rlax} \widehat{\sX} \to \sY$ transports to a functor $F: \widehat{\sX} \to \sY \times P^{\op}$ such that for all $p \in P$, the fiber $F_p: \widehat{\sX}_p \to \sY$ is the right adjoint to the composite
\[ \sY \xto{f^*} \mathrm{lim}^{\rlax} \widehat{\sX} \xto{\ev_p} \widehat{\sX}_p. \]
We now observe that $f_*$ is a geometric morphism of $P$-stratified $\infty$-topoi if and only if for all $p \in P$, $F_p$ is a geometric morphism, $F_p$ factors through $\sY_p$, and the resulting map $F: \widehat{\sX} \to \sG(\sY)$ preserves locally cocartesian edges. Indeed, the ``only if' implication follows from the first half of \cref{obs:StratGeometricMorphisms}, while for the ``if'' implication, we note that $f_*$ factors as the composite
$$\mathrm{lim}^{\rlax} \widehat{\sX} \xtolong{\mathrm{lim}^{\rlax} F}{1.2} \mathrm{lim}^{\rlax} \sG(\sY) \xto{\Theta_P} \sY,$$
which respect $P$-stratifications by the second half of \cref{obs:StratGeometricMorphisms} and \cref{cor:ReconstructionOfStrat}, respectively. Therefore, $\psi''$ restricts to the desired natural equivalence
\[ \psi''': \Map_{\StrTop_{\infty,P}}(\mathrm{lim}^{\rlax} \widehat{\sX}, \sY) \simeq \Map_{{\LocCocart^{\mathrm{top}}_{P^{\op}}}}(\widehat{\sX}, \sG(\sY)).\]
% which by construction is natural in $(\sY, \rho_*)$ and $\lambda: \widehat{\sX} \to P^{\op}$.
We conclude that $\mathrm{lim}^{\rlax} \dashv \sG$. Furthermore, unpacking this equivalence of mapping spaces shows that $\Theta_P$ is the counit of the adjunction. Since $\Theta_P$ is an equivalence by \cref{thm:ReconstructionTopoi}, it remains to show that the unit $\eta$ is an equivalence. But the compatibility of the equivalence $\psi'''$ with restriction in the base $P$ shows that $\eta_p$ is homotopic to the identity for all $p \in P$, hence $\eta$ is an equivalence. 
\end{proof}

\begin{rem}
\cref{thm:AdjunctionTopos} should be viewed as the unstable counterpart to \cite[Thm.~A]{AMGRb}, which sets up a similar equivalence between $P$-stratified stable presentable $\infty$-categories (\cite[Def.~2.4.3]{AMGRb}) and locally cocartesian fibrations fibered in such with exact accessible pushforward functors.
\end{rem}

\bibliographystyle{amsalpha}
\bibliography{master}

\end{document}